\documentclass{article}
\usepackage{a4wide}
\usepackage{amsmath,amssymb,amsfonts,amstext,amsbsy,amscd}
\usepackage{amsthm}

\usepackage{comment}
\usepackage{graphicx}
\usepackage{mathrsfs}
\usepackage{color}
\usepackage[authoryear]{natbib}
\RequirePackage[colorlinks,citecolor=blue,urlcolor=blue]{hyperref}

\usepackage[textwidth=4cm]{todonotes}

\allowdisplaybreaks[4]

\numberwithin{equation}{section}

\swapnumbers

\newtheorem{satz}{Satz}[section]

\newtheorem{theorem}[satz]{Theorem}
\newtheorem{proposition}[satz]{Proposition}
\newtheorem{corollary}[satz]{Corollary}
\newtheorem{lemma}[satz]{Lemma}
\newtheorem{assumption}[satz]{Assumption}

 \theoremstyle{definition}
\newtheorem{definition}[satz]{Definition}

\newtheorem{remark}[satz]{Remark}

\newtheorem{example}[satz]{Example}

\DeclareMathOperator{\E}{{\mathbb E}}

\DeclareMathOperator{\R}{{\mathbb R}}
\DeclareMathOperator{\C}{{\mathbb C}}
\DeclareMathOperator{\Z}{{\mathbb Z}}
\DeclareMathOperator{\N}{{\mathbb N}}

\DeclareMathOperator{\T}{{\mathbb T}}
\DeclareMathOperator{\PP}{{\mathbb P}}

\DeclareMathOperator{\trace}{trace}
\DeclareMathOperator{\supp}{supp} \DeclareMathOperator{\im}{im}

 \DeclareMathOperator{\Id}{Id}

\DeclareMathOperator{\Var}{Var} \DeclareMathOperator{\Cov}{Cov}
\DeclareMathOperator{\dom}{dom}

\providecommand{\eps}{\varepsilon}
\renewcommand{\phi}{\varphi}
\renewcommand{\theta}{\vartheta}
\renewcommand{\subset}{\subseteq}

\renewcommand{\cdot}{{\scriptstyle \bullet} }
\providecommand{\abs}[1]{\lvert #1 \rvert}
\providecommand{\norm}[1]{\lVert #1 \rVert}
\providecommand{\bnorm}[1]{{\Bigl\lVert #1 \Bigr\rVert}}
\providecommand{\babs}[1]{{\Bigl\lvert #1 \Bigr\rvert}}
\providecommand{\scapro}[2]{\langle #1,#2 \rangle}
\providecommand{\bscapro}[2]{\Bigl\langle #1,#2 \Bigr\rangle}
\providecommand{\floor}[1]{\lfloor #1 \rfloor}

\renewcommand{\Re}{\operatorname{Re}}
\renewcommand{\Im}{\operatorname{Im}}
\renewcommand{\le}{\leqslant}
\renewcommand{\ge}{\geqslant}

\providecommand{\MR}{$\clubsuit$}
\allowdisplaybreaks[1]
\let\scr\mathscr     

\newcommand{\cF}{\mathcal{F}}

\title{Information bounds for inference in stochastic evolution equations observed under noise}

\author{
Gregor Pasemann\thanks{
Institut für Mathematik, Humboldt-Universität zu Berlin, Germany
({\tt gregor.pasemann@hu-berlin.de}, {\tt mreiss@math.hu-berlin.de})
This research has been partially funded by the Deutsche Forschungsgemeinschaft (DFG)- Project-ID 318763901 - SFB1294.
}
\and
Markus Rei\ss{}\footnotemark[1]
}

\begin{document}

\maketitle

\begin{abstract}
We consider statistics for stochastic evolution equations in Hilbert space with emphasis on stochastic partial differential equations (SPDEs). We observe a solution process under additional measurement errors and want to estimate a real or functional parameter in the drift. Main targets of estimation are the diffusivity, transport or source coefficient in a parabolic SPDE. By bounding the Hellinger distance between observation laws under different parameters we derive lower bounds on the estimation error, which reveal the underlying information structure. The estimation rates depend on the measurement noise level, the observation time, the covariance of the dynamic noise, the dimension and the order, at which the parametrised coefficient appears in the differential operator. A general estimation procedure attains these rates in many parametric cases and proves their minimax optimality. For nonparametric estimation problems, where the parameter is an unknown function, the lower bounds exhibit an even more complex information structure. The proofs are to a large extent based on functional calculus, perturbation theory and monotonicity of the semigroup generators.
\end{abstract}



\section{Introduction}

Many dynamical systems in nature and society are subject to randomness and stochastic partial differential equations (SPDEs) yield prototypical examples of such stochastic dynamical models. Let us consider linear parabolic SPDEs of the form
\begin{equation}\label{EqSPDEPhysics}
\dot X(t,y)=D_\theta X(t,y)+\dot W(t,y),\quad t\ge 0,\,y\in\Lambda,
\end{equation}
with space-time Gaussian white noise $\dot W$ or equivalently in It\^o form
\begin{equation}\label{EqSPDE}
dX(t,y)=D_\theta X(t,y)\,dt+dW(t,y),\quad t\ge 0,\,y\in\Lambda,
\end{equation}
with a cylindrical Brownian motion $W(t)$ on $L^2(\Lambda)$, subject to some initial condition $X(0,y)$, $y\in\Lambda$.
 Here, $\Lambda\subset \R^d$ is a spatial domain and $D_\theta$ denotes a general second-order differential operator
\[ D_\theta f(y):=\nabla\cdot\big(c^{(2)}_\theta(y)\nabla f+c^{(1)}_\theta(y) f\big)(y)+c^{(0)}_\theta(y)f(y),\quad y\in\Lambda,
\]
involving some boundary condition and (matrix-/vector-/scalar-valued) coefficients $c^{(2)}_\theta,c^{(1)}_\theta,c^{(0)}_\theta$, parametrised by a Euclidean or functional parameter $\theta\in\Theta$. First fundamental results for  estimation in a general SPDE framework have been obtained by \citet{HuRo1995} and \citet{IbrKha2001}. Afterwards, estimators for $\theta$ in several specific settings and from observing $X$ globally, locally or at discrete points in time and space  have been developed, see \citet{cialenco2018} for an excellent survey of the state of the art up to this time. Recently, in the general setting \eqref{EqSPDE} a parametric estimation theory has been developed by \citet{ATW2024} under multiple local observations of $X$.

In many cases, however, measurements of $X$ involve an additional error, see e.g \citet{PFABS2020} for measurements of cell motility. Under measurement noise different observation schemes often lead to the same estimation theory because the measurement noise typically dominates the discretisation errors.
Therefore we consider the noisy observations $dY$ of $X$ given by
\begin{equation}\label{EqYWN}
dY(t,y)=B X(t,y)\,dt+\eps dV(t,y),\quad   t\in[0,T],\,y\in\Lambda,
\end{equation}
where $V$ is a cylindrical Brownian motion on $L^2(\Lambda)$, independent of $W$, $B$ is an observation operator and $\eps\ge 0$ is the noise level. This formulation is  standard for stochastic filtering problems in dynamical systems, see e.g. \citet{BaCr2009}. From an alternative point of view, we may consider a regression setting with discrete data
\[ Y_{i,j}=(B X)(t_i,y_j)+\eps_{i,j}\]
at some time points $t_i$ and spatial locations $y_j$ with  i.i.d. error variables $\eps_{i,j}\sim {\cal N}(0,\sigma^2)$, independent of the driving noise $W$. If the design points $(t_i,x_j)$ are equally spaced and become dense in $[0,T]\times\Lambda$ for sample size $n\to\infty$, then one can conclude in a rigorous manner that these observation schemes become asymptotically equivalent with observing \eqref{EqYWN} for $\eps^2=\frac{\sigma^2T\abs{\Lambda}}n$, compare \citet{reiss2008}, Section 3 in \citet{BHMR2014} and the references therein. This interpretation, however, requires that point evaluations of $BX$ are well defined.

 Given the observations \eqref{EqYWN}, we aim at establishing minimax optimal estimation rates for $\theta$. It turns out that these rates differ significantly whenever this parameter only appears in the {\it diffusivity} (or {\it conductivity}) coefficient $c^{(2)}_\theta$, the {\it transport} (or {\it advection}) coefficient $c^{(1)}_\theta$ or in the {\it source} (or {\it reaction}) coefficient $c^{(0)}_\theta$. Moreover, when these coefficient functions $c^{(k)}_\theta(y)=\theta(y)$ are not parametrised by a real parameter, but assumed to belong just to a functional class of H\"older regularity $\alpha>0$, then the rates become nonparametric and exhibit structural differences, depending on the asymptotics taken.

Establishing minimax optimal estimation rates requires to construct estimators attaining these rates as well as to derive general information-theoretic lower bounds proving that no estimator can converge faster, see \cite{Tsyb} for a comprehensive introduction. Lower bounds provide a genuine insight into the intrinsic difficulty in drawing inference on the parameter and reveal the underlying information structure. We thus first concentrate on a general lower bound result in the framework of stochastic evolution equations (see e.g. \citet{DPZa2014,RoLo2012}), generalizing the SPDE \eqref{EqSPDE},
\begin{equation}\label{EqX}
 dX_t=A_\theta X_tdt+dW_t,\quad t\in[0,T],
 \end{equation}
on a real Hilbert space $H$ where $A_\theta$ is the generator of a strongly continuous semigroup on $H$. Then the observations are given by
\begin{equation}\label{EqY}
dY_t=BX_tdt+\eps dV_t,\quad t\in[0,T],
\end{equation}
with an independent $H$-valued cylindrical Brownian motion $V$. By linearity, the observations $dY$ follow a cylindrical Gaussian law on $L^2([0,T];H)$ and our main information-theoretic result in Theorem \ref{ThmLB} gives a tight bound for the Hellinger distance between the observation laws for two different parameters $\theta_0,\theta_1\in\Theta$. The bound is given in terms of the Hilbert-Schmidt norm of functions of $A_{\theta_0},A_{\theta_1}$. The fundamental functional-analytic difficulty consist in reducing the intrinsic Hilbert-Schmidt norm for the covariance operators of the time-continuous processes on $L^2([0,T];H)$ to a norm of the generators on $H$ without assuming commutativity of the generators. A particular consequence is the absolute continuity of the laws, which for direct observations has been studied by \cite{peszat1992}.

In Section \ref{SecParRates} we construct estimators in a general parametric setting with real parameter $\theta$ and commuting operators which under quite general conditions attain the lower bound rates, thus establishing minimax optimality. The concrete construction and analysis of the estimators is rather involved, but the main idea is to reduce the observational noise first by averaging with an operator-valued kernel and then to apply a continuous-regression estimator as in the case of direct observations of $X$.

The implications of these abstract results for fundamental parametric estimation problems are demonstrated in Section \ref{SEcParEx}. Already for noisy observations of the standard Ornstein-Uhlenbeck process
\[   dY_t=X_tdt+\eps\,dV_t\text{ with } dX_t=-\theta X_tdt+\sigma\,dW_t,\quad t\in[0,T],\,X_0=0,\]
with $\theta>0$ our optimal estimation rate $v_n=(\theta^{1/2}T^{-1/2}+ T^{-1}) (\eps^2\sigma^{-2}\theta^2+ 1)$ from Proposition \ref{PropOU}  generalizes results from \citet{kutoyants2004} and reveals interesting phenomena.  For fixed $\theta>0$ we obtain the classical $T^{-1/2}$-rate under ergodicity when $\eps/\sigma$ remains bounded and for $\theta\downarrow 0$ we approach the null-recurrent $T^{-1}$-rate. In both cases the rate does not suffer from measurement noise and equals that for direct observation of $X$ ($\eps=0$).

\begin{table}\centering

\begin{tabular}{|l|l|}
  \hline
  SPDE with parameter $\theta>0$ &  rate $v_n^{par}$\\\hline
   $dX(t)=\theta\Delta X(t)dt+dW(t)$ & $T^{-1/2}\eps^{(d+2)/4}$ \\
   $dX(t)=(\nu\Delta X(t)+\theta\partial_\xi X(t))dt+dW(t)$ & $T^{-1/2}\eps^{d/4}\nu^{(d+2)/4}$ \\
   $dX(t)=(\nu\Delta X(t)-\theta X(t))dt+dW(t)$ &  $T^{-1/2}\eps^{(d-2)_+/4}\nu^{d/4}$ \\
  \hline
\end{tabular}
\caption{Rates for estimating $\theta>0$ in different coefficients as a function of observation time $T$, static noise level $\eps$, dimension $d$ and diffusivity $\nu$ (dropping a log factor in the last row for $d=2$).}\label{TabParRates}
\end{table}

In Table \ref{TabParRates} we gather the minimax rates obtained for SPDEs of the form \eqref{EqSPDE} with a Laplacian $\Delta$ on a $d$-dimensional bounded domain $\Lambda$ and a first-order derivative $\partial_\xi$ in direction $\xi$ under standard boundary conditions, see Propositions \ref{PropParDiff}, \ref{PropParametricTransport} and \ref{PropParametricSource} for the precise formulations.
Ergodicity in time yields the $T^{-1/2}$-rate, keeping the range of $\theta$ fixed.
The higher the order of the coefficient is, the easier it is estimated for small static noise level $\eps$. There is yet another asymptotics that permits consistent estimation of transport or source coefficient, namely that the second-order coefficient degenerates as $\nu\downarrow 0$, which was first statistically exploited for direct observations of $X$ by \citet{GaRe2022}. The larger the dimension $d$, the faster the rates become, which is explained by the Weyl asymptotics of the underlying eigenvalues. All rates are valid up to some maximal dimension $d$, afterwards $\theta$ is directly identifiable. In Section \ref{SEcParEx} we also discuss the rate for the fractional Laplacian, the general question of absolute continuity of laws of $X$ under different parameters and the dependence of the rate on different spatial correlation induced by the operator $B$.

The nonparametric lower bounds are much more involved because the operators $A_\theta$ do not commute, even in standard examples like a space-dependent diffusivity $\theta(y)$. The main tool to bound the Hilbert-Schmidt norm in the Hellinger bound for differential operators $A_\theta=D_\theta$ are estimates of the form $g(D_{\theta_1})\preccurlyeq C g(D_{\theta_0})$ for functions $g$ and some constant $C>0$ with respect to the partial order $\preccurlyeq$ induced by positive semi-definiteness. Since the functions $g$ are mostly not operator monotone in the sense of \citet{bhatia2013}, a careful specific analysis is required, drawing on PDE ideas in \citet{EN2000} and similar to the non-commutative SPDE analysis by \citet{LoRo2000}.

\begin{table}\centering
\begin{tabular}{|l|l|l|}
  \hline
  SPDE with $\alpha$-regular function $\theta(\cdot)$ &  rate $(v_n^{par})^{\frac{2\alpha}{2\alpha+d}}$ & rate for $\eps\thicksim 1$\\\hline
   $dX(t)=\nabla\cdot(\theta\nabla  X)(t)dt+dW(t)$ & for $T\le\eps^{1-\alpha}$ & $T^{-\frac{\alpha}{2\alpha+3}}$ \\
   $d X(t)=(\Delta X(t)+\partial_\xi (\theta X)(t))dt+d W(t)$ &  for $T\le\eps^{-\alpha}$  & $T^{-\frac{\alpha}{2\alpha+5}}$\\
   $dX(t)=(\Delta X(t)-(\theta X)(t))dt+d W(t)$ &  for $T\le\eps^{-\alpha-(d\wedge 2)/2}$  & $T^{-\frac{\alpha}{2\alpha+4+d\wedge 3}}$\\
  \hline
\end{tabular}
\caption{Rates for estimating $\theta(\cdot)$ nonparametrically in different coefficients. The second column shows when the classical scaling of the parametric rate $v_n^{par}$ applies (a log factor is dropped in the last row for $d=2$). The third column gives the rate for non-vanishing noise level $\eps$.}\label{TabNonparRates}
\end{table}

The rates obtained for space-dependent diffusivity, transport and source terms exhibit an ellbow effect: they relate to the corresponding parametric rates $v_n^{par}$ as in classical nonparametric regression whenever $T$ is not growing too fast (or is constant) as $\eps\downarrow 0$, while they are completely different for $T\to\infty$ with $T\ge \eps^{-p}$ for certain powers $p$. Table \ref{TabNonparRates} reports where (at which $p$) the ellbow occurs and states the result for $T\to\infty$ and noise levels $\eps$ of order one. Detailed results are given in Theorems \ref{ThmNonpar1}, \ref{ThmNonpar2} and \ref{ThmNonpar3}. In particular, for $\eps\thicksim 1$ the rates necessarily slow down in smaller dimensions compared to the classical $T^{-\alpha/(2\alpha+d)}$-rate  ($d\le 2$ for diffusivity, $d\le 4$ for transport and $d\le 6$ for source estimation). In the companion paper \citet{GPMR2024}, a nonparametric estimator of diffusivity was constructed that attains the lower bound rate for fixed $T$ and regularity $\alpha\in[1,\alpha_{max}]$, which turned out to be highly non-trivial. The analysis there shows that for $\alpha<1$ the approach to reduce locally the static noise first cannot yield optimal rates, which gives an upper bound perspective of the ellbow effect. We conjecture that our nonparametric lower bounds give the minimax rates also in all other cases, but the construction and analysis of corresponding estimators remains a challenging open problem.

In Section \ref{SecHellinger} we bound the Hellinger distance between two cylindrical Gaussian measures, which might be of independent interest. This is then used in Section \ref{SecSEE}  to bound the laws for noisy observations of stochastic evolution equations. In a generic setting we construct parametric estimators in Section \ref{SecParRates} whose risk attain the  lower bounds in wide generality. This is exemplified in Section \ref{SEcParEx} for the Ornstein-Uhlenbeck process and standard SPDEs. Section \ref{SecNonparEx} derives the nonparametric lower bounds for SPDEs. We introduce specific notation before its first usage and gather all notation in Appendix \ref{SecNotation} for the convenience of the reader. Appendix \ref{SecProofs} provides more standard proofs together with the bounds for the error of the estimator constructed in Section \ref{SecParRates}. Auxiliary results are collected in Appendix \ref{SecAux}.

\section{Hellinger bounds for cylindrical Gaussian measures}\label{SecHellinger}

To establish minimax lower bounds, we use the classical approach based on bounding the Hellinger distance between two parameters. For observations from the statistical model $({\cal X},{\scr F},(\PP_\theta)_{\theta\in\Theta})$ where the (non-empty) parameter set $\Theta$ is equipped with a semi-metric $d$, we combine the reduction scheme in \citet[Section 2.2]{Tsyb} with \citet[Thm. 2.2(ii)]{Tsyb} to obtain

\begin{theorem}\label{ThmLB}
Let $\delta>0$. Assume there are $\theta_0,\theta_1\in\Theta$ with semi-distance $d(\theta_0,\theta_1)\ge \delta$ such that their respective laws have Hellinger distance $H(\PP_{\theta_0},\PP_{\theta_1})\le 1$. Then
\begin{equation}\label{EqLBgen} \inf_{\hat\theta}\sup_{\theta\in\Theta}\PP_\theta\big(\delta^{-1}d(\hat\theta,\theta)\ge 1/2\big)\ge \tfrac{2-\sqrt{3}}{4}
\end{equation}
holds where the infimum is taken over all estimators (measurable $\Theta$-valued functions) in the statistical model.
\end{theorem}

In a nutshell, the idea for establishing tight minimax lower bounds is to find two parameters (real values or functions) under which the observation laws satisfy a non-trivial Hellinger bound and which at the same time have a large (Euclidean or functional) distance. We shall apply the lower bound for sequences of models, which become more informative in $n\in\N$,  and try to find the largest $\delta=\delta_n$ in \eqref{EqLBgen}. Then no estimator sequence $\hat\theta_n$ can satisfy $\delta_n^{-1}d(\hat\theta_n,\theta)\xrightarrow{\PP_\theta} 0$ uniformly over $\theta$. In other words, estimators $\hat\theta_n$ are minimax rate-optimal when they satisfy $d(\hat\theta_n,\theta)={\cal O}_{\PP_\theta}(v_n)$ for a rate $v_n\thicksim \delta_n$, uniformly over $\theta$ (Appendix \ref{SecNotation} recalls uniform stochastic convergence).

Next, we study the Hellinger distance between cylindrical Gaussian measures with different covariance operators.

\begin{lemma}\label{LemHellreal}
The squared Hellinger distance between Gaussian laws ${\cal N}(0,\sigma_0^2)$ and ${\cal N}(0,\sigma_1^2)$ satisfies  for $\sigma_0,\sigma_1>0$
\[ H^2({\cal N}(0,\sigma_0^2),{\cal N}(0,\sigma_1^2))\le  \frac14\Big(\frac{\sigma_1}{\sigma_0}-\frac{\sigma_0}{\sigma_1}\Big)^2.\]
\end{lemma}

\begin{proof}
By invariance of the Hellinger distance under bi-measurable bijections, see \citet[Appendix A.1]{reiss2011}, we have $H^2({\cal N}(0,\sigma_0^2),{\cal N}(0,\sigma_1^2))= H^2({\cal N}(0,1),{\cal N}(0,(\sigma_1/\sigma_0)^2))$, so it suffices to show
\begin{align}\label{eq:lemHellReal}
	H^2({\cal N}(0,1),{\cal N}(0,\sigma^2))\le  \tfrac14(\sigma-\sigma^{-1})^2
\end{align}
for $\sigma>0$.
If $(\sigma-\sigma^{-1})^2 > 8$, this bound holds trivially as the squared Hellinger distance is always bounded by two.
Now assume $(\sigma-\sigma^{-1})^2\le 8$.
Following \citet{reiss2011} the squared Hellinger distance for scalar Gaussian laws is given by
\[ H^2({\cal N}(0,1),{\cal N}(0,\sigma^2))=2-2\sqrt{2\sigma/(\sigma^2+1)}.\]
We use $(\sigma^2+1)(1+\sigma^{-1})^{2}\ge 8$ (the mininum is attained at $\sigma=1$) and bound, differently from that reference,
\begin{align*}
\sqrt{2\sigma/(\sigma^2+1)}&=\sqrt{1-(\sigma-1)^2/(\sigma^2+1)}\ge \sqrt{1-\tfrac18(\sigma-1)^2(1+\sigma^{-1})^2}\\
&=\sqrt{1-\tfrac18(\sigma-\sigma^{-1})^2}\ge 1-\tfrac18(\sigma-\sigma^{-1})^2,
\end{align*}
given that $0\le (\sigma-\sigma^{-1})^2\le 8$. This implies \eqref{eq:lemHellReal}.
\end{proof}

This real-valued bound gives rise to a corresponding Hilbert space bound. For details on cylindrical Gaussian measures ${\cal N}_{cyl}(\mu,Q)$ we refer to \citet{bogachev1998}. Here, the main use is that $X\sim {\cal N}_{cyl}(0,Q)$ means that $(\scapro{X}{v})_{v\in H}$ is a centred Gaussian process indexed by $v$ with $\Cov(\scapro{X}{v_1},\scapro{X}{v_2})=\scapro{Qv_1}{v_2}$ where the covariance operator $Q:H\to H$ is bounded (not necessarily trace class) and positive semi-definite, notation $Q\succcurlyeq0$.

\begin{proposition}\label{PropHellinger}
Consider cylindrical Gaussian laws ${\cal N}_{cyl}(0,Q_{\bf_0}),{\cal N}_{cyl}(0,Q_{\bf 1})$ in some separable real Hilbert space $H$ with  covariance operators $Q_{\bf 0},Q_{\bf 1}$. If $Q_{\bf 0},Q_{\bf 1}$ are one-to-one and $\im(Q_{\bf 0}^{1/2})=\im(Q_{\bf 1}^{1/2})$,
then
\[ H^2({\cal N}_{cyl}(0,Q_{\bf 0}),{\cal N}_{cyl}(0,Q_{\bf 1}))\le  \tfrac14\norm{Q_{\bf 0}^{-1/2}Q_{\bf 1}^{1/2}-(Q_{\bf 1}^{-1/2}Q_{\bf 0}^{1/2})^\ast}_{HS}^2,\]
provided the right-hand side is finite. In that case the laws are equivalent.
\end{proposition}

\begin{remark}\label{RemHell} \
\begin{enumerate}
\item If ${\cal N}(0,Q_{\bf 0}),{\cal N}(0,Q_{\bf 1})$ are equivalent proper laws on $H$ (i.e., $Q_{\bf 0},Q_{\bf 1}$ are trace class),
the Feldman-H\'ajek Theorem \cite[Thm. 2.25]{DPZa2014} shows that $\im(Q_{\bf 0}^{1/2})=\im(Q_{\bf 1}^{1/2})$  and that $Q_{\bf 0}^{-1/2}Q_{\bf 1}^{1/2}(Q_{\bf 0}^{-1/2}Q_{\bf 1}^{1/2})^\ast-\Id$ is a  Hilbert-Schmidt  operator on $H$ (note that here $\im(Q_{\bf 0}^{1/2})$ is dense in $H$ by the injectivity of $Q_{\bf 0}$). Since $(Q_{\bf 0}^{-1/2}Q_{\bf 1}^{1/2})^\ast$ has the bounded inverse $(Q_{\bf 1}^{-1/2}Q_{\bf 0}^{1/2})^\ast$, we deduce that also the product $Q_{\bf 0}^{-1/2}Q_{\bf 1}^{1/2}-(Q_{\bf 1}^{-1/2}Q_{\bf 0}^{1/2})^\ast$ is a  Hilbert-Schmidt  operator and the Hellinger bound is finite.
\item
We may write
\[ Q_{\bf 0}^{-1/2}Q_{\bf 1}^{1/2}-(Q_{\bf 1}^{-1/2}Q_{\bf 0}^{1/2})^\ast=Q_{\bf 0}^{-1/2}(Q_{\bf 1}-Q_{\bf 0})Q_{\bf 1}^{-1/2}\]
if we consider the Gelfand triple $H^1\hookrightarrow H\hookrightarrow H^{-1}$ with $H^1=\im(Q_i^{1/2})$, $H^{-1}=(H^1)^\ast$ and interpret $Q_{\bf 1}-Q_{\bf 0}: H^{-1}\to H^1$. In contrast to the bound (cf. \citet{reiss2011})
\[H^2({\cal N}_{cyl}(0,Q_{\bf 0}),{\cal N}_{cyl}(0,Q_{\bf 1}))\le  2\norm{Q_{\bf 0}^{-1/2}(Q_{\bf 1}-Q_{\bf 0})Q_{\bf 0}^{-1/2}}_{HS}^2,\]
which is asymmetric in $Q_{\bf 0}$ and $Q_{\bf 1}$, the bound of Proposition \ref{PropHellinger} will allow us to obtain feasible expressions also for non-commuting covariance operators.
\end{enumerate}
\end{remark}

\begin{proof}
Following  the proof of the Feldman-H\'ajek Theorem (step 2 for Thm. 2.25 in \cite{DPZa2014}) we let
 $(e_k)_{k\ge 1}$ be the orthonormal basis of eigenvectors of
 \[R:=Q_{\bf 0}^{-1/2}Q_{\bf 1}^{1/2}(Q_{\bf 0}^{-1/2}Q_{\bf 1}^{1/2})^\ast\]
 with corresponding positive eigenvalues $(\tau_k)_{k\ge 1}$, noting that $R$ is an injective positive operator and that
 \[R-\Id=\big(Q_{\bf 0}^{-1/2}Q_{\bf 1}^{1/2}-(Q_{\bf 1}^{-1/2}Q_{\bf 0}^{1/2})^\ast\big)(Q_{\bf 0}^{-1/2}Q_{\bf 1}^{1/2})^\ast\]
 is by assumption a Hilbert-Schmidt operator, which has an orthonormal basis of eigenvectors.

We can generate the laws via an i.i.d. sequence $\zeta_k\sim N(0,1)$:
 \[ \sum_{k\ge 1} \zeta_kQ_{\bf 0}^{1/2}e_k\sim {\cal N}_{cyl}(0,Q_{\bf 0}),\quad \sum_{k\ge 1} \zeta_k\tau_k^{1/2}Q_{\bf 0}^{1/2}e_k\sim {\cal N}_{cyl}(0,Q_{\bf 1}).
 \]
 In fact, we must check for all $ g\in H$
 \[\sum_{k\ge 1} \zeta_k\scapro{Q_{\bf 0}^{1/2}e_k}{g}\sim {\cal N}(0,\scapro{Q_{\bf 0}g}{g}), \quad \sum_{k\ge 1} \zeta_k\scapro{\tau_k^{1/2}Q_{\bf 0}^{1/2}e_k}{g}\sim {\cal N}(0,\scapro{Q_{\bf 1}g}{g}),
 \] compare \citet[Thm. 2.2.4]{bogachev1998}. The first statement follows from $\sum_{k\ge 1}\scapro{Q_{\bf 0}^{1/2}e_k}{g}^2=\norm{Q_{\bf 0}^{1/2}g}^2$, the second statement from
 \[ \sum_{k\ge 1}\scapro{Q_{\bf 0}^{1/2}e_k}{g}\scapro{\tau_kQ_{\bf 0}^{1/2}e_k}{g}=\scapro{Q_{\bf 0}^{1/2}g}{RQ_{\bf 0}^{1/2}g}=\scapro{Q_{\bf 1}g}{g}.\]

Then by  the subadditivity of the squared Hellinger distance under independence  and Lemma \ref{LemHellreal}
\begin{align*}
&H^2({\cal N}_{cyl}(0,Q_{\bf 0}),{\cal N}_{cyl}(0,Q_{\bf 1}))\\
&= H^2\Big(\bigotimes_{k\ge 1}{\cal N}(0,\norm{Q_{\bf 0}^{1/2}e_k}^2),\bigotimes_{k\ge 1}{\cal N}(0,\norm{\tau_k^{1/2}Q_{\bf 0}^{1/2}e_k}^2)\Big)\\
 &\le \sum_{k\ge 1} H^2\Big({\cal N}(0,\norm{Q_{\bf 0}^{1/2}e_k}^2),{\cal N}(0,\tau_k\norm{Q_{\bf 0}^{1/2}e_k}^2)\Big)\\
&\le \frac14\sum_{k\ge 1} (\tau_k^{1/2}-\tau_k^{-1/2})^2\\
&= \frac14 \trace\big(R-2\Id+R^{-1}\big)\\
&= \frac14 \trace\Big(\big((Q_{\bf 0}^{-1/2}Q_{\bf 1}^{1/2})^\ast-Q_{\bf 1}^{-1/2}Q_{\bf 0}^{1/2}\big)\big(Q_{\bf 0}^{-1/2}Q_{\bf 1}^{1/2}-(Q_{\bf 1}^{-1/2}Q_{\bf 0}^{1/2})^\ast\big)\Big)\\
&=\tfrac14\norm{Q_{\bf 0}^{-1/2}Q_{\bf 1}^{1/2}-(Q_{\bf 1}^{-1/2}Q_{\bf 0}^{1/2})^\ast}_{HS}^2,
\end{align*}
where we used $R^{-1}=(Q_{\bf 1}^{-1/2}Q_{\bf 0}^{1/2})^\ast Q_{\bf 1}^{-1/2}Q_{\bf 0}^{1/2}$ and commutativity under the trace.

The equivalence follows verbatim as for the Feldman-H\'ajek Theorem.
\end{proof}

\section{Stochastic evolution equations under noise}\label{SecSEE}

For a comprehensive treatment of stochastic evolution equations we refer to \cite{DPZa2014} and \cite{RoLo2012}. We consider the observation process $dY_t$ \eqref{EqY} of the solution $X$ to the stochastic evolution equation \eqref{EqX}. This means in particular that conditionally on $X$ we observe $dY\sim {\cal N}_{cyl}(BX,\eps^2\Id)$ on $L^2(H)$, that is $\int_0^T f(t)dY_t\sim {\cal N}(\int_0^T\scapro{BX_t}{f(t)}\,dt,\eps^2\norm{f}_{L^2(H)}^2)$ for all test functions $f\in L^2(H)$. This generalises the setting of \citet[Section 3.1]{kutoyants2004} (there, $X$ and $Y$ are interchanged) and of \cite{GPMR2024} (there, formally $dY_t/dt$ describes the observation process).

For the observational noise level we assume $\eps\in[0,1]$ so that also non-noisy observations of $BX$ are included for $\eps=0$. The observation time is $T \ge 1$ and usually remains fixed or tends to infinity. $B:H\to H$ is a known bounded linear operator. The unknown parameter is denoted by $\theta\in\Theta$ and for the parameter set  we assume throughout that at least two parameters $\theta_0,\theta_1$ lie in $\Theta$. We  abbreviate $A_{\bf 0}:=A_{\theta_0}$, $A_{\bf 1}:=A_{\theta_1}$ and apply this rule to all further indexing. For simplicity the initial condition $X_0$ is taken to be zero, noting that all minimax lower bounds derived later trivially extend when the maximum is taken over arbitrary $X_0\in H$.

The possibly unbounded linear operators $A_\theta:\dom(A_\theta)\subset H\to H$ are  normal   with the same domain $\dom(A_\theta)$ for all $\theta$ and generate a strongly continuous semigroup $(e^{A_\theta u})_{u\ge 0}$ by functional calculus. Note that then $(e^{A_\theta^\ast u})_{u\ge 0}$ is its adjoint semigroup.
 As a generator of a semigroup, $A_\theta$ is necessarily quasi-dissipative, meaning that $A_\theta$ has a spectrum whose real part is bounded from above, see e.g. \citet[Lemma 7.2.6]{BS2018}. Then
\begin{equation}\label{EqVoC} X_t:=\int_0^t e^{A_\theta(t-s)}dW_s,\quad t\in[0,T],
\end{equation}
defines for each $t$ a cylindrical Gaussian measure on $H$ via $\scapro{X_t}{z}=\int_0^t\scapro{e^{A_\theta^\ast(t-s)}z}{dW_s}$, $z\in H$, which  solves for $z\in\dom(A_\theta^\ast)$ the weak formulation \cite[Thm. 5.4]{DPZa2014}
\[ d\scapro{z}{X_t}=\scapro{A_\theta^\ast z}{X_t}dt+\scapro{z}{dW_t},\quad t\in[0,T],\text{ with } \scapro{z}{X_0}=0.
\]
Let us stress that we do not assume commutativity  of the operators $(A_\theta)_{\theta\in\Theta}$.

\begin{example}\label{ex:ObsToDynamicNoise}\

Writing $\tilde X_t=BX_t$ and assuming $B$ to commute with all $A_\theta$,  we observe
\[dY_t=\tilde X_tdt+\eps dV_t\text{ with } d\tilde X_t= A_\theta\tilde X_tdt+BdW_t,\; \tilde X_0=0.
\]
This way we can also treat cases of evolution equations with more general noise as a driver.
An instructive example  is given by $B=\sigma\Id$ for some known $\sigma>0$. Then we observe
\[dY_t=\tilde X_tdt+ \eps dV_t\text{ with } d\tilde X_t=A_\theta \tilde X_tdt+\sigma dW_t,\; \tilde X_0=0.\]
On the other hand, multiplying the observation process \eqref{EqY} by $\sigma^{-1}$ and replacing $Y$ by $\sigma^{-1}Y$, we observe equivalently
\[ dY_t=X_tdt+\eps\sigma^{-1} dV_t\text{ with } dX_t=A_\theta X_tdt+dW_t,\; X_0=0,\]
Consequently, introducing a dynamic noise level $\sigma>0$ is statistically equivalent to replacing the observation noise level $\eps>0$ by the ratio $\eps/\sigma$. This means that higher dynamic noise levels will lead to smaller statistical errors, compare the discussion in \cite{GPMR2024}.
\end{example}

\begin{lemma}\label{LemCovOp}
The observations $(\int_0^T \scapro{g(t)}{dY_t},\,g\in L^2(H))$ with $dY$ from \eqref{EqY} generate a cylindrical Gaussian measure ${\cal N}_{cyl}(0,Q_\theta)$  on $L^2(H)$ with
\[Q_\theta=\eps^2 \Id+BC_\theta B^\ast\text{, where }C_\theta=S_\theta S_\theta^\ast\]
 and for $f\in L^2(H)$
\[S_\theta f(t)=\int_0^t e^{A_\theta(t-s)}f(s)\,ds,\quad S_\theta^\ast f(t)=\int_t^T e^{A_\theta^\ast(s-t)}f(s)\,ds,\quad t\in[0,T].\]
\end{lemma}

\begin{remark}
Since $BC_\theta B^\ast$ is positive semi-definite, $Q_\theta$ is strictly positive for $\eps>0$ and an isomorphism on $L^2(H)$. For proper Gaussian measures and $\eps=0$ the result from Lemma \ref{LemCovOp} and consequences for absolutely continuous laws can be found in \cite{peszat1992}.
\end{remark}

\begin{proof}
The variation-of-constants formula \eqref{EqVoC} and the stochastic Fubini Theorem \cite[Thm. 4.18]{DPZa2014} yield
\begin{align*}
\int_0^T\scapro{g(t)}{X(t)}\,dt &=\int_0^T\int_0^t \scapro{g(t)}{e^{A_\theta(t-s)}dW_s}\,dt\\
&=\int_0^T \bscapro{\int_s^T e^{A_\theta^\ast(t-s)}g(t)\,dt}{dW_s}=\int_0^T \scapro{S_\theta^\ast g(s)}{dW_s}.
\end{align*}
We deduce by It\^o isometry
\[ \E\Big[\Big(\int_0^T\scapro{g(t)}{X(t)}\,dt\Big)^2\Big]=\norm{S_\theta^\ast g}_{L^2(H)}^2=\scapro{S_\theta S_\theta^\ast g}{g}_{L^2(H)}.
\]
By independence of $dV$ and $dW$ we have
\begin{align*}
 \E\Big[\Big(\int_0^T \scapro{g(t)}{dY_t}\Big)^2\Big]&=\E\Big[\Big(\int_0^T\scapro{g(t)}{B X(t)}\,dt\Big)^2\Big]+\eps^2\norm{g}_{L^2(H)}^2\\
 &=\scapro{S_\theta S_\theta^\ast B^\ast g}{B^\ast g}_{L^2(H)}+\eps^2\scapro{g}{g}_{L^2(H)}\\
&=\scapro{Q_\theta  g}{g}_{L^2(H)}.
\end{align*}
By linearity and polarisation, $(\int_0^T \scapro{g(t)}{dY_t},\,g\in L^2(H))\sim {\cal N}_{cyl}(0,Q_\theta)$ follows.
\end{proof}

In order to be able to apply functional calculus for normal operators, we complexify $H$ and all operators in the usual way (see Section 5.1 in \citet{BS2018} for more formal details) by introducing the complex Hilbert space $H^{\C}:=H+iH$ with $\norm{v+iw}_{H^{\C}}^2:=\norm{v}_H^2+\norm{w}_H^2$ and the extension $A_\theta^{\C}:\dom(A_\theta)+i\dom(A_\theta)\subset H^{\C}\to H^{\C}$ via $A_\theta^{\C}(v+iw):=A_\theta v+iA_\theta w$ for $v,w\in \dom(A_\theta)$. Then by normality of $A_\theta$ ($(A_\theta^{\C})^\ast$ denotes the complex adjoint)
\[ (A_\theta^{\C})^\ast A_\theta^{\C}(v+iw)=A_\theta^\ast A_\theta v-iA_\theta^\ast A_\theta w= A_\theta^{\C}(A_\theta^{\C})^\ast (v+iw)
\]
for all $v,w\in \dom(A_\theta^\ast A_\theta)=\dom(A_\theta A_\theta^\ast)$ and $A_\theta^{\C}$ is normal on $H^{\C}$. With the canonical isometric injection $\iota:H\to H^{\C}$, $\iota (v):=v+i\cdot0$ we have $A_\theta^{\C}\iota(v)=A_\theta v$, $v\in H$, and $A_\theta^{\C}$ extends $A_\theta$. The extension of a bounded operator $T$ on $H$ satisfies $\norm{T^{\C}}=\norm{T}$ and $\norm{T^{\C}}_{HS}^2=2\norm{T}_{HS}^2$ in case of a Hilbert-Schmidt operator. In order to ease notation, the superscript $\C$ will be dropped from now on and functional calculus for operators will be implicitly complexified. Let us emphasise that the observation model and in particular the Gaussian  noise is always defined over the real numbers.

By \citet[Theorem 6.3.11]{BS2018}  the normal operators $A_\theta$ on the complex Hilbert space $H$ satisfy $\dom(A_\theta)=\dom(A_\theta^\ast)$ and can be decomposed by two self-adjoint operators $R_\theta:\dom(R_\theta)\to H$, $J_\theta:\dom(J_\theta)\to H$ with $\dom(A_\theta)=\dom(R_\theta)\cap\dom(J_\theta)$ such that
\[ A_\theta v=R_\theta v + iJ_\theta v,\quad A_\theta^\ast v=R_\theta v - iJ_\theta v,\quad \norm{A_\theta v}^2=\norm{R_\theta v}^2+\norm{J_\theta v}^2
\]
for all $v\in\dom(A_\theta)$.
Note that any two operators from $A_\theta, A^*_\theta, R_\theta, J_\theta$ commute.
Since $A_\theta$ is normal, we can use its spectral measure to define $f(A_\theta)$ for measurable $f:\C\to\C$ \citep{BiSo2012,schmudgen2012}, in particular $R_\theta=\Re(A_\theta)$ and $J_\theta=\Im(A_\theta)$ hold on $\dom(A_\theta)$. Frequently, we shall make use of the bound $\norm{f(A_\theta)}\le\norm{f}_\infty$ for bounded $f$. Moreover, we shall employ the Bochner integral over Banach space-valued functions (e.g. with values in a Hilbert space or a space of bounded linear operators) and use standard properties, as exposed e.g. in \citet[Appendix C]{EN2000}. We lift linear operators $L$ on $H$ naturally to $L^2(H)$ by setting pointwise $(Lf)(t):=L(f(t))$ for $f\in L^2(H)$.

\begin{lemma}\label{LemCformula}
In the setting of Lemma \ref{LemCovOp},
we have the explicit representation
\[
C_\theta f(t) = \frac12\int_0^T e^{iJ_\theta t}\Big(\int_{\abs{t-s}}^{t+s} e^{R_\theta v}dv\Big)e^{-iJ_\theta s} f(s)\,ds.
\]
\end{lemma}

\begin{proof}
The asserted identity  follows from
Fubini's theorem for Bochner integrals with  continuous integrands:
\begin{align*}
C_\theta f(t)=S_\theta S_\theta^\ast f(t)
&= \int_0^t e^{A_\theta(t-u)}\Big(\int_u^Te^{A^*_\theta(s-u)}f(s)\,ds\Big)\,du\\
&=\int_0^T \int_0^{t\wedge s}e^{A_\theta t-2R_\theta u+A_\theta^\ast s}du\, f(s)\,ds\\
&=\int_0^T e^{iJ_\theta t}\int_0^{t\wedge s}e^{R_\theta t-2R_\theta u+R_\theta s}du\,e^{-iJ_\theta s} f(s)\,ds\\
&=\frac12\int_0^T e^{iJ_\theta t}\Big(\int_{\abs{t-s}}^{t+s} e^{R_\theta v}dv\Big)e^{-iJ_\theta s} f(s)\,ds.
\end{align*}
\end{proof}

We proceed to bound the operator norm of $R_\theta S_\theta$ on $L^2(H)$. This is analogous to the convolution operator $f\mapsto \int_0^T \Re(\lambda)e^{\lambda(\cdot-s)}f(s)ds$ on $L^2([0,T];\C)$ with the classical norm bound $\norm{\Re(\lambda)e^{\Re(\lambda)\cdot}}_{L^1([0,T])}$. The proof involves, however, more involved functional calculus and tensorisation.

\begin{proposition}\label{PropRSnorm}
With $(x)_+=x\vee 0$ we have for the operator norm in $L^2(H)$
\begin{equation}\label{EqAlphatheta} \norm{R_\theta S_\theta}^2=\norm{S^*_\theta R_\theta}^2\le \tfrac34+\tfrac{1}{4} e^{2\norm{(R_\theta)_+}T}=:\alpha_\theta^2.
\end{equation}
\end{proposition}

\begin{proof}
See Section \ref{SecProofSEE}.
\end{proof}

We have a tighter bound in the contractive setting where $R_\theta\preccurlyeq 0$. In the sequel, we also write short $\abs{a}_T^{-p}:=\abs{a}^{-p}\wedge T^p$ and $\abs{a}_{T^{-1}}^{p}:=\abs{a}^p\vee T^{-p}=(\abs{a}_T^{-p})^{-1}$ for $a\in\C$, $p,T>0$ with $\abs{0}_T^{-p}:=T^p$.

\begin{corollary}\label{CorRSnorm}
If $R_\theta\preccurlyeq0$ holds, then we can bound
\begin{equation}\label{EqAlphatheta2} \norm{\abs{R_\theta}_{T^{-1}} S_\theta}=\norm{S^*_\theta \abs{R_\theta}_{T^{-1}}}\le 1.
\end{equation}
\end{corollary}

\begin{proof}
For $R_\theta\preccurlyeq 0$ we have $(R_\theta)_+=0$ and $\alpha_\theta=1$ in Proposition \ref{PropRSnorm}. By functional calculus, $\Pi_{\theta,T}:={\bf 1}(R_\theta\in[-T^{-1},0])$ and $\Id-\Pi_{\theta,T}$ are orthogonal projections in $L^2(H)$ ($R_\theta$ being lifted from $H$ to $L^2(H)$). Since $S_\theta$ commutes with $R_\theta$, it also commutes with $\Pi_{\theta,T}$ so that by Proposition \ref{PropRSnorm} for $f\in L^2(H)$
\[ \norm{(\Id-\Pi_{\theta,T})\abs{R_\theta} S_\theta f}_{L^2(H)}=\norm{-R_\theta S_\theta (\Id-\Pi_{\theta,T})f}_{L^2(H)}\le \norm{(\Id-\Pi_{\theta,T})f}_{L^2(H)}
\]
follows. On the other hand, we have by direct calculation
\begin{align*}
\norm{\Pi_{\theta,T} S_\theta f}_{L^2(H)}^2 &=\int_0^T\bnorm{\int_0^t e^{A_\theta(t-s)}\Pi_{\theta,T}f(s)\,ds}^2dt\\
&\le \int_0^T\Big(\int_0^t \norm{e^{A_\theta(t-s)}}\norm{\Pi_{\theta,T}f(s)}\,ds\Big)^2dt\\
&\le T^2\norm{\Pi_{\theta,T}f}_{L^2(H)}^2,
\end{align*}
where $\norm{e^{A_\theta(t-s)}}=\norm{e^{R_\theta(t-s)}}\le 1$ was used due to $R_\theta\preccurlyeq0$. Combining the two bounds, we obtain the claim, writing $(\Id-\Pi_{\theta,T})\abs{R_\theta}+T^{-1}\Pi_{\theta,T}=\abs{R_\theta}_{T^{-1}}$.
\end{proof}

In view of Proposition \ref{PropHellinger} we can bound the Hellinger distance between the observations in \eqref{EqY} for $\theta_0$, $\theta_1$ by bounding the Hilbert-Schmidt norm of $Q_{\bf 0}^{-1/2}Q_{\bf 1}^{1/2}-(Q_{\bf 1}^{-1/2}Q_{\bf 0}^{1/2})^\ast$ on $L^2(H)$. Our aim is to find tight bounds only involving the generators $A_{\bf 0}$ and $A_{\bf 1}$ on $H$.

We establish simple properties of the operator $S_\theta$, which are similar to the well known solution theory for non-homogeneous Cauchy problems \cite[Appendix 3]{DPZa2014}. We use the $H$-valued Sobolev space ${\cal H}^1([0,T];H)={\cal H}^1(H)$ and its subspaces ${\cal H}_0^1(H)$, ${\cal H}_T^1(H)$ of functions $f\in {\cal H}^1(H)$ vanishing in $0$ and in $T$, respectively, see Appendix \ref{SecNotation} for more details.

\begin{proposition}\label{PropSInverse} Suppose $\dom(A_\theta)=\dom(R_\theta)$ for some $\theta\in\Theta$. Then:
\begin{enumerate}
\item $S_\theta,S_\theta^\ast$ map $L^2([0,T];H)$ into $L^2([0,T];\dom(A_\theta))$, $S_\theta$ maps  ${\cal H}^1(H)$ into ${\cal H}_0^1(H)$ and $S_\theta^\ast$ maps ${\cal H}^1(H)$ into ${\cal H}_T^1(H)$.
\item Let $\partial_t$ denote the time derivative. Then we have the identities $(\partial_t-A_\theta)S_\theta=(-\partial_t-A_\theta^\ast)S_\theta^\ast=\Id$ on ${\cal H}^1(H)$ and $S_\theta (\partial_t-A_\theta)=S_\theta^\ast (-\partial_t-A_\theta^\ast)=\Id$ on ${\cal H}_0^1(H)\cap L^2([0,T];\dom(A_\theta))$.
\item We have on $L^2(H)$
\[ S_{\bf 1}-S_{\bf 0}=S_{\bf 0}(A_{\bf 1}-A_{\bf 0})S_{\bf 1}=S_{\bf 1}(A_{\bf 1}-A_{\bf 0})S_{\bf 0}.\]
\end{enumerate}
\end{proposition}

\begin{proof}
	See Section \ref{SecProofSEE}.
\end{proof}

The perturbation property in (c) will be essential for us. The assumption $\dom(A_\theta)=\dom(R_\theta)$ is required to have all evaluations well defined. Since always $\dom(A_\theta)\subset\dom(R_\theta)$ holds, it means $\abs{J_\theta}\preccurlyeq C(\abs{R_\theta}+\Id)$ for some constant $C>0$ and $A_\theta$ is a sectorial operator.

The next step is to find an appropriate upper bound for $Q_\theta^{-1/2}$, appearing in the Hellinger bound. To do so, we assume  that there are injective  operators $\bar B_\theta\succ 0$, commuting with $A_\theta$, such that
\begin{equation}\label{EqBtheta}
 \bar B_\theta^2\succcurlyeq B^\ast B\text{ or equivalently } \norm{\bar B_\theta^{-1}B^\ast}\le 1.
\end{equation}
Assuming $\bar B_\theta$ to be injective is not very restrictive. If $\bar B_\theta$ is not invertible, we can add $\delta \Id$ and then let $\delta\downarrow 0$ in the final lower bound.

\begin{lemma}\label{LemQbound}
For an invertible operator $\bar B_\theta$, commuting with $A_\theta$ and satisfying \eqref{EqBtheta}, we have
\[  (BS_\theta)^{\ast}Q_\theta^{-1}BS_\theta \preccurlyeq (\eps_\theta^{2}R_\theta^2\bar B_\theta^{-2}+\Id)^{-1}
\]
for $\eps>0$,
where $\eps_\theta:=\eps/\alpha_\theta$ with $\alpha_\theta$ from \eqref{EqAlphatheta}. If additionally $R_\theta\preccurlyeq0$, then we even have
\[  (BS_\theta)^{\ast}Q_\theta^{-1}BS_\theta \preccurlyeq (\eps^{2}\abs{R_\theta}_{T^{-1}}^2\bar B_\theta^{-2}+\Id)^{-1}.
\]
\end{lemma}

\begin{proof}
Proposition \ref{PropRSnorm} yields $\Id\succcurlyeq \alpha_\theta^{-2}S_\theta R_\theta^2S_\theta^\ast$, whence by Lemma \ref{LemCovOp} and the property \eqref{EqBtheta} of $\bar B_\theta$
\[ Q_\theta=\eps^2 \Id+BC_\theta B^\ast\succcurlyeq B(\eps^2 \bar B_\theta^{-2}+S_\theta S_\theta^\ast)B^\ast \succcurlyeq BS_\theta(\eps_\theta^{2}R_\theta^2\bar B_\theta^{-2}+\Id)(BS_\theta)^\ast.\]
Therefore the first claim follows from Lemma \ref{LemTraceBound}(b) below. The second claim follows when Corollary \ref{CorRSnorm} is used instead of Proposition \ref{PropRSnorm}.
\end{proof}

We arrive at our main general result. For this let $f(\lambda)=(\int_0^1\int_0^t e^{\lambda v}dvdt)^{1/2}$, i.e. $f(\lambda)=(\frac{e^{\lambda}-1-\lambda}{\lambda^2})^{1/2}$ for $\lambda\in\R\setminus\{0\}$ and $f(0)=1/\sqrt 2$.

\begin{theorem}\label{ThmMain}
Assume for $\theta\in\{\theta_0,\theta_1\}$ that $\dom(A_\theta)=\dom(R_\theta)$ and that there are $\bar B_\theta\succ0$ which commute with $A_\theta$ and satisfy \eqref{EqBtheta}. Then:
\begin{align*}
&H^2({\cal N}_{cyl}(0,Q_{\bf 0}),{\cal N}_{cyl}(0,Q_{\bf 1}))\\
&\le \tfrac{T^2}2\norm{(\eps_{\bf 0}^2R_{\bf 0}^2\bar B_{\bf 0}^{-2}+\Id)^{-1/2}(A_{\bf 1}-A_{\bf 0}) f(2TR_{\bf 1})(\eps_{\bf 1}^2R_{\bf 1}^2\bar B_{\bf 1}^{-2}+\Id)^{-1/2}}_{HS}^2\\
&\quad +\tfrac{T^2}2\norm{(\eps_{\bf 1}^2R_{\bf 1}^2\bar B_{\bf 1}^{-2}+\Id)^{-1/2}(A_{\bf 1}-A_{\bf 0}) f(2TR_{\bf 0})(\eps_{\bf 0}^2R_{\bf 0}^2\bar B_{\bf 0}^{-2}+\Id)^{-1/2}}_{HS}^2.
\end{align*}
If additionally $R_{\bf 1}\preccurlyeq0$ and $R_{\bf 0}\preccurlyeq 0$ hold, then we even have
\begin{align*}
&H^2({\cal N}_{cyl}(0,Q_{\bf 0}),{\cal N}_{cyl}(0,Q_{\bf 1}))\\
&\le \tfrac {T}4 \norm{(\eps^2\abs{R_{\bf 0}}_{T^{-1}}^2\bar B_{\bf 0}^{-2}+\Id)^{-1/2}(A_{\bf 1}-A_{\bf 0}) \abs{R_{\bf 1}}_{T^{-1}}^{-1/2}(\eps^2\abs{R_{\bf 1}}_{T^{-1}}^2\bar B_{\bf 1}^{-2}+\Id)^{-1/2}}_{HS}^2\\
&\quad +\tfrac {T}4 \norm{(\eps^2\abs{R_{\bf 1}}_{T^{-1}}^2\bar B_{\bf 1}^{-2}+\Id)^{-1/2}(A_{\bf 1}-A_{\bf 0}) \abs{R_{\bf 0}}_{T^{-1}}^{-1/2}(\eps^2\abs{R_{\bf 0}}_{T^{-1}}^2\bar B_{\bf 0}^{-2}+\Id)^{-1/2}}_{HS}^2.
\end{align*}
\end{theorem}

\begin{proof}
We apply Proposition \ref{PropHellinger} with the interpretation of Remark \ref{RemHell}(b), Lemma \ref{LemCovOp} on the form of the covariance operator $Q_\theta$ and  the formula from Proposition \ref{PropSInverse}(c)   consecutively to obtain:
\begin{align*}
&H^2({\cal N}_{cyl}(0,Q_{\bf 0}),{\cal N}_{cyl}(0,Q_{\bf 1}))\\
& \le \tfrac14 \norm{Q_{\bf 0}^{-1/2}(Q_{\bf 1}-Q_{\bf 0})Q_{\bf 1}^{-1/2}}_{HS(L^2([0,T];H))}^2\\
& = \tfrac14 \norm{Q_{\bf 0}^{-1/2}B(S_{\bf 1}S_{\bf 1}^\ast-S_{\bf 0}S_{\bf 0}^\ast)B^\ast Q_{\bf 1}^{-1/2}}_{HS(L^2(H))}^2\\
& = \tfrac14 \norm{Q_{\bf 0}^{-1/2}B((S_{\bf 1}-S_{\bf 0})S_{\bf 1}^\ast+S_{\bf 0}(S_{\bf 1}-S_{\bf 0})^\ast)B^\ast Q_{\bf 1}^{-1/2}}_{HS(L^2(H))}^2\\
& = \tfrac14 \norm{Q_{\bf 0}^{-1/2}BS_{\bf 0}((A_{\bf 1}-A_{\bf 0})S_{\bf 1}+((A_{\bf 1}-A_{\bf 0})S_{\bf 0})^\ast)(BS_{\bf 1})^\ast Q_{\bf 1}^{-1/2}}_{HS(L^2(H))}^2
\end{align*}
In view of the monotonicity in Lemma \ref{LemTraceBound}(a) below this can be bounded by Lemma \ref{LemQbound} for $\eps>0$ and by the polar decomposition $BS_\theta=Q_\theta^{1/2}U$ with unitary $U$ for $\eps=0$ such that
\begin{align}
&H^2({\cal N}_{cyl}(0,Q_{\bf 0}),{\cal N}_{cyl}(0,Q_{\bf 1}))\label{EqH2QR}\\
&\le
 \tfrac12 \norm{(\eps_{\bf 0}^{2}R_{\bf 0}^2\bar B_{\bf 0}^{-2}+\Id)^{-1/2} (A_{\bf 1}-A_{\bf 0})S_{\bf 1}(\eps_{\bf 1}^{2}R_{\bf 1}^2\bar B_{\bf 1}^{-2}+\Id)^{-1/2} }_{HS(L^2(H))}^2\nonumber\\
 &+\tfrac12 \norm{(\eps_{\bf 0}^{2}R_{\bf 0}^2\bar B_{\bf 0}^{-2}+\Id)^{-1/2} ((A_{\bf 1}-A_{\bf 0})S_{\bf 0})^\ast (\eps_{\bf 1}^{2}R_{\bf 1}^2\bar B_{\bf 1}^{-2}+\Id)^{-1/2} }_{HS(L^2(H))}^2.\nonumber
\end{align}
Given the form of $S_{\bf 1}$ in Lemma \ref{LemCovOp}, we obtain from the Hilbert-Schmidt norm for kernel operators in Lemma \ref{lem:HS:kernel} below
\begin{align*}
& \norm{(\eps_{\bf 0}^{2}R_{\bf 0}^2\bar B_{\bf 0}^{-2}+\Id)^{-1/2}(A_{\bf 1}-A_{\bf 0})S_{\bf 1} (\eps_{\bf 1}^{2}R_{\bf 1}^2\bar B_{\bf 1}^{-2}+\Id)^{-1/2} }_{HS(L^2(H))}^2\\
&= \int_0^T\int_0^t \norm{(\eps_{\bf 0}^{2}R_{\bf 0}^2\bar B_{\bf 0}^{-2}+\Id)^{-1/2}(A_{\bf 1}-A_{\bf 0})e^{A_{\bf 1}(t-v)} (\eps_{\bf 1}^{2}R_{\bf 1}^2\bar B_{\bf 1}^{-2}+\Id)^{-1/2} }_{HS(H)}^2 dvdt.
\end{align*}
Now, $A_{\bf 1},R_{\bf 1},\bar B_{\bf 1}$ commute and by the linearity of the trace the last expression equals
\begin{align*}
&\int_0^T\int_0^t \trace\Big((A_{\bf 1}-A_{\bf 0})^\ast(\eps_{\bf 0}^{2}R_{\bf 0}^2\bar B_{\bf 0}^{-2}+\Id)^{-1}(A_{\bf 1}-A_{\bf 0})e^{2R_{\bf 1}u} (\eps_{\bf 1}^{2}R_{\bf 1}^2\bar B_{\bf 1}^{-2}+\Id)^{-1}\Big)dudt\\
&= \trace\Big((A_{\bf 1}-A_{\bf 0})^\ast(\eps_{\bf 0}^{2}R_{\bf 0}^2\bar B_{\bf 0}^{-2}+\Id)^{-1}(A_{\bf 1}-A_{\bf 0})T^2 f(2TR_{\bf 1})^2 (\eps_{\bf 1}^{2}R_{\bf 1}^2\bar B_{\bf 1}^{-2}+\Id)^{-1}\Big)\\
&= T^2\norm{(\eps_{\bf 0}^{2}R_{\bf 0}^2\bar B_{\bf 0}^{-2}+\Id)^{-1/2}(A_{\bf 1}-A_{\bf 0}) f(2TR_{\bf 1}) (\eps_{\bf 1}^{2}R_{\bf 1}^2\bar B_{\bf 1}^{-2}+\Id)^{-1/2}}_{HS}^2.
\end{align*}
Changing the roles of $\theta_0$ and $\theta_1$ and using $\norm{F}_{HS}^2=\norm{-F^*}^2_{HS}$ for Hilbert--Schmidt operators $F$, we obtain an analogous bound for the second summand in \eqref{EqH2QR}.
This proves the first inequality.

The improved bound for $R_\theta\preccurlyeq0$ follows the same way, but using the second inequality of Lemma \ref{LemQbound} for bounding \eqref{EqH2QR}  and using Lemma \ref{LemTraceBound}(a) below with $f(2TR_\theta)^2\preccurlyeq \frac12 T^{-1}\abs{R_\theta}_T^{-1}$, which follows from $f(-x)^2=(e^{-x}-1+x)/x^2\le \frac12\wedge\abs{x}^{-1}$ for $x> 0$.
\end{proof}

\section{Parametric minimax rates}\label{SecParRates}

We consider a slightly simpler model for $X$ and $dY$ than \eqref{EqX}, \eqref{EqY}, in particular assuming a real parameter $\theta$ and commutativity of all operators. We construct an estimator whose error rate attains under mild restrictions the lower bound. This establishes the minimax rate in a quite general parametric setting.

For each $n\in\N$ let the observations $(dY_t,t\in[0,T_n])$ be given by
\begin{equation}\label{EqPar}
dY_t=  B_nX_tdt+\eps_n dV_t\text{ with } dX_t=A_{\theta} X_t\,dt+ dW_t
\end{equation}
and $X_0=0$, $T_n\ge 1$, $\eps_n\in[0,1]$, $\theta\in[\underline\theta_n,\bar\theta_n]$ for $\bar\theta_n>\underline\theta_n\ge 0$. The parametrisation of the operator $A_{\theta}$ is given by
\[A_\theta=A_{\theta,n}=M_n+\theta\Lambda\]
with known possibly unbounded linear operators $M_n,\Lambda$.

\begin{assumption}\label{AssPar}\
\begin{enumerate}
\item $M_n$ is selfadjoint and $\Lambda$ is normal with $M_n\precsim 0$, $\Re(\Lambda)\precsim 0$. Moreover, $\dom(A_\theta)=\dom(R_\theta)$ holds for  $\theta\in[\underline\theta_n,\bar\theta_n]$.
\item $B_n$ is a bounded and invertible selfadjoint operator.
\item All three operators $M_n$, $\Lambda$ and $B_n$ commute and allow for a common functional calculus.
\item $A_\theta$ has a compact resolvent  for  $\theta\in[\underline\theta_n,\bar\theta_n]$.
\end{enumerate}
\end{assumption}

We abbreviate $\bar A_n=A_{\bar\theta_n}$, $\bar R_n=R_{\bar\theta_n}$, $\bar J_n=J_{\bar\theta_n}$. By $M_n\precsim0$, $\Re(\Lambda)\precsim 0$ and   $J_{\theta}=\theta\Im(\Lambda)$, we have the ordering $\abs{R_{\theta}}\preccurlyeq\abs{\bar R_n}$, $\abs{A_{\theta}}\preccurlyeq\abs{\bar A_n}$ under Assumption \ref{AssPar}(a). For later applications note that  under Assumption \ref{AssPar}(c) observing \eqref{EqPar} is equivalent to observing
\begin{equation}\label{EqPartilde}
dY_t=  \tilde X_tdt+\eps_n dV_t\text{ with } d\tilde X_t=A_{\theta} \tilde X_t\,dt+ B_ndW_t
\end{equation}
with $\tilde X_0=0$ as in Example \ref{ex:ObsToDynamicNoise}.
We obtain immediately a general lower bound for this setting.

\begin{theorem}\label{ThmLBpar}
Grant Assumption \ref{AssPar}(a-c). Let
\[ v_n= T_n^{-1/2} \Big(\trace\big(\abs{\Lambda}^2 (\eps_n^4\abs{\bar R_n}_{T_n^{-1}}^4 B_n^{-4}+\Id)^{-1}\abs{\bar R_n}_{T_n}^{-1}\big)\Big)^{-1/2},\]
and assume $\bar\theta_n-\underline\theta_n\ge v_n>0$. Then the lower bound
\[ \inf_{\hat\theta_n}\sup_{\theta\in[\underline{\theta}_n,\bar\theta_n]}\PP_\theta\big(v_n^{-1}\abs{\hat\theta_n-\theta} \ge 2^{-7/2}\big)\ge \tfrac{2-\sqrt{3}}{4}, \]
holds where the infimum is taken over all estimators based on observing \eqref{EqPar}. Hence, estimators $\hat\theta_n$ that fulfill $\hat\theta_n-\theta={\cal O}_{\PP_\theta}(v_n)$ uniformly over $\theta\in [\underline{\theta}_n,\bar\theta_n]$ are minimax rate-optimal.
\end{theorem}

\begin{proof}
By Theorem \ref{ThmMain} for $R_{\bf 0}, R_{\bf 1}\preccurlyeq 0$ we have for $\theta_{\bf 0}=\bar\theta_n-2^{-5/2}v_n$ and $\theta_{\bf 1}=\bar\theta_n$ in our setting, noting $\theta_{\bf 0},\theta_{\bf 1}\in[\underline{\theta}_n,\bar\theta_n]$, $\abs{R_{\bf 0}} \preccurlyeq\abs{R_{\bf 1}}$ and the commutativity,
\begin{align*}
H^2({\cal N}_{cyl}(0,Q_{\bf 0}),{\cal N}_{cyl}(0,Q_{\bf 1}))
&\le T_n2^{-5}v_n^2 \norm{(\eps_n^2\abs{R_{\bf 0}}_{T_n^{-1}}^2 B_n^{-2}+\Id)^{-1}\Lambda \abs{R_{\bf 0}}_{T_n}^{-1/2}}_{HS}^2\\
&\le T_n v_n^2 \trace\Big(\abs{\Lambda}^2 (\eps_n^4 2^4\abs{R_{\bf 0}}_{T_n^{-1}}^4 B_n^{-4}+\Id)^{-1}\tfrac12\abs{R_{\bf 0}}_{T_n}^{-1}\Big).
\end{align*}
We have $\abs{R_{\bf 0}}^{-1}\preccurlyeq \frac{\bar\theta_n}{\bar\theta_n-2^{-5/2}v_n} \abs{\bar R_n}^{-1}\preccurlyeq 2\abs{\bar R_n}^{-1}$  so that our choice of $v_n$ yields
$H^2({\cal N}_{cyl}(0,Q_{\bf 0}),{\cal N}_{cyl}(0,Q_{\bf 1}))\le 1$.
Therefore the result follows from Theorem \ref{ThmLB} with $\delta=2^{-5/2}v_n$.
\end{proof}

Estimation of the parameter $\theta$ in \eqref{EqPar} is non-trivial. In the finite-dimensional case the laws of the Ornstein-Uhlenbeck processes are equivalent for different $\theta$ and estimation of $\theta$ could be based on standard filter theory. Even then, however, the maximum-likelihood estimator (MLE) is not explicit and a one-step MLE based on a preliminary estimator is  preferred \citep{KuZh2021}. In the infinite-dimensional setting, e.g. for the stochastic heat equation with $A_\theta=\theta\Delta$, the laws of $(X_t,t\in[0,T_n])$ are often singular for different parameters $\theta$, see also the interesting study by \citet{HiTr2021} for one-dimensional SPDEs under different observation schemes. In that case the observations $(dY_t,t\in[0,T_n])$ only become equivalent through the action of the observation noise $dV_t$ and the likelihood becomes intractable besides an abstract Gaussian approach.

Our estimation procedure is based on a preaveraging method, where the observational noise is first reduced by local averaging and then a regression-type estimator is applied to the averaged data. A similar approach is employed in nonparametric drift estimation for diffusions under noise by \citet{schmisser2011}, yet there the measurement noise level persists in the final rate, which we avoid, when specialised to the parametric Ornstein-Uhlenbeck case.

If we observed $X$ directly, a natural estimation approach for $\theta$ would be to regress $dX_t-M_nX_tdt=\theta\Lambda X_tdt+dW_t$ on $\Lambda X_tdt$  over $t\in[0,T]$. Given the noisy observations $dY$ of $X$, we smooth out the  observational noise $dV$ by regressing the average $\int_0^T (\partial_t K^{(n)}(t,s)-M_nK^{(n)}(t,s))\,dY_s$ on $\Lambda\,dY_t$ with an operator-valued kernel $K^{(n)}(t,s)$, similarly to estimation in instrumental variable regression.  We require $K^{(n)}(t,s)=0$ for $s> t$ so that the stochastic integral is taken over $s\in[0,t]$ to keep the martingale structure. Integration by parts with $(\partial_t-M_n)^\ast=-(\partial_t+M_n)$ and vanishing boundary terms leads to our estimator.

\begin{definition}\label{Defhattheta}
For the operator-valued kernel $K^{(n)}$ given by
\begin{align*}
&K^{(n)}(t,s) = K_n\psi_{t-s,s}^{(n)}(\bar A_n),\quad t,s\in[0,T],\text{ with }\\
& \psi_{v,s}^{(n)}(a)=v_+\wedge\big(\abs{a}_{T_n-s}^{-1}-v\big)_+,\;
K_n=B_n^2\big(\eps_n^4\abs{\bar R_n}_{T_n^{-1}}+B_n^4\abs{\bar A_n}_{T_n}^{-3}\big)^{-1} \abs{\bar A_n}_{T_n}^{-1},
\end{align*}
 consider the estimator $\hat\theta_n:=(Z_n/N_n){\bf 1}(N_n\not=0)$ with
\begin{align}
 Z_n&=-\int_0^{T_n}\int_0^t \scapro{(\partial_t \Lambda K^{(n)}(t,s)+M_n \Lambda K^{(n)}(t,s))dY_s}{dY_t}\label{EqZ},\\
 N_n&=\int_0^{T_n}\int_0^t \scapro{\abs{\Lambda}^2 K^{(n)}(t,s)dY_s}{dY_t} \label{EqN}
 \end{align}
 based on observing $dY_t$ according to \eqref{EqPar}.
 \end{definition}

A detailed analysis yields the main parametric upper bound.

\begin{theorem}\label{ThmUpperBound}
Grant Assumption \ref{AssPar}(a-d). Assume that
\[{\cal I}_n(\theta):=T_n\trace\Big(\abs{\Lambda}^2 \big(\eps_n^4\abs{\bar R_n}_{T_n^{-1}}\abs{\bar A_n}_{T_n^{-1}}^3B_n^{-4}+\Id\big)^{-1} \abs{ R_{\theta}}_{T_n}^{-1} \Big)
\]
is finite for all $n$ with $\lim_{n\to\infty}{\cal I}_n(\theta)=\infty$ and
\begin{align}
&T_n\trace\Big(\abs{\Lambda}^4\big(\eps_n^4\abs{\bar R_n}_{T_n^{-1}}\abs{\bar A_n}_{T_n^{-1}}^3 B_n^{-4}+\Id)^{-2}
 \big(\eps_n^4\abs{\bar A_n}_{T_n^{-1}}B_n^{-4}+\abs{R_{\theta}}_{T_n}^{-3}\big)\Big)\nonumber\\
 &= o({\cal I}_n(\theta)^2)\label{EqVarNBound}
 \end{align}
uniformly over $\theta\in[\underline\theta_n,\bar\theta_n]$.
Then $\hat\theta_n$ is well defined and
\[ \hat\theta_n-\theta={\cal O}_{\PP_\theta}\big({\cal I}_n(\theta)^{-1/2}\big)\]
holds uniformly over $\theta\in[\underline\theta_n,\bar\theta_n]$.
\end{theorem}

\begin{proof} See Section \ref{SecUpperBoundProofs} below.\end{proof}

\begin{remark}\label{RemIdent}
If ${\cal I}_n(\theta)=\infty$ holds already for finite $n$, then $\theta$ can be usually identified  non-asymptotically. To see this, denote by $\Pi_j$ the orthogonal projection from $H$ onto the eigenspace $V_j$ for the first $j$ eigenvalues of $A_\theta$ (ordered arbitrarily, taking account of multiplicities). Remark that $A_\theta$ has pure point spectrum under  Assumption \ref{AssPar}(d). Then for fixed $n$ we can consider the projected finite-dimensional observations $\Pi_jdY_t$ of $dY_t$ from \eqref{EqPar} on $V_j$, involving $\Pi_jdX_t=A_{\theta,j} X_tdt+dW_{t,j}$ with finite-rank operator $A_{\theta,j}=\Pi_jA_\theta$ and $V_j$-valued Brownian motion $W_{t,j}=\Pi_jW_t$. The corresponding estimator $\hat\theta_{n,j}$ converges for $j\to\infty$ to $\theta$ in probability (or identify $\theta$) if
\begin{align*}
&T_n\trace\Big(\abs{\Lambda_j}^4\big(\eps_n^4\abs{\bar R_{n,j}}_{T_n^{-1}}\abs{\bar A_{n,j}}_{T_n^{-1}}^3 B_{n,j}^{-4}+\Id)^{-2}
 \big(\eps_n^4\abs{\bar A_{n,j}}_{T_n^{-1}}B_{n,j}^{-4}+\abs{R_{\theta,j}}_{T_n}^{-3}\big)\Big)\\
 &= o({\cal I}_{n,j}(\theta)^2)
 \end{align*}
 holds with $F_j:=\Pi_jF|_{V_j}$ denoting the projection of an operator $F$ onto $V_j$. The convergence follows from Theorem \ref{ThmUpperBound} with index $j$ instead of $n$
 because $\dim(V_j)<\infty$ and $\bigcup_{j\ge 1}V_j=H$ imply ${\cal I}_{n,j}(\theta)<\infty$ and ${\cal I}_{n,j}(\theta)\uparrow {\cal I}_n(\theta)=\infty$ as $j\to\infty$.
\end{remark}

We can find simple sufficient conditions for rate optimality.

\begin{corollary}\label{CorUpperBound}
Grant the assumptions of Theorem \ref{ThmUpperBound}. If $\abs{\bar J_n}\preccurlyeq C\abs{\bar R_n}$ holds for a constant $C>0$, independent of $n$, then we have $\hat\theta_n-\theta={\cal O}_{\PP_\theta}({\cal I}_n(\theta)^{-1/2})$ uniformly over $\theta\in[\underline\theta_n,\bar\theta_n]$ with
\begin{equation}\label{EqI2}
{\cal I}_n(\theta)\thicksim T_n\trace\Big(\abs{\Lambda}^2 (\eps_n^4\abs{\bar R_n}_{T_n^{-1}}^4B_n^{-4}+\Id)^{-1} \abs{ R_{\theta}}_{T_n}^{-1}\Big)
\end{equation}
and $\hat\theta_n$ is minimax rate-optimal if $\underline{\theta}_n+{\cal I}_n(\bar\theta_n)^{-1/2}\le\bar\theta_n\lesssim \underline\theta_n$.

Moreover, in this case a sufficient condition for \eqref{EqVarNBound} is
\begin{equation}\label{EqVarNbound2}
T_n\trace\Big(\abs{\Lambda}^4\big(\eps_n^4\abs{\bar R_n}_{T_n^{-1}}^4 B_n^{-4}+\Id)^{-1}\abs{R_{\theta}}_{T_n}^{-3}\Big)=o({\cal I}(\theta)^2)
\end{equation}
uniformly for $\theta\in[\underline\theta_n,\bar\theta_n]$, which is satisfied for
$\norm{\Lambda\abs{R_{\underline{\theta}_n}}_{T_n}^{-1}}\lesssim 1$.
\end{corollary}

\begin{proof}
From $\abs{\bar J_n} \preccurlyeq C\abs{\bar R_n}$ we infer $\abs{\bar R_n} \preccurlyeq\abs{\bar A_n} \preccurlyeq (C+1)\abs{\bar R_n}$ and hence in Theorem \ref{ThmUpperBound}
\begin{equation*}
{\cal I}_n(\theta)\thicksim T_n\trace\Big(\abs{\Lambda}^2 \big(\eps_n^4 \abs{\bar R_n}_{T_n^{-1}}^4B_n^{-4}+\Id\big)^{-1} \abs{R_{\theta}}_{T_n}^{-1} \Big).
\end{equation*}
The minimax optimality follows directly from the lower bound in Theorem \ref{ThmLBpar} because $\abs{R_\theta}$ has the same order for all $\theta\in[\underline\theta_n,\bar\theta_n]$ by $\bar\theta_n\lesssim \underline\theta_n$.

Using that $\abs{\bar R_n}$ is of the same order as $\abs{\bar A_n}$ and $\abs{R_{\theta}}\preccurlyeq\abs{\bar R_n}$, the left-hand side of Condition \eqref{EqVarNBound} is at most of order
\[T_n\trace\Big(\abs{\Lambda}^4\big(\eps_n^4\abs{\bar R_n}_{T_n^{-1}}^4 B_n^{-4}+\Id)^{-1}\abs{R_{\theta}}_{T_n}^{-3}\Big)\le {\cal I}_n(\theta)\norm{\abs{\Lambda}^2 \abs{R_{\theta}}_{T_n}^{-2}}.
\]
For $\norm{\Lambda\abs{R_{\underline{\theta}_n}}_{T_n}^{-1}}\lesssim 1$ this bound has the order ${\cal I}_n(\theta)=o({\cal I}_n(\theta)^2)$ due to $\abs{R_{\theta}}_{T_n}^{-1}\preccurlyeq\abs{\bar R_{\underline{\theta}_n}}_{T_n}^{-1}$ and ${\cal I}_n(\theta)\to\infty$. This yields the two sufficient conditions for \eqref{EqVarNBound}.
\end{proof}

If the spectrum of the real part $R_{\theta}$ is asymptotically approaching zero, then the upper bound may not match the lower bound. The reason is that the kernel $K^{(n)}(t,s)$ used for $\hat\theta_n$ involves an indicator ${\bf 1}(\abs{\bar A_n}\le(t-s)^{-1})$, on which event $\Re(e^{A_{\theta}(t-s)})\succcurlyeq e^{-1}\cos(1)\Id$ is ensured and thus $\E[N_n]\gtrsim {\cal I}_n(\theta)$ in Proposition \ref{PropN} below. For unknown $J_{\theta}$ and significantly smaller $R_\theta$ it remains an intriguing open question how to attain the lower bound. Yet, in our cases of interest it concerns only transport estimation under a very small diffusivity asymptotics, see Remark \ref{RemParTransport} below.

\section{Fundamental parametric examples}\label{SEcParEx}

%

\subsection{Scalar Ornstein-Uhlenbeck processes}

Consider the scalar case $H=\R$ of \eqref{EqPartilde} where we observe $(dY_t,t\in[0,T_n])$ for each $n\in\N$ given by
\begin{equation}\label{EqOU}
dY_t=X_tdt+\eps_n dV_t\text{ with } dX_t=-\theta X_tdt+\sigma_n d W_t,\,X_0=0.
\end{equation}
That is, we have $M_n=0$, $\Lambda=-1$ and $B_n=\sigma_n>0$. The parametric theory yields directly upper and lower bounds.

\begin{proposition}\label{PropOU}
Assume  $\bar\theta_n-\underline{\theta}_n\ge v_n(\bar\theta_n)$ with
\[v_n(\theta):=(\theta^{1/2}T_n^{-1/2}\vee T_n^{-1}) (\eps_n^2\sigma_n^{-2}\bar \theta_n^2\vee 1)\to 0\text{ as }n\to\infty.\]
Then for a constant $c>0$ we have the lower bound
\[ \liminf_{n\to\infty}\inf_{\hat\theta_n}\sup_{\theta\in[\underline\theta_n,\bar\theta_n]}
\PP_\theta\big(v_n(\bar\theta_n)^{-1}\abs{\hat\theta_n-\theta}\ge c\big)>0,
\]
where for each $n$ the infimum is taken over all estimators $\hat\theta_n$ based on \eqref{EqOU}.

If $T_n\underline\theta_n(1+\eps_n^4\sigma_n^{-4}\bar\theta_n^4)^{-1}\to\infty$, then the estimator $\hat\theta_n$ from Definition \ref{Defhattheta} satisfies
\[ \hat\theta_n-\theta={\cal O}_{\PP_\theta}(v_n(\theta))\text{ uniformly over }\theta\in[\underline\theta_n,\bar\theta_n].
\]
In particular, for $\theta_n\sigma_n^{-1}\eps_n\lesssim 1$ and $T_n\theta_n\to\infty$ the rate $v_n(\theta_n)=\theta_n^{1/2}T_n^{-1/2}$
is minimax optimal, even locally over $[\theta_n-v_n,\theta_n]$.
\end{proposition}

\begin{proof}
The lower bound follows directly from Theorem \ref{ThmLBpar} for the Ornstein-Uhlenbeck case. For the upper bound we obtain in Theorem \ref{ThmUpperBound} due to $\bar\theta_n\ge\underline\theta_n\gtrsim T_n^{-1}$
\[{\cal I}_n(\theta)\thicksim T_n\big(\eps_n^4\bar \theta_n^4\sigma_n^{-4}+1\big)^{-1} (\theta^{-1}\wedge T_n) \thicksim v_n(\theta)^{-2}.
\]
It remains to check condition \eqref{EqVarNBound} which reads here
\begin{align*}
T_n\big(\eps_n^4\bar\theta_n^4\sigma_n^{-4}+1\big)^{-1}
 \big(\theta^{-3}\wedge T_n^{3}\big)&= o\Big(T_n^2\big(\eps_n^4\bar \theta_n^4\sigma_n^{-4}+1\big)^{-2} (\theta^{-2}\wedge T_n^2) \Big)\\
\iff \eps_n^4\bar\theta_n^4\sigma_n^{-4}+1
 &= o(T_n\theta\vee 1).
\end{align*}
This holds uniformly over $\theta$ by $T_n\underline\theta_n(1+\eps_n^4\sigma_n^{-4}\bar\theta_n^4)^{-1}\to\infty$ and yields the upper bound. Applying the lower bound to $\bar\theta_n=\theta_n$ and $\underline\theta_n=\theta_n-v_n(\theta_n)$ yields minimax optimality due to $\bar\theta_n-v_n(\theta_n)\thicksim\theta_n\ge 0$.
\end{proof}

Let us discuss this minimax rate in different cases and first assume $\eps_n=0$ (no noise). Then the rate $T_n^{-1}\vee \theta_n^{1/2}T_n^{-1/2}$ combines the asymptotic Fisher information bound $\sqrt{2\theta/T_n}$  for fixed $\theta>0$ (positive recurrent case) as $T_n\to\infty$ with the rate $T_n^{-1}$ for $\theta=0$ (null-recurrent case) \citep[Example 2.14, Prop. 3.46]{kutoyants2004}. For the upper bound we need the condition $T_n\theta_n\to\infty$ to stabilise the denominator in $\hat\theta_n$ around its expectation. A more specific analysis, allowing for a diffuse distributional limit, would avoid this condition in a classical way.

In the noisy case $\eps_n>0$ we see that for the asymptotics $T_n\to\infty$ and $\eps_n,\sigma_n,\bar\theta_n$ fixed the rate does not suffer from noisy observations even with constant noise intensity, which for fixed $\theta$ is also well known \cite[Example 3.3]{kutoyants2004}. In view of the spectral approach to SPDE statistics \citep{HuRo1995}, where the operator $A_\theta$ is unbounded and each Fourier mode forms an Ornstein-Uhlenbeck process with drift given by the eigenvalues of $A_\theta$, we note that the noise level intervenes  for the asymptotics $\theta_n\sigma_n^{-1}\eps_n\to \infty$.

\subsection{Stochastic evolution equations without noise}

For $\eps=0$, $B=\Id$ and $R_{\bf0},R_{\bf 1}\preccurlyeq0$ Theorem \ref{ThmMain} shows that $H^2({\cal N}_{cyl}(0,C_{\bf 0}),{\cal N}_{cyl}(0,C_{\bf 1}))$ is bounded by
\begin{align*}
\tfrac{T}4\big(\norm{(A_{\bf 1}-A_{\bf 0}) \abs{R_{\bf 1}}_{T^{-1}}^{-1}}_{HS}^2
 +\norm{(A_{\bf 1}-A_{\bf 0}) \abs{R_{\bf 0}}_{T^{-1}}^{-1}}_{HS}^2\big).
\end{align*}
Let us consider the case of selfadjoint, negative $A_{\bf 0}$ and $A_{\bf 1}$ that can be jointly diagonalized with common eigenvectors $v_k$ and negative eigenvalues $\lambda_{k,0}$ and $\lambda_{k,1}$, respectively. Then
\begin{align*}
H^2({\cal N}_{cyl}(0,C_{\bf 0}),{\cal N}_{cyl}(0,C_{\bf 1}))
&\le T\sum_{k=1}^\infty\frac{(\lambda_{k,0}-\lambda_{k,1})^2}{\abs{\lambda_{k,0}}\wedge\abs{\lambda_{k,1}}}
\end{align*}
follows.
A well known example is given by the fractional Laplacian $A_{\bf 0}=-(-\Delta)^{m/2}$ and $A_{\bf 1}=A_{\bf 0}-\delta (-\Delta)^{m_1/2}$ for the Laplace operator $\Delta$ on a smooth bounded domain $\Lambda\subset\R^d$ with Dirichlet boundary conditions and $m\ge m_1\ge 0$, $\delta>0$. Then  $\abs{\lambda_{k,0}}\wedge\abs{\lambda_{k,1}}=\abs{\lambda_{k,0}}\thicksim k^{m/d}$ and $\abs{\lambda_{k,1}-\lambda_{k,0}}\thicksim \delta k^{m_1/d}$ hold by the Weyl asymptotics of Lemma \ref{lem:Weyl} below. Using that $A_{\bf 0},A_{\bf 1}$ have the same eigenfunctions, we obtain
\begin{align}\label{eq:SEEwithoutNoise:LaplaceSum}
	H^2({\cal N}_{cyl}(0,C_{\bf 0}),{\cal N}_{cyl}(0,C_{\bf 1}))\lesssim T\delta^2\sum_{k= 1}^\infty k^{(2m_1-m)/d}.
\end{align}
Consequently, the equivalence statement of Proposition \ref{PropHellinger} yields

\begin{proposition}
	For $A_{\bf 0}=-(-\Delta)^{m/2}$ and $A_{\bf 1}=A_{\bf 0}-\delta (-\Delta)^{m_1/2}$ with
	\begin{align}\label{eq:SEEwithoutNoise:Condition}
		m_1<(m-d)/2
	\end{align}
	and $\delta>0$ the laws ${\cal N}_{cyl}(0,C_{\bf 0})$ and ${\cal N}_{cyl}(0,C_{\bf 1})$ are equivalent.
\end{proposition}

For $A_{\bf 0}=\Delta$ ($m=2$), \eqref{eq:SEEwithoutNoise:Condition} holds if $d=1$ and $m_1<1/2$.  This is exactly the \citet[Cor. 2.3]{HuRo1995} condition for equivalence of the laws ${\cal N}_{cyl}(0,C_{\bf 0}),{\cal N}_{cyl}(0,C_{\bf 1})$. In the asymptotics $T\to\infty$ we can then estimate $\theta$ in $A_\theta=A_{\bf 0}-\theta (-\Delta)^{m_1/2}$ at best with rate $T^{-1/2}$, setting $\delta\thicksim T^{-1/2}$ in the lower bound of Theorem \ref{ThmLB}. This asymptotic result  slightly extends  a result in \citet{HuRo1995} for fixed $T$.

In the sequel we are mainly interested in the noisy case $\eps>0$ where equivalence of observation laws holds in far greater generality. Still, the Hellinger bound will become infinite for differential operators in large dimensions.

\subsection{Parameter in  the fractional Laplacian}

We study a generalisation of the classical stochastic heat equation $dX_t=\theta\Delta X_tdt+dW_t$ where the diffusivity $\theta>0$ is the target of estimation.
For each $n\in\N$ let the observations $(dY_t,t\in[0,T_n])$ be given by
\begin{equation}\label{EqExPar1}
dY_t= X_tdt+\eps_n dV_t\text{ with } dX_t=-\theta (-\Delta)^\rho X_tdt+ (\Id+(-\Delta)^\rho)^{-\beta} dW_t,
\end{equation}
with $X_0=0$, $T_n\ge 1$, $\eps_n\in[0,1]$, $\theta,\rho> 0$ and $\beta\ge 0$.
The Laplacian $\Delta$ on a $d$-dimensional domain $D$ is supposed  to have eigenvalues $-\lambda_k\thicksim k^{2/d}$ according to the Weyl asymptotics. By Lemma \ref{lem:Weyl} this holds for a Laplacian with Dirichlet, Neumann or periodic boundary conditions, or on a compact $d$-dimensional manifold without boundary. We are in the general parametric setting of \eqref{EqPartilde} with $M_n=0$, $\Lambda=-(-\Delta)^\rho$, $B_n=(\Id+(-\Delta)^\rho)^{-\beta}$ and $J_\theta=0$. We study the dependence of the estimation rate on the dimension $d$, the fractional index $\rho$, the dynamic noise correlation index $\beta$, the observational noise level $\eps_n$ and the observation time $T_n$.

We apply Corollary \ref{CorUpperBound} with
\[ {\cal I}_n(\theta)= T_n\trace\Big((-\Delta)^{2\rho} (\eps_n^4\bar \theta_n^4(-\Delta)^{4\rho}(\Id+(-\Delta)^\rho)^{4\beta}+\Id)^{-1} (\theta^{-1}(-\Delta)^{-\rho}\wedge T_n)\Big)
\]
For a fixed range of parameters $\theta\in[\underline{\theta},\bar\theta]$ with $\bar\theta>\underline\theta>0$ we obtain
\begin{align*}
{\cal I}_n(\theta) &\thicksim T_n\trace\Big(((-\Delta)^{\rho}\wedge T_n(-\Delta)^{2\rho})(\eps_n^4(-\Delta)^{4\rho(1+\beta)}+\Id)^{-1} \Big)\\
&\thicksim T_n \Big(\sum_{k:\lambda_k^{\rho+\beta\rho}\le\eps_n^{-1}}\lambda_k^\rho+ \eps_n^{-4}\sum_{k:\lambda_k^{\rho+\beta\rho}>\eps_n^{-1}}\lambda_k^{-3\rho-4\beta\rho}\Big)
\thicksim T_n \eps_n^{-(2\rho+d)/(2\rho(1+\beta))},
\end{align*}
whenever $d<(6+8\beta)\rho$. From Corollary \ref{CorUpperBound}, noting $\norm{\Lambda\abs{R_{\underline{\theta}}}_{T_n}^{-1}}\le \underline\theta^{-1}\thicksim 1$, we thus deduce the following minimax rate.

\begin{proposition}\label{PropParDiff}
Assume $d<(6+8\beta)\rho$ and the asymptotics
\[v_n:=T_n^{-1/2}\eps_n^{(2\rho+d)/(4\rho(1+\beta))}\to 0\text{ as }n\to\infty.\]
Then uniformly over $\theta\in[\underline\theta,\bar\theta]$  the estimator $\hat\theta_n$ from Definition \ref{Defhattheta} satisfies
\[ \hat\theta_n-\theta={\cal O}_{\PP_\theta}(v_n)
\]
and the rate $v_n$ is minimax optimal.
\end{proposition}

The rate for $\eps_n\to 0$ is becoming faster with the dimension $d$ and slower with the fractional index $\rho$ and the dynamic correlation index $\beta$. The fact that the rate slows down for larger $\beta$, i.e. with more correlation in the dynamic noise, reflects classical behaviour, in contrast to the inverse scaling of the dynamic noise from Example \ref{ex:ObsToDynamicNoise}.
In the fundamental case of the classical white-noise stochastic heat equation with $\rho=1$ and $B_n=\Id$ ($\beta=0$) the minimax rate becomes
\begin{equation}\label{EqExPar1Fund}
 v_n= T_n^{-1/2}\eps_n^{(2+d)/4} \text{ for }d\le 5.
 \end{equation}
 Note that  $\theta$ is not identifiable for fixed $T_n,\eps_n>0$ in dimension $d\le 5$ because the observation laws are equivalent. This is in stark contrast to the noiseless case $\eps_n=0$ where $\theta$ is identifiable  for fixed $T_n$ \citep{HuRo1995}. In case $d\ge 6$ (or generally $d\ge (6+8\beta)\rho$) one can check  by Remark \ref{RemIdent} that ${\cal I}_n(\theta)=\infty$ holds and $\theta$ is identifiable already non-asymptotically, which means that the observation laws are singular for different $\theta$.

\subsection{Parameter in  the transport term}

We consider the estimation of the first order coefficient $\theta$ in a second order differential operator $A_\theta$. This is exemplified by $A_\theta=\nu_n\Delta+\theta \partial_\xi$
with the directional derivative $\partial_\xi=\xi\cdot\nabla=\sum_{j=1}^d\xi_j\partial_{x_j}$ for $\xi\in\R^d\setminus\{0\}$. For the Laplacian $\Delta$ the operator $A_\theta=\nu_n\Delta+\theta\partial_\xi$ on the $d$-dimensional torus $[0,1]^d$ with periodic boundary conditions is normal  with eigenfunctions $e_\ell(x)=e^{2\pi i \scapro \ell x}$, $\ell\in\Z^d$,  and corresponding eigenvalues $\lambda_\ell=-(2\pi)^2\nu_n\abs{\ell}^2+2\pi i \scapro \xi\ell$. Thus, for each $n\in\N$ let the observations $(dY_t,t\in[0,T_n])$ be given by
\begin{equation}\label{EqExPar2}
dY_t= X_tdt+\eps_n dV_t\text{ with } dX_t=(\nu_n\Delta X_t+\theta \partial_\xi X_t)\,dt+ (\Id-\Delta)^{-\beta} dW_t
\end{equation}
and $X_0=0$, $T_n\ge 1$, $\eps_n\in[0,1]$, $\theta\in [\underline\theta,\bar\theta]$, $\nu_n\in(0,1]$, $\beta\ge 0$. This has the general form \eqref{EqPartilde} with $R_\theta=M_n=\nu_n\Delta$, $\Lambda=\partial_\xi$, $B_n=(\Id-\Delta)^{-\beta}$. Theorem \ref{ThmLBpar} yields the minimax lower bound rate
\begin{align*}
v_n &= T_n^{-1/2} \trace\Big(\abs{\Lambda}^2 (\eps_n^4\bar R_n^4 B_n^{-4}+\Id)^{-1}\abs{\bar R_n}_{T_n}^{-1}\Big)^{-1/2}\\
&\thicksim T_n^{-1/2}\Big(\sum_{\ell\in\Z^d}\Big(\scapro{\xi}{\ell}^2(\eps_n^4\nu_n^{4}\abs{\ell}^8(1+\abs{\ell}^2)^{4\beta}+1)^{-1}(\nu_n^{-1}\abs{\ell}^{-2}\wedge T_n)\Big)\Big)^{-1/2}\\
&\thicksim T_n^{-1/2}\Big(\sum_{k\ge 1}\Big((\eps_n^{-4}\nu_n^{-4}k^{-(8+8\beta)/d}\wedge 1)(\nu_n^{-1}\wedge T_nk^{2/d})\Big)\Big)^{-1/2}.
\end{align*}
The sum over $k\ge (\nu_n\eps_n)^{-d/(2+2\beta)}$ equals
\begin{align*}
\sum_{k\ge (\nu_n\eps_n)^{-d/(2+2\beta)}}\eps_n^{-4}\nu_n^{-4}k^{-(8+8\beta)/d}\nu_n^{-1}\thicksim \nu_n^{-(d+2+2\beta)/(2+2\beta)}\eps_n^{-d/(2+2\beta)}
\end{align*}
for $d<8+8\beta$.
The sum over $1\le k <(\nu_n\eps_n)^{-d/(2+2\beta)}$  is
\[ \sum_{1\le k<(\nu_n\eps_n)^{-d/(2+2\beta)}}(\nu_n^{-1}\wedge T_nk^{2/d})\lesssim \nu_n^{-(d+2+2\beta)/(2+2\beta)}\eps_n^{-d/(2+2\beta)}
\]
and thus bounded by the sum over larger $k$. Theorem \ref{ThmLBpar} therefore yields a simple lower bound result.

\begin{proposition}\label{PropParametricTransportLB}
Let $d<8+8\beta$ and
\[v_n=T_n^{-1/2}\nu_n^{(d+2+2\beta)/(4+4\beta)}\eps_n^{d/(4+4\beta)}.\]
Then there is a constant $c>0$ such that
\[ \inf_{\hat\theta_n}\sup_{\theta\in[\underline{\theta},\bar\theta]} \PP_\theta\big(v_n^{-1}\abs{\hat\theta_n-\theta}\ge c\big)\ge \tfrac{2-\sqrt{3}}{4} \]
holds for all $n\in\N$, where the infimum is taken over all estimators based on observing \eqref{EqExPar2}.
\end{proposition}

The lower bound suggests that $\theta$ can be consistently estimated under any of the three asymptotics $\eps_n\to 0$, $T_n\to\infty$, $\nu_n\to0$. The small diffusivity asymptotics $\nu_n\to 0$ quantifies how the SPDE \eqref{EqExPar2} approaches the usually singular first order SPDE $dX_t=\theta\partial_\xi X_tdt+(\Id-\Delta)^{-\beta}dW_t$, see also  \cite{GaRe2022}.

The upper bound is more involved because $J_\theta=\theta\partial_\xi$ is not dominated by $R_\theta=\nu_n\Delta$ when $\nu_n\to 0$. In Theorem \ref{ThmUpperBound} we have
\[{\cal I}_n(\theta)\thicksim T_n\sum_{k\ge 1}\Big(k^{2/d} \big(\eps_n^4(\nu_n k^{2/d}\vee T_n^{-1})(\nu_n k^{2/d}+k^{1/d})^3k^{8\beta/d}+1\big)^{-1} (\nu_n^{-1}k^{-2/d}\wedge T_n) \Big)
\]
Consider indices $k_n^\ast$ satisfying
\[ k^\ast_n\thicksim(\nu_n\eps_n)^{-d/(2+2\beta)} \wedge (\nu_n\eps_n^4)^{-d/(5+8\beta)}.
\]
Let us assume $\nu_n (k_n^\ast)^{2/d}\gtrsim T_n^{-1}$, that is $\nu_n^{3+8\beta}\gtrsim \eps_n^8T_n^{-5-8\beta}$. Then we can lower bound
\begin{align*}
{\cal I}_n(\theta) &\gtrsim T_n\nu_n^{-1}\sum_{k_n^\ast\le k\le 2k_n^\ast}\big(\eps_n^4\big(\nu_n^{4}k^{(8+8\beta)/d}\vee\nu_nk^{(5+8\beta)/d}\big) +1\big)^{-1} \\
&\thicksim T_n\nu_n^{-1}k_n^\ast
\thicksim T_n\nu_n^{-1} \big((\nu_n\eps_n)^{-d/(2+2\beta)} \wedge (\nu_n\eps_n^4)^{-d/(5+8\beta)}\big),
\end{align*}
where we used that the summands are of order $1$ for $k\thicksim k_n^\ast$.

It remains to check Condition \eqref{EqVarNBound} in Theorem \ref{ThmUpperBound}. A simple bound of its left-hand side is ${\cal I}_n(\theta)\norm{\abs{\Lambda}^2\abs{R_\theta}_{T_n}^{-2}}\lesssim T_n{\cal I}_n(\theta)$, which for $\eps_n\nu_n\to 0$ is $o({\cal I}_n(\theta)^2)$ as required. Otherwise, we have $\nu_n,\eps_n\thicksim 1$ and $T_n\to\infty$. In that case Condition \eqref{EqVarNBound} follows from $T_n=o(T_n^2)$. Disentangling the cases thus yields:

\begin{proposition}\label{PropParametricTransport}
Let $d<8+8\beta$ and $\nu_n^{3+8\beta}\gtrsim \eps_n^{8}T_n^{-5-8\beta}$. Under the asymptotics $v_n\to 0$ for
\begin{equation}\label{EqParRateTransp}
v_n:=\begin{cases} T_n^{-1/2}\nu_n^{(2+d+2\beta)/(4+4\beta)}\eps_n^{d/(4+4\beta)},&\text{ if }\nu_n\ge \eps_n^{1/(1+2\beta)},\\
T_n^{-1/2}\nu_n^{(5+d+8\beta)/(10+16\beta)}\eps_n^{2d/(5+8\beta)},&\text{ if }  \nu_n< \eps_n^{1/(1+2\beta)},
\end{cases}
\end{equation}
the estimator $\hat\theta_n$ from Definition \ref{Defhattheta} satisfies uniformly over $\theta\in[\underline\theta,\bar\theta]$
\[ \hat\theta_n-\theta={\cal O}_{\PP_\theta}(v_n).
\]
The rate $v_n$ is minimax optimal for $\nu_n\gtrsim \eps_n^{1/(1+2\beta)}$.
\end{proposition}

\begin{remark}\label{RemParTransport}
The results hold generally for real-valued $\underline{\theta},\bar\theta$ because $R_\theta$ is here independent of $\theta$, compare the proofs of Theorem \ref{ThmLBpar} and \ref{ThmUpperBound}.

For $\nu_n=o(\eps_n^{1/(1+2\beta)})$ the upper bound does not match the lower bound because the real part $R_\theta=\nu_n\Delta$ is too close to zero compared to the imaginary part $J_\theta=\theta\partial_\xi$. If even $\nu_n^{3+8\beta}=o(\eps_n^{8}T_n^{-5-8\beta})$ holds, then the upper bound rate stays the same as for $\nu_n^{3+8\beta}=\eps_n^{8}T_n^{-5-8\beta}$.
\end{remark}

In the fundamental case $B_n=\Id$ ($\beta=0$) and $\nu_n>0$ fixed, the minimax rate is
\begin{equation}\label{EqExPar2Fund}
 v_n= T_n^{-1/2}\eps_n^{d/4} \text{ for }d\le 7
 \end{equation}
for all $T_n,\eps_n$.
 Compared with the rate \eqref{EqExPar1Fund} for a parameter in front of the Laplacian, the rate is by a factor $\eps_n^{-1/2}$ slower.

Again, for $d\ge 8+8\beta$ we have ${\cal I}_n(\theta)=\infty$ and the parameter $\theta$ is non-asymptotically identifiable, checking Remark \ref{RemIdent}.

\subsection{Parameter in the source term}

Finally, we consider estimation of the coefficient in the zero order term of a second order differential operator $A_\theta$. We specify $A_\theta=\nu\Delta-\theta\Id$ which satisfies $A_\theta=R_\theta\preccurlyeq0$ for $\theta\ge 0$, $\nu>0$.
For each $n\in\N$ let the observations $(dY_t,t\in[0,T_n])$ be given by
\begin{equation}\label{EqExPar3}
dY_t= X_tdt+\eps_n dV_t\text{ with } dX_t=(\nu_n\Delta X_t-\theta X_t)\,dt+ (\Id-\nu_n\Delta)^{-\beta} dW_t
\end{equation}
and $X_0=0$, $T_n\ge 1$, $\eps_n\in[0,1]$, $\nu_n\in (0,1]$, $\beta\ge 0$. The Laplacian $\Delta$ on a $d$-dimensional domain $\Lambda$ is supposed  to have eigenvalues $-\lambda_k\thicksim k^{2/d}$ according to the Weyl asymptotics. By Lemma \ref{lem:Weyl} this holds for a Laplacian with Dirichlet, Neumann or periodic boundary conditions, or on a compact $d$-dimensional manifold without boundary.  We apply Corollary \ref{CorUpperBound}
with $M_n=\nu_n\Delta$, $\Lambda=\Id$, $B_n=(\Id-\nu_n\Delta)^{-\beta}$ and a fixed parameter range $\theta\in[\underline\theta,\bar\theta]$, $\bar\theta>\underline\theta>0$, so that

\begin{align*}
{\cal I}_n(\theta) &= T_n\trace\Big( \big(\eps_n^4(\theta\Id-\nu_n\Delta)^4(\Id-\nu_n\Delta)^{4\beta}+\Id\big)^{-1} \abs{\theta\Id-\nu_n\Delta}_{T_n}^{-1}\Big)
\\
 &\thicksim T_n\trace\Big( \big(\eps_n^4(\Id-\nu_n\Delta)^{4(1+\beta)}+\Id\big)^{-1} (\Id-\nu_n\Delta)^{-1}\Big).
\end{align*}
Given the Weyl asymptotics $-\lambda_k\thicksim k^{2/d}$,  we calculate
\begin{align*}
{\cal I}_n(\theta)
 &\thicksim T_n\Big(\sum_{k:\nu_n\abs{\lambda_k}\le 1} 1+ \sum_{k:1<\nu_n\abs{\lambda_k}\le\eps_n^{-1/(1+\beta)}}\abs{\nu_n\lambda_k}^{-1}\\
 &\qquad\qquad
 +\eps_n^{-4}\nu_n^{-5-4\beta}
\sum_{k:\nu_n\abs{\lambda_k}>\eps_n^{-1/(1+\beta)}}\abs{\lambda_k}^{-5-4\beta}\Big)
\\
 &\thicksim T_n\Big(\nu_n^{-d/2}+\nu_n^{-1} \sum_{\nu_n^{-d/2}<k\le\nu_n^{-d/2}\eps_n^{-d/(2+2\beta)}}k^{-2/d} \\
 &\qquad\qquad +\eps_n^{-4}\nu_n^{-5-4\beta}\sum_{k>\nu_n^{-d/2}\eps_n^{-d/(2+2\beta)}}k^{-(10+8\beta)/d}\Big)\\
 &\thicksim T_n\nu_n^{-d/2}\big(1+\eps_n^{-(d-2)/(2+2\beta)} \big),
\end{align*}
where we assumed $d<10+8\beta$ for summability and $d\not=2$ for bounding the second sum by the first and third sum. In the case $d=2$ the second sum dominates and
\[ {\cal I}_n(\theta)\thicksim T_n\nu_n^{-1}\log(e\eps_n^{-1})\]
holds. Thus Corollary \ref{CorUpperBound}, noting  $\norm{\Lambda\abs{R_{\underline{\theta}}}_{T_n}^{-1}}\lesssim 1$, gives the following result.

\begin{proposition}\label{PropParametricSource}
For $d<10+8\beta$ and under the asymptotics $v_n\to 0$ for
\begin{equation}\label{EqParRateSource}
v_n:=\begin{cases} T_n^{-1/2}\nu_n^{1/4},&\text{ if }d=1,\\
T_n^{-1/2}\nu_n^{1/2}(\log(e\eps_n^{-1}))^{-1/2},&\text{ if } d=2,\\
T_n^{-1/2}\nu_n^{d/4}\eps_n^{(d-2)/(4+4\beta)},&\text{ if } d\ge 3,
\end{cases}
\end{equation}
the estimator $\hat\theta_n$ from Definition \ref{Defhattheta} satisfies uniformly over $\theta\in[\underline\theta,\bar\theta]$
\[ \hat\theta_n-\theta={\cal O}_{\PP_\theta}(v_n)
\]
and the rate $v_n$ is minimax optimal.
\end{proposition}

\begin{remark}
If the noise covariance operator is $B_n=(\Id-\Delta)^{-\beta}$, that is not depending on $\nu_n$, then similar calculations for $\eps_n\lesssim \nu_n^\beta$ yield the same rate for $d\in\{1,2\}$, but $v_n=T_n^{-1/2}\nu_n^{(d+2\beta)/(4+4\beta)}\eps_n^{(d-2)/(4+4\beta)}$ for $3\le d<10+8\beta$.
This is slower in $\nu_n$ than before which is reasonable in view of Example \ref{ex:ObsToDynamicNoise} because the noise covariance operator is smaller in the high frequencies.
\end{remark}

This means that in dimension $d=1$ the minimax rate is independent of $\beta$ and $\eps_n$. It matches the upper bound derived by \citet{GaRe2022} in the vanishing diffusivity regime $\nu_n\to 0$ without noise. In the fundamental case $\beta=0$ and $\nu_n=1$ the rate is
\begin{equation}\label{EqExPar3Fund}
 v_n= \begin{cases} T_n^{-1/2},&\text{ if }d=1,\\ (T_n\log(e\eps_n^{-1}))^{-1/2},&\text{ if }d=2,\\T_n^{-1/2}\eps_n^{(d-2)/4},&\text{ if } 3\le d\le 9.\end{cases}
 \end{equation}

The fact that vanishing observation noise $\eps_n\rightarrow 0$ leads to consistent estimation of the reaction parameter in $d\ge 2$ agrees with the results from \cite{HuRo1995}, where in the noiseless setting $\eps_n=0$ the reaction parameter is identified in these dimensions. Generally, for $d\ge 10+8\beta$ we have ${\cal I}_n(\theta)=\infty$ and $\theta$ is non-asymptotically identifiable by Remark \ref{RemIdent}.

\section{Fundamental nonparametric examples}\label{SecNonparEx}

\subsection{Space-dependent diffusivity}

Passing to nonparametric problems, we now consider a space-dependent coefficient in the leading order of the second-order differential operator, more specifically the weighted Laplacian $\Delta_\theta=\nabla\cdot(\theta(x)\nabla)$ on $H=L^2(\R^d)$ with space-dependent diffusivity $\theta:\R^d\to\R^+$ in a nonparametric regularity class.
For each $n\in\N$ we consider observations $(dY_t,t\in[0,T_n])$ given by
\begin{equation}\label{EqExNonpar1}
dY_t= (\Id-\Delta)^{-\beta}X_tdt+\eps_n dV_t\text{ with } dX_t=\Delta_\theta X_t\,dt+ dW_t,
\end{equation}
where
$X_0=0$,  $T_n\ge 1$, $\beta\in[0,1/2]$, $\eps_n\in[0,1]$, and
 $\theta(\cdot)$ belonging to the class
\begin{align}
	\Theta_{dif}(\alpha,R):=\Big\{\theta\in C^{4\vee\alpha}(\R^d)\,\Big|\,\inf_{x\in\R^d}\theta(x)\ge \tfrac12;\; \norm{\theta}_{C^\alpha}\le R\Big\},\; \alpha>0, \,R>1.
\end{align}
\begin{remark}
For the domain of the weighted Laplacians we need  $\dom(\Delta_\theta)={\cal H}^2(\R^d)$.
Taking into account that by Theorem \ref{ThmNonpar1} below we
only consider $d<6+8\beta\le 10$,
an easy sufficient condition is $\theta\in C^4(\R^d)$ by \citet[Section 2.8.2]{Triebel2010}.
Notice that for different $\theta$, the operators $\Delta_\theta$ do not commute.
The bound $\inf_{x\in\R^d}\theta(x)\ge 1/2$ ensures that $(u, v)\mapsto\scapro{u}{(-\Delta_\theta) v}=\scapro{\nabla u}{\theta\nabla v}$ defines a positive definite bilinear form on ${\cal H}^1(\R^d)$.
\end{remark}

In order to apply Theorem \ref{ThmMain}, we consider $A_{\bf 0}=\Delta$ and the alternative $A_{\bf 1}=\Delta_\theta$ with $\theta(x)=1+L(x/h)$ for some test function $L:\R^d\to\R$  and a bandwidth $h\in(0,1)$, to be chosen later.
We have $B=(\Id-\Delta)^{-\beta}$ and set $\bar B_{\bf 0}=B$. $B$ does not commute with $A_{\bf 1}$, but $-\Delta_\theta \preccurlyeq \norm{\theta}_\infty(-\Delta)$ implies
\[ (\Id-\norm{\theta}_\infty^{-1}\Delta_\theta)^{-2\beta}\succcurlyeq (\Id-\Delta)^{-2\beta}=B^\ast B\text{ for }\beta\in[0,1/2],\]
using operator monotonicity of $t\mapsto -t^{-r}$ on $\R^+$ for $r\in(0,1]$ \citep{bhatia2013}. For $\bar B_{\bf 1}:=(\Id-\norm{\theta}_\infty^{-1}\Delta_\theta)^{-\beta}$ this yields \eqref{EqBtheta}.
By \citet[Thm. VI.5.22]{EN2000} $\dom(A_{\bf 1})=\dom(A_{\bf 0})={\cal H}^2(\R^d)$ holds, as soon as $L\in C^1(\R^d)$.
The minimax lower bound later can still be established over H\"older classes with $\alpha<1$ by using smooth $L$ whose bounds on the first derivatives are finite, but grow with the asymptotics.

To deal with the non-commutativity of $A_{\bf 0}$ and $A_{\bf 1}$, we establish for certain functions $f:\R^+\to\R$ that $f(-\Delta_\theta)\preccurlyeq C_{f,\theta} f(-\Delta)$ with a well quantified constant $C_{f,\theta}$.
Let us remark that most monotone functions $g$ are not operator monotone so that operators $T_1,T_2$ with $T_1\preccurlyeq T_2$ do not necessarily satisfy $g(T_1)\preccurlyeq g(T_2)$. In particular, $g:\R^+\to\R$, $g(\lambda)=-(\eps^2\lambda^{2+2\beta}+1)^{-1}$ is not an operator monotone function because $\lambda\mapsto\lambda^{2+2\beta}$ is not operator monotone, see \citet[Chapter V]{bhatia2013} for this and more results on operator monotonicity for matrices which directly extend to linear operators.

Based on perturbation ideas from \citet[Prop. VI.5.24]{EN2000}, however, we are able to establish $(\eps^2(-\Delta_\theta)^{2+2\beta}+\Id)^{-1}\preccurlyeq C_{d, \beta} (\eps^2(-\Delta)^{2+2\beta}+\Id)^{-1}$ for some precise constant $C_{d, \beta}$, provided a suitable smoothness norm of $\theta$ is bounded by $\eps^{-1}$. 

\begin{proposition}\label{PropOrderDeltatheta}
Consider the operators $\Delta_\theta$ and $\Delta$ on ${\cal H}^2(\R^d)\subset L^2(\R^d)$ with $\theta(x)=1+L(x/h)$, $L\in C^7(\R^d)$,
$\supp L\subset[-1/2,1/2]^d$ and $\norm{L}_\infty\le 1/2$.
There are positive quantities $C_{d, \beta}^{(i)}$, $i=1,\ldots,4$, only depending on $d\ge 1$ and $\beta\in[0,1/2]$,
such that:
\begin{enumerate}
\item  if  $C_{d, \beta}^{(1)}h^{-2}(\norm{\nabla L}_\infty^{4}+\norm{\nabla^2 L}_\infty^{4/3})\le\eps^{-1/(1+\beta)}$, then
\begin{equation}\label{EqDeltathetaorder}
 \big(\eps^2(-\Delta_\theta)^{2+2\beta}+\Id\big)^{-1}\preccurlyeq C_{d, \beta}^{(2)}\big( \eps^2(-\Delta)^{2+2\beta} +\Id\big)^{-1};
\end{equation}
\item if
$C_{d, \beta}^{(3)}h^{-2}(\norm{\nabla L}_\infty^{4} +\norm{\nabla L}_\infty^{10/3}
 + \norm{\nabla^2 L}_\infty^{4/3}+  \norm{\nabla\Delta L}_\infty^{4/5})\le\eps^{-1/(1+\beta)}$,
then
\begin{equation}\label{EqDeltathetaorder2}
\abs{\Delta_\theta}_{T^{-1}}^{-1}\big(\eps^2(-\Delta_\theta)^{2+2\beta}+\Id\big)^{-1}\preccurlyeq C_{d, \beta}^{(4)} \abs{\Delta}_{T^{-1}}^{-1}\big(\eps^2(-\Delta)^{2+2\beta}+\Id\big)^{-1}.
\end{equation}
\end{enumerate}
\end{proposition}

\begin{proof}
	See Section \ref{SecProofNonpar}.
\end{proof}

\begin{remark}\label{RemOpMon}
The condition $L\in C^7(\R^d)$ ensures $\dom(\Delta_\theta^4)={\cal H}^8(\R^d)$.
Later we shall consider kernels $L$ of the form $L=h_n^\alpha K$ with $h_n\downarrow 0$ and $K$ fixed so that the norms in terms of $L$ matter for the asymptotics.

In the case $\beta=0$ Proposition \ref{PropOrderDeltatheta}(a) holds already under the condition $C_d^{(1)}h^{-2}\norm{\nabla L}_\infty^4\le\eps^{-1}$ and Proposition \ref{PropOrderDeltatheta}(b) under $C_d^{(3)}h^{-2}(\norm{\nabla L}_\infty^4+\norm{\nabla^2 L}_\infty^{4/3}) \le\eps^{-1}$.
The reason is that in the proof we only need to consider $(-\Delta_\theta)^m$ for $m\le 3$ instead of $m\le 4$.
\end{remark}

\begin{theorem}\label{ThmNonpar1}
For each $n$ consider the observations  given by \eqref{EqExNonpar1}
with a space-varying diffusivity $\theta\in\Theta_{dif}(\alpha,R)$
for $\alpha>0$,\,$R>1$.
Assume
\[v_n:= \begin{cases} (T_n\eps_n^{-(d+2)/(2+2\beta)})^{-\alpha/(2\alpha+d)},&\text{if } T_n\le\eps_n^{(1-\alpha)/(1+\beta)},\\
(T_n\eps_n^{-(5+4\beta)/(2+2\beta)})^{-\alpha/(2\alpha+3+4\beta)},&\text{if } T_n\ge\eps_n^{(1-\alpha)/(1+\beta)}
\end{cases}
\quad\longrightarrow 0
\]
as  $n\to\infty$. If $1\le d <6+8\beta$ and
\begin{equation}\label{Eqalphamin}
	\alpha > \limsup_{n\rightarrow\infty}\tfrac{(5+5\beta)\log(T_n) + 5\log(\eps_n^{-1})}{(2+2\beta)\log(T_n) + (10+4\beta)\log(\eps_n^{-1})}
\end{equation}
is satisfied, then for a constant $c>0$
\[ \liminf_{n\to\infty}\inf_{\hat\theta_n}\sup_{\theta\in \Theta_{dif}(\alpha,R)}\PP_\theta\big(v_n^{-1}\abs{\hat\theta_n(0)-\theta(0)}\ge c\big)>0
\]
holds, where the infimum is taken over all estimators $\hat\theta_n$ of $\theta$ based on \eqref{EqExNonpar1}.
\end{theorem}

\begin{remark}\
\begin{enumerate}

\item The same lower bound holds for $\abs{\hat\theta_n(x_0)-\theta(x_0)}$ at any $x_0\in\R^d$, the concrete choice $x_0=0$ just simplifies notation. Working on the unbounded domain $\R^d$ allows us to use Fourier transforms which facilitates the analysis considerably. It is intuitive, but not rigorously established  that for our pointwise estimation risk the asymptotic results transfer to smooth bounded domains.

\item We can understand the rate in case $T_n\le\eps_n^{(1-\alpha)/(1+\beta)}$   in terms of the parametric rate $T_n^{-1/2}\eps_n^{(d+2)/(4+4\beta)}$ from Proposition \ref{PropParDiff}.
For fixed $T_n>0$ and $\beta=0$ the rates simplify to
\begin{equation}\label{EqExNonpar1Fund}
 v_n= \begin{cases}\eps_n^{(d+2)\alpha/(4\alpha+2d)}&\text{, if } \alpha\ge 1,\\ \eps_n^{5\alpha/(4\alpha+6)}&\text{, if }  \alpha\in(1/2,1].
\end{cases}
\end{equation}
In view of Remark \ref{RemOpMon} the second case holds even for $\alpha\in(3/10,1]$. With considerable efforts \cite{GPMR2024} have obtained a corresponding upper bound for the case $\alpha\ge 1$.

\item
In dimension $d<3+4\beta$, the rate slows down in the regime $T_n\ll\eps_n^{(1-\alpha)/(1+\beta)}$.
This {\it ellbow effect} stems from the bound \eqref{EqHellNonpar} below in the Sobolev space ${\cal H}^{(d-3-4\beta)/2}(\R^d)$ of negative order for relatively rough $\theta$ with bandwidth $h_n\lesssim\eps_n^{1/(2+2\beta)}$.  In  the critical case $\alpha<1$ for fixed $T_n$ a standard preaveraging estimator does no longer work \citep{GPMR2024}, which might give an upper bound perspective on this effect.

\item The additional condition \eqref{Eqalphamin} on $\alpha$ is technically required in order to profit from the operator monotonicity in Proposition \ref{PropOrderDeltatheta}. It requires necessarily $\alpha>5/(10+4\beta)$ (which suffices also for fixed $T_n$) and is always satisfied if
$\alpha\ge 5/2$.

\end{enumerate}

\end{remark}

\begin{proof}
For simplicity we drop the index $n$ at $\eps_n,T_n$.
We  transfer the problem into the spectral domain by the Fourier transform ${\cal F}f(u)=\int e^{i\scapro{u}{x}}f(x)\,dx$. In particular, we have ${\cal F}(\Delta_\theta g)(u)=(2\pi)^{-d}M[-iu^\top]C[{\cal F}\theta(u)]M[-iu]{\cal F}g(u)$ for $g\in {\cal H}^2(\R^d)$, where $M[f]g=fg$ and $C[f]g=f*g$ are multiplication and convolution operators.
Consider $\theta\in C^\infty(\R^d)$ fulfilling the conditions of Proposition \ref{PropOrderDeltatheta}(a,b) and the derived constant $C_{d, \beta}^{(5)}:=\tfrac12(C_{d, \beta}^{(2)} + C_{d, \beta}^{(4)})$.
Note that using $x^2(1+x)^{2\beta}\ge x^{2+2\beta}$, $x\ge 0$, and functional calculus, we have
\begin{align*}
	(\eps^2R_{\bf 0}^2\bar B_{\bf 0}^{-2}+\Id)^{-1}&\preccurlyeq (\eps^2\Delta^{2+2\beta}+\Id)^{-1},\\
	(\eps^2R_{\bf 1}^2\bar B_{\bf 1}^{-2}+\Id)^{-1}&\preccurlyeq (\eps^2\norm{\theta}_\infty^{-2\beta}\Delta_\theta^{2+2\beta}+\Id)^{-1}\preccurlyeq \norm{\theta}_\infty^{2\beta}(\eps^2\Delta_\theta^{2+2\beta}+\Id)^{-1}.
\end{align*}
Then by using these estimates, Proposition \ref{PropOrderDeltatheta}, Lemma \ref{LemTraceBound} (a) below and  the isometry of $(2\pi)^{-d/2}\cal F$, we obtain from Theorem \ref{ThmMain} for the case $R_\theta\preccurlyeq 0$,
writing $\delta=\theta-1$
\begin{align*}
&\norm{\theta}_\infty^{-2\beta}H^2({\cal N}_{cyl}(0,Q_{\bf 0}),{\cal N}_{cyl}(0,Q_{\bf 1}))\\
&\le \tfrac {T}4 \norm{(\eps^2\Delta^{2+2\beta}+\Id)^{-1/2}(\Delta_\theta-\Delta) \abs{\Delta_\theta}_{T^{-1}}^{-1/2}(\eps^2\Delta_\theta^{2+2\beta}+\Id)^{-1/2}}_{HS}^2\\
&\quad +\tfrac {T}4 \norm{(\eps^2\Delta_\theta^{2+2\beta}+\Id)^{-1/2}(\Delta_\theta-\Delta) \abs{\Delta}_{T^{-1}}^{-1/2}(\eps^2\Delta^{2+2\beta}+\Id)^{-1/2}}_{HS}^2\\
&\le C_{d, \beta}^{(5)} T \norm{(\eps^2\Delta^{2+2\beta}+\Id)^{-1/2}(\Delta_\theta-\Delta)\abs{\Delta}_{T^{-1}}^{-1/2}(\eps^2\Delta^{2+2\beta}+\Id)^{-1/2}}_{HS}^2\\
&\le C_{d, \beta}^{(5)} T\Big\|M[(\eps^2 \abs{u}^{4+4\beta}+1)^{-1/2}(-iu^\top)]C[{\cal F}\delta]\\
 &\qquad\qquad M[(-iu)(\abs{u}^{-1}\wedge T^{1/2})(\eps^2\abs{u}^{4+4\beta}+1)^{-1/2}]\Big\|_{HS}^2.
\end{align*}
The operator $K$ inside the Hilbert-Schmidt norm is given as an integral operator of the form $Kf(u)=\int_{\R^d} k(u, v)f(v)\,dv$ with real-valued kernel
\[k(u,v)=(\eps^2 \abs{u}^{4+4\beta}+1)^{-1/2}\scapro{u}{v}{\cal F}\delta(u-v)(\abs{v}^{-1}\wedge T^{1/2}) (\eps^2\abs{v}^{4+4\beta}+1)^{-1/2}.\]
Its Hilbert-Schmidt norm is therefore given by $\norm{k}_{L^2(\R^d\times\R^d)}$ and we obtain with the substitution $w=\eps^{1/(2+2\beta)}u$
\begin{align*}
&(\norm{\theta}_\infty^{2\beta}C_{d, \beta}^{(5)})^{-1} H^2({\cal N}_{cyl}(0,Q_{\bf 0}),{\cal N}_{cyl}(0,Q_{\bf 1}))\\
&\le T\int_{\R^d}\int_{\R^d} (\eps^2 \abs{u}^{4+4\beta}+1)^{-1}\scapro{u}{v}^2(\abs{v}^{-2}\wedge T) \abs{{\cal F}\delta(u-v)}^2(\eps^2\abs{v}^{4+4\beta}+1)^{-1}dudv\\
&\le  T\int_{\R^d} \abs{{\cal F}\delta(r)}^2 \Big(\int_{\R^d} (\eps^2 \abs{u}^{4+4\beta}+1)^{-1}\abs{u}^2 (\eps^2\abs{u-r}^{4+4\beta}+1)^{-1}du\Big)\,dr\\
&\le T\eps^{-\frac{d+2}{2+2\beta}}\int_{\R^d} \abs{{\cal F}\delta(r)}^2 \Big(\int_{\R^d} (\abs{w}^{4+4\beta}+1)^{-1}\abs{w}^2 (\abs{\abs{w}-\abs{\eps^{\frac1{2+2\beta}}r}}^{4+4\beta}+1)^{-1}dw\Big)\,dr\\
&\le C_d^{(6)}T\eps^{-\frac{d+2}{2+2\beta}}\int_{\R^d} \abs{{\cal F}\delta(r)}^2 \Big(\int_0^\infty (z^{d-3-4\beta}\wedge z^{d+1}) (\abs{z-\eps^{\frac1{2+2\beta}}\abs{r}}^{-4-4\beta}\wedge 1)\,dz\Big)\,dr,
\end{align*}
where $C_d^{(6)}$ is the surface of the $d$-dimensional unit sphere.
By Lemma \ref{lem:aux:integrals}(b) below,
for $d< 6+8\beta$ the inner integral is finite and bounded in order by $(1+\eps^{1/(2+2\beta)}\abs{r})^{d-3-4\beta}$.
Hence, we obtain
\begin{equation}\label{EqHellNonpar}
H^2({\cal N}_{cyl}(0,Q_{\bf 0}),{\cal N}_{cyl}(0,Q_{\bf 1}))\lesssim  T\eps^{-(d+2)/(2+2\beta)}\int_{\R^d} \abs{{\cal F}\delta(r)}^2 (1+\eps^{1/(2+2\beta)}\abs{r})^{d-3-4\beta} dr.
\end{equation}
Note that the weight function in the integral is fundamentally different for $d<3+4\beta$ and $d>3+4\beta$. The integral corresponds to a squared $L^2$-Sobolev norm of $\delta=\theta-1$ with order $(d-3-4\beta)/2$ and additional $\eps$-dependent weighting.

We choose $\delta(x)=h^\alpha K(x/h)$
for some bandwidth $h\in(0,1)$ and $K\in C^\infty(\R^d)$ with support in $[-1/2,1/2]^d$, $\norm{K}_\infty\le 1/2$ and $\norm{K}_{C^\alpha}\le R-1$.
In particular, $K$, $xK$ and all derivatives of $K$ up to order $d$ are in $L^1(\R^d)$.
We assume $K(0)\not=0$ as well as $\int_{\R^d}x^mK(x)\,dx=0$, $m\in\{0,1\}$.
Then $\abs{{\cal F}K(u)}\lesssim \abs{u}^2\wedge \abs{u}^{-d}$, and
\begin{align*}
&\int_{\R^d} \abs{{\cal F}\delta(r)}^2 (1+\eps^{1/(2+2\beta)}\abs{r})^{d-3-4\beta} dr\\
 &= h^{2\alpha+2d}\int_{\R^d} \abs{{\cal F}K(hr)}^2 (1+\eps^{1/(2+2\beta)}\abs{r})^{d-3-4\beta} dr\\
&= h^{2\alpha+d}\int_{\R^d} \abs{{\cal F}K(u)}^2 (1+\eps^{1/(2+2\beta)}h^{-1}\abs{u})^{d-3-4\beta} du\\
&\lesssim h^{2\alpha+d}\int_{\R^d} (\abs{u}^4\wedge \abs{u}^{-2d}) (1+\eps^{1/(2+2\beta)}h^{-1}\abs{u})^{d-3-4\beta} du\\
&\lesssim h^{2\alpha+d}\int_0^\infty  (z^{d+3}\wedge z^{-d-1})(1+\eps^{1/(2+2\beta)}h^{-1}z)^{d-3-4\beta} dz.
\end{align*}
Lemma \ref{lem:aux:integrals}(a) below together with $d+3+d-3-4\beta\ge 0$ (recall $\beta\le 1/2$) as well as $-d-1+d-3-4\beta<-1$ shows that the integral is finite and of order $(1+\eps^{1/(2+2\beta)}h^{-1})^{d-3-4\beta}$. Inserting into \eqref{EqHellNonpar}, we arrive at
\begin{align*}
H^2({\cal N}_{cyl}(0,Q_{\bf 0}),{\cal N}_{cyl}(0,Q_{\bf 1}))
 &\lesssim T\eps^{-(d+2)/(2+2\beta)}h^{2\alpha+d}(1+\eps^{1/(2+2\beta)}h^{-1})^{d-3-4\beta}.
\end{align*}
For $h\ge\eps^{1/(2+2\beta)}$ this gives the order $T h^{2\alpha+d}\eps^{-(d+2)/(2+2\beta)}$ and for $h\le \eps^{1/(2+2\beta)}$ the order $Th^{2\alpha+3+4\beta}\eps^{-(5+4\beta)/(2+2\beta)}$. We choose
\begin{equation}\label{Eqh1} h\thicksim \begin{cases} (T\eps^{-(d+2)/(2+2\beta)})^{-1/(2\alpha+d)},&\text{if } T\le\eps^{(1-\alpha)/(1+\beta)},\\
(T\eps^{-(5+4\beta)/(2+2\beta)})^{-1/(2\alpha+3+4\beta)},&\text{if } T\ge\eps^{(1-\alpha)/(1+\beta)}.
\end{cases}
\end{equation}
With Theorem \ref{ThmLB} this yields the claim in view of ${\bf 1},\theta\in \Theta_{dif}(\alpha,R)$
by construction (note $\norm{\theta}_{C^\alpha}\le 1+\norm{\delta}_{C^\alpha}\le 1+\norm{K}_{C^\alpha}\le R$) and $\abs{\theta(0)-1}=h^{\alpha}\abs{K(0)}\thicksim h^\alpha$.

It remains to verify the conditions from Proposition \ref{PropOrderDeltatheta}(a,b). In order to do so, we note for $L=h^\alpha K$ that $h^{4\alpha/5-2}\lesssim \eps^{-1/(1+\beta)}$
with a sufficiently small constant implies both conditions. In the first case of \eqref{Eqh1} this holds due to $h\gtrsim\eps^{1/(2+2\beta)}$. In the second case of \eqref{Eqh1} a short calculation shows that it suffices to consider the case $\alpha<5/2$, and in this case the condition is equivalent to
$T\lesssim \eps^{-((10+4\beta)\alpha+5)/((1+\beta)(5-2\alpha))}$,
which is equivalent to \eqref{Eqalphamin} by Lemma \ref{lem:aux:alpha} below.
\end{proof}

\subsection{Space-dependent transport}

For each $n\in\N$ consider the observations $(dY_t,t\in[0,T_n])$  given by
\begin{equation}\label{EqExNonpar2}
dY_t= X_tdt+\eps_n dV_t\text{ with } dX_t=(\nu_n\Delta X_t+\nabla\cdot(\theta(x)X_t))\,dt+ dW_t
\end{equation}
and $X_0=0$, $T_n\ge 1$, $\eps_n\in [0,1]$ and $\nu_n\in(0,1]$. The  Laplace operator $\Delta$ is considered on ${\cal H}^2(\R^d)\subset L^2(\R^d)=H$ and $\theta:\R^d\to \R^d$ is a smooth divergence-free vector field.
More specifically, we assume that $\theta$ is in
\begin{align*}
	\Theta_{tra}(\alpha,R):=\Big\{\theta\in C^{3\vee\alpha}(\R^d;\R^d)\,\Big|\, \norm{\theta}_{C^\alpha}\le R\text{ and } \nabla\cdot\theta=0\Big\},\; \alpha>0,\,R>0.
\end{align*}
\begin{remark}
	We require that $U\mapsto\nabla\cdot(\theta U)$ defines a bounded operator ${\cal H}^2(\R^d)\rightarrow L^2(\R^d)$.
	From Theorem \ref{ThmNonpar2} below we see that we are only interested in the case $d\le 7$, in which case $\theta\in C^3(\R^d;\R^d)$ is a simple sufficient condition by \citet[Section 2.8.2]{Triebel2010}.
	In particular, we have $\dom(A_\theta)={\cal H}^2(\R^d)$ and $\nabla\cdot(\theta U)=\theta\cdot \nabla U$ on ${\cal H}^2(\R^d)$, which implies $A_{\theta}^*f=\nu_n\Delta f-\nabla\cdot(\theta f)=A_{-\theta}f$ such that $R_\theta = \nu_n\Delta$.
\end{remark}
In the notation of Theorem \ref{ThmMain}
consider $\bar B_{\bf 0}:=\bar B_{\bf 1}:=B=\Id$, $A_{\bf 0}=\nu_n\Delta$ and  $A_{\bf 1}=\nu_n\Delta+\theta(x)\cdot\nabla$.
We only consider the case $B=\Id$ (white noise), but with more technical efforts we can handle for instance $B=(\Id-\Delta)^{-\beta}$, $\beta\in[0,1/2]$, as in the preceding section.

\begin{theorem}\label{ThmNonpar2}
For each $n\in\N$ consider the observations  given by \eqref{EqExNonpar2} with the space-varying transport coefficient $\theta\in\Theta_{tra}(\alpha,R)$ for some $\alpha>0$, $R>0$.
Assume the asymptotics
\[v_n:= \begin{cases} (T_n^{-1}\eps_n^{d/2}\nu_n^{(d+2)/2})^{\alpha/(2\alpha+d)} &\text{, if }T_n\le\eps_n^{-\alpha}\nu_n^{1-\alpha},\\ T_n^{-\alpha/(2\alpha+5)} \eps_n^{5\alpha/(4\alpha+10)}\nu_n^{7\alpha/(4\alpha+10)}&\text{, if }T_n\ge\eps_n^{-\alpha}\nu_n^{1-\alpha}, \end{cases} \quad\rightarrow 0
\]
as $n\to\infty$. For $1\le d\le 7$ there is a constant $c>0$ such that
\[ \liminf_{n\to\infty}\inf_{\hat\theta_n}\sup_{\theta\in \Theta_{tra}(\alpha,R)}\PP_\theta\big(v_n^{-1}\abs{\hat\theta_n(0)-\theta(0)}\ge c\big)>0,
\]
where the infimum is taken over all estimators $\hat\theta_n$ of $\theta$ based on \eqref{EqExNonpar2}.
\end{theorem}

\begin{remark}
In the first case $T_n\le\eps_n^{-\alpha}\nu_n^{1-\alpha}$ we obtain the nonparametric analogue of the parametric rate from Proposition \ref{PropParametricTransportLB}.
In general, we observe again an ellbow effect whenever $h_n=v_n^{1/\alpha}$ becomes smaller than $\eps_n^{1/2}\nu_n^{1/2}$ in the sense that the rate slows down for $d\le 4$ and speeds up for $d\ge 6$. There is no further restriction on $\alpha>0$ as in Theorem \ref{ThmNonpar1} because the real part $R_\theta$ does not depend on $\theta$ for divergence-free vector fields $\theta$.
\end{remark}

\begin{proof}
Throughout the proof we drop again the index $n$.
Passing into the Fourier domain as for Theorem \ref{ThmNonpar1} and applying multiplication and convolution operations coefficientwise, we find by Theorem \ref{ThmMain}:
\begin{align*}
&H^2({\cal N}_{cyl}(0,Q_{\bf 0}),{\cal N}_{cyl}(0,Q_{\bf 1}))\\
&\le \tfrac T2\norm{(\eps^2\nu^2\Delta^2+\Id)^{-1/2}(M[\theta]\cdot \nabla)\abs{\nu\Delta}_{T^{-1}}^{-1/2}(\eps^2\nu^2\Delta^2+\Id)^{-1/2}}_{HS}^2\\
&\le T\norm{M[(\eps^2 \nu^2\abs{u}^4+1)^{-\frac12}]C[{\cal F}\theta]\cdot M[(-iu)(\nu^{-\frac12}\abs{u}^{-1}\wedge T^{-1/2})(\eps^2\nu^2\abs{u}^4+1)^{-\frac12}]}_{HS}^2,
\end{align*}
where we write $A\cdot B:=\sum_{i=1}^dA_iB_i$ for operators $A_i,B_i$, $1\le i\le d$,
thus
\begin{align*}
&H^2({\cal N}_{cyl}(0,Q_{\bf 0}),{\cal N}_{cyl}(0,Q_{\bf 1}))\\
&\le \frac T{\nu} \int_{\R^d}\int_{\R^d} (\eps^2\nu^2 \abs{u}^4+1)^{-1} \abs{{\cal F}\theta(u-v)}^2(\eps^2\nu^2\abs{v}^4+1)^{-1}dudv\\
&=\frac T{\nu}\int_{\R^d} \abs{{\cal F}\theta(r)}^2 \Big(\int_{\R^d} (\eps^2 \nu^2\abs{u}^4+1)^{-1}(\eps^2\nu^2\abs{u-r}^4+1)^{-1}du\Big)\,dr\\
&=\frac T{\nu}(\nu\eps)^{-d/2}\int_{\R^d} \abs{{\cal F}\theta(r)}^2 \Big(\int_{\R^d} (\abs{w}^4+1)^{-1} (\abs{w-\eps^{1/2}\nu^{1/2}r}^4+1)^{-1}dw\Big)\,dr \\
&\lesssim\frac T{\nu}(\nu\eps)^{-d/2}\int_{\R^d} \abs{{\cal F}\theta(r)}^2 \Big(\int_{0}^\infty (z^{d-5}\wedge z^{d-1}) (\abs{z-\eps^{1/2}\nu^{1/2}\abs{r}}^{-4}\wedge 1)dz\Big)\,dr.
\end{align*}
As in the derivation of \eqref{EqHellNonpar},
using Lemma \ref{lem:aux:integrals}(b) below,
the last line can be bounded in order by
\[  T\eps^{-d/2}\nu^{-(d+2)/2}\int_{\R^d} \abs{{\cal F}\theta(r)}^2 (1+\eps^{1/2}\nu^{1/2}\abs{r})^{d-5}dr,
\]
provided $d\le 7$. We take $\theta(x)=h^\alpha K(x/h)$
for some bandwidth $h\in(0,1)$ and a smooth divergence-free vector field $K=(K_1,\ldots,K_d):\R^d\to\R^d$ with $K(0)\not=0$, $\norm{K}_{C^\alpha}\le R$, $\int_{\R^d}K_i(x)\,dx=0$ and $\int_{\R^d}K_i(x)x\,dx=0$, $i=1,\ldots,d$, such that $K_i$, $xK_i$, $\abs{x}^2K_i$ and all derivatives of $K_i$ up to order $d$ are in $L^1(\R^d)$ implying $\abs{{\cal F}K(u)}\lesssim \abs{u}^2\wedge \abs{u}^{-d}$.
A concrete construction is given by $K(x)=A\nabla\phi(x)$ for an invertible antisymmetric matrix $A\in\R^{d\times d}$ and $\phi\in C^\infty(\R^d)$ with compact support, $\nabla\varphi(0)\neq 0$ and $\int_{-\infty}^\infty \phi(x_1,\ldots,x_d)dx_j=0$ for all $j=1,\ldots,d$ and $x\in\R^d$. Then $K$ is divergence-free and ${\cal F}K(u)={\cal F}\phi(u)A(-iu)$ has the desired properties.
Following the proof of Theorem \ref{ThmNonpar1} and with $\norm{A\nabla\varphi}_{C^\alpha}$ sufficiently small, we deduce by Lemma \ref{lem:aux:integrals}(a) below
\begin{align*}
&H^2({\cal N}_{cyl}(0,Q_{\bf 0}),{\cal N}_{cyl}(0,Q_{\bf 1}))\\
 &\lesssim T\eps^{-d/2}\nu^{-(d+2)/2}h^{2\alpha+d}\int_{\R^d} (\abs{u}^4\wedge \abs{u}^{-2d}) (1+\eps^{1/2}\nu^{1/2}h^{-1}\abs{u})^{d-5} du\\
 &\lesssim T\eps^{-d/2}\nu^{-(d+2)/2}h^{2\alpha+d}\int_0^\infty (z^{d+3}\wedge z^{-d-1}) (1+\eps^{1/2}\nu^{1/2}h^{-1}z)^{d-5} dz\\
&\lesssim \begin{cases} T\eps^{-d/2}\nu^{-(d+2)/2}h^{2\alpha+d}&\text{, if }h\ge\eps^{1/2}\nu^{1/2},\\ T\eps^{-5/2}\nu^{-7/2}h^{2\alpha+5}&\text{, if }h\le\eps^{1/2}\nu^{1/2}. \end{cases}
\end{align*}
We choose $h=(T^{-1}\eps^{d/2}\nu^{(d+2)/2})^{1/(2\alpha+d)}$ in case $T\le\eps^{-\alpha}\nu^{1-\alpha}$ and $h=(T^{-1}\eps^{5/2}\nu^{7/2})^{1/(2\alpha+5)}$ otherwise. Then we can apply  Theorem \ref{ThmLB} with Euclidean norm $\abs{\theta(0)}\thicksim h^\alpha$, noting
${\bf 0}, \theta\in\Theta_{tra}(\alpha, R)$ for large $n$
by construction.
\end{proof}

\subsection{Space-dependent source}

For each $n\in\N$ consider observations $(dY_t,t\in[0,T_n])$ given by
\begin{equation}\label{EqExNonpar3}
dY_t= X_tdt+\eps_n dV_t\text{ with } dX_t=(\nu_n\Delta X_t-M[\theta]X_t)\,dt+ dW_t
\end{equation}
with the multiplication operator $M[\cdot]$
and $X_0=0$,  $T_n\ge 1$, $\eps_n\in[0,1]$, and $\nu_n\in(0,1]$, where $\Delta$ is the Laplacian on ${\cal H}^2(\R^d)\subset L^2(\R^d)=H$ and $\theta(\cdot)$ belongs to the class
\begin{align*}
	\Theta_{sou}(\alpha,R):=\Big\{\theta\in C^\alpha(\R^d)\,\Big|\, \theta\ge 0,\, \norm{\theta}_{C^\alpha}\le R\Big\},\; \alpha>0,\, R>1.
\end{align*}
We have $A_\theta=R_\theta=\nu_n\Delta-M[\theta]$   and $\dom(A_\theta)={\cal H}^2(\R^d)$ for $\theta\in\Theta_{sou}(\alpha,\R)$.
To apply Theorem \ref{ThmMain}, we take $A_{\bf 0}=\nu_n\Delta-\Id$  and the alternative $A_{\bf 1}=\nu_n\Delta-M[\theta]$ with $\theta(x)=1+h^\alpha K(x/h)$, $h\in(0,1)$, and $K\in C^{1\vee\alpha}(\R^d)$ of compact support.
Then $A_{\bf 0}$ and $A_{\bf 1}$ are selfadjoint negative operators. We have $\bar B_{\bf 0}=\bar B_{\bf 1}=B=\Id$. First, we establish an operator monotonicity result for this case in analogy to Proposition \ref{PropOrderDeltatheta}.

\begin{proposition}\label{PropOrderDeltaplustheta}
Consider $A_{\bf 0}=\nu\Delta-\Id$, $A_{\bf 1}=\nu\Delta-M[\theta]$ with $\nu>0$, $\theta(x)=1+L(x/h)$ for $h\in(0,1)$, $L\in C^1(\R^d)$ and $\supp L\subset[-1/2,1/2]^d$, $\norm{L}_\infty\le 1/2$.
\begin{enumerate}
\item  We have
\begin{equation}\label{EqDeltaplusthetaorder}
(\eps^2 A_{\bf 1}^2+\Id)^{-1}\preccurlyeq 2 (\eps^2 A_{\bf 0}^2+\Id)^{-1}.
\end{equation}

\item If $\norm{L}_\infty+d^{1/2}\nu^{1/4} h^{-1/2}\norm{\nabla L}_\infty\le \frac1{4\sqrt 2}\eps^{-1}$, then
    we have
\begin{equation}\label{EqDeltaplusthetaorder2}
\abs{A_{\bf 1}}_{T^{-1}}^{-1}(\eps^2A_{\bf 1}^2+\Id)^{-1}\preccurlyeq 16 \abs{A_{\bf 1}}_{T^{-1}}^{-1}(\eps^2A_{\bf 0}^2+\Id)^{-1}.
\end{equation}
\end{enumerate}
\end{proposition}

\begin{proof}
	See Section \ref{SecProofNonpar}.
\end{proof}

\begin{theorem}\label{ThmNonpar3}
In terms of the parametric rate $v_n^{par}$ from \eqref{EqParRateSource} with $\beta=0$ define the rate
\[v_n:= \begin{cases} (v_n^{par})^{2\alpha/(2\alpha+d)}, & \text{ if } v_n^{par}\ge (\nu_n\eps_n)^{(2\alpha+d)/4},\\
(\nu_n\eps_n v_n^{par})^{2\alpha/(2\alpha+d+4)}, & \text{ if } d\le 2,\,v_n^{par}\le (\nu_n\eps_n)^{(2\alpha+d)/4},\\
 ((\nu_n\eps_n)^{(7-d)/4}v_n^{par})^{2\alpha/(2\alpha+7)}, & \text{ if } d\ge 3, v_n^{par}\le (\nu_n\eps_n)^{(2\alpha+d)/4}. \end{cases}
\]
If $1\le d\le 9$ and $v_n\to 0$
hold, then there is a constant $c>0$ such that
\[ \liminf_{n\to\infty}\inf_{\hat\theta_n}\sup_{\theta\in \Theta_{sou}(\alpha,R)}\PP_\theta\big(v_n^{-1}\abs{\hat\theta_n(0)-\theta(0)}\ge c\big)>0,
\]
where the infimum is taken over all estimators $\hat\theta_n$ of $\theta$  based on \eqref{EqExNonpar3}, provided the following condition on $\alpha$ is satisfied:
in case $v_n^{par}\ge (\nu_n\eps_n)^{(2\alpha+d)/4}$
\[
\alpha > \limsup_{n\to\infty}
\begin{cases}
	\frac{\log(T_n)-2d\log(\eps_n^{-1})}{2\log(T_n)+4\log(\eps_n^{-1})+(d+1)\log(\nu_n^{-1})},&\text{ if } d=1,\\
	\frac{\log(T_n)+\log(\log(\eps_n^{-1}))-2d\log(\eps_n^{-1})} {2\log(T_n)+2\log(\log(\eps_n^{-1}))+4\log(\eps_n^{-1})+(d+1)\log(\nu_n^{-1})},&\text{ if } d=2,\\
	\frac{\log(T_n)-(1.5d+1)\log(\eps_n^{-1})}{2\log(T_n)+(d+2)\log(\eps_n^{-1})+(d+1)\log(\nu_n^{-1})},&\text{ if } d\ge 3,
\end{cases}
\]
and in case $v_n^{par}\le (\nu_n\eps_n)^{(2\alpha+d)/4}$
\[
\alpha > \limsup_{n\to\infty}
\begin{cases}
	\frac{\log(T_n)-(2d+6)\log(\eps_n^{-1})}{2\log(T_n)+8\log(\eps_n^{-1})+(d+5)\log(\nu_n^{-1})},&\text{ if } d=1,\\
	\frac{\log(T_n)+\log(\log(\eps_n^{-1}))-(2d+6)\log(\eps_n^{-1})} {2\log(T_n)+2\log(\log(\eps_n^{-1}))+8\log(\eps_n^{-1})+(d+5)\log(\nu_n^{-1})},&\text{ if } d=2,\\
	\frac{\log(T_n)-11.5\log(\eps_n^{-1})}{2\log(T_n)+9\log(\eps_n^{-1})+8\log(\nu_n^{-1})},&\text{ if } d\ge 3.
\end{cases}
\]
\end{theorem}

\begin{remark} \
\begin{enumerate}
	\item Writing $v_n^{par}=T_n^{-1/2}\nu_n^{d/4}\eps_n^{(d-2)_+/4}$ (neglecting log terms), we can write the rate as
\[
	v_n =
	\begin{cases}
		(T_n^{-1/2}\nu_n^{d/4}\eps_n^{(d-2)_+/4})^{2\alpha/(2\alpha+d)}, & \text{ if } v_n^{par}\ge (\nu_n\eps_n)^{(2\alpha+d)/4},\\
		(T_n^{-1/2}\nu_n^{(d+4)/4}\eps_n)^{2\alpha/(2\alpha+d+4)}, & \text{ if } d\le 2,\,v_n^{par}\le (\nu_n\eps_n)^{(2\alpha+d)/4},\\
		(T_n^{-1/2}\nu_n^{7/4}\eps_n^{5/4})^{2\alpha/(2\alpha+7)}, & \text{ if } d\ge 3, v_n^{par}\le (\nu_n\eps_n)^{(2\alpha+d)/4}.
 	\end{cases}
\]
The first case gives the classical nonparametric analogue of the parametric rate, which always applies for fixed observation time $T_n$. On the other hand, for $T_n\to\infty$ and $\nu_n,\eps_n$ fixed, the two other cases apply and for $d\le 2$ the rate necessarily slows down to $T_n^{-\alpha/(2\alpha+d+4)}$ instead of $T_n^{-\alpha/(2\alpha+d)}$.
	\item The technical condition that $\alpha$ is not too small is always satisfied if $\alpha\ge 1/2$ or if $T_n\lesssim \eps_n^{-2}$ in $d=1$ or $T_n\lesssim\eps_n^{-(11.5\wedge(1.5d+1))}$ in $d\ge 2$. In particular, it is true if $T_n$ is fixed.
\end{enumerate}
\end{remark}

\begin{proof}
We follow the road exposed in Theorem \ref{ThmNonpar1} and also drop the index $n$.
Using $A_{\bf 0}$ and $A_{\bf 1}$ as in Proposition \ref{PropOrderDeltaplustheta}, for which the conditions will be discussed below, we find from Theorem \ref{ThmMain}, writing $\delta=\theta-1$
\begin{align*}
&\tfrac{1}{4} H^2({\cal N}_{cyl}(0,Q_{\bf 0}),{\cal N}_{cyl}(0,Q_{\bf 1}))\\
&\le  T\norm{(\eps^2(\nu\Delta-\Id)^2+\Id)^{-1/2}M[\delta]\abs{\Id-\nu\Delta}_{T^{-1}}^{-1/2}
(\eps^2(\nu\Delta-\Id)^2+\Id)^{-1/2}}_{HS}^2\\
&\le T\norm{(\eps^2\nu^2\Delta^2+\Id)^{-1/2}M[\delta](\Id-\nu\Delta)^{-1/2}
(\eps^2\nu^2\Delta^2+\Id)^{-1/2}}_{HS}^2\\
&\le  T\norm{M[(\eps^2\nu^2 \abs{u}^4+1)^{-1/2}]C[{\cal F}\delta(u)] M[(\nu\abs{u}^2+1)^{-1/2}(\eps^2\nu^2\abs{u}^4+1)^{-1/2}]}_{HS}^2\\
&=  T\int_{\R^d}\int_{\R^d} (\eps^2 \nu^2\abs{u}^4+1)^{-1} \abs{{\cal F}\delta(u-v)}^2(\nu\abs{v}^2+1)^{-1}(\eps^2\nu^2\abs{v}^4+1)^{-1}dudv\\
&=  T (\nu\eps)^{-\frac d2}\int_{\R^d}\int_{\R^d} \abs{{\cal F}\delta(r)}^2
 (\abs{w+(\nu\eps)^{\frac12}r}^4+1)^{-1} (\eps^{-1}\abs{w}^2+1)^{-1} (\abs{w}^4+1)^{-1}dwdr\\
&\lesssim  T (\nu\eps)^{-\frac d2}\int_{\R^d} \abs{{\cal F}\delta(r)}^2
\Big(\int_0^\infty (\abs{z-(\nu\eps)^{\frac12}\abs{r}}^{-4}\wedge 1) ((\eps z^{-2})\wedge 1) (z^{d-5}\wedge z^{d-1})\,dz\Big)\,dr.
\end{align*}
Set $\rho:=(\nu\eps)^{1/2}\abs{r}\ge 0$. For $1\le d\le 2$ the inner integral can be estimated in order with the help of Lemma \ref{lem:aux:integrals}(b) below, inserting $a=0$, $b=d-7$, $c=-4$ therein, by
\begin{align*}
& \int_0^1 (\eps z^{d-3}\wedge z^{d-1})\,dz\max_{0\le z\le 1}\big(\abs{z-\rho}^{-4}\wedge 1\big) +\eps \int_1^\infty (\abs{z-\rho}^{-4}\wedge 1)z^{d-7}dz\\
&\lesssim \Big(\int_0^{\sqrt\eps}z^{d-1}dz+\eps\int_{\sqrt\eps}^1z^{d-3}dz+\eps\Big)(1+\rho)^{-4}\thicksim \begin{cases} \eps^{d/2}(1+\rho)^{-4},& d=1,\\ \eps^{d/2}\log(e\eps^{-1})(1+\rho)^{-4},& d=2.\end{cases}
\end{align*}
For $3\le d\le 9$ the inner integral can be estimated in order by
\begin{align*}
& \eps \int_0^\infty (\abs{z-\rho}^{-4}\wedge 1)(z^{d-7}\wedge z^{d-3})\,dz
\lesssim \eps(1+\rho)^{d-7}
\end{align*}
using Lemma \ref{lem:aux:integrals}(b) below.
For $d\ge 10$ the inner integral is infinite.

Now let $\delta(x)=h^\alpha K(x/h)$ for $h\in(0,1)$ and $K\in C^\infty(\R^d)$ with support in $[-1/2,1/2]^d$, $\norm{K}_\infty\le 1/2$, $\norm{K}_\alpha\le R-1$, and $K(0)\neq 0$. Then $K$, $xK$ and all derivatives of $K$ are in $L^1(\R^d)$. Assume $\int_{\R^d}x^mK(x)dx=0$ for $m\in\{0,1\}$ so that
 $\abs{{\cal F}K(u)}\lesssim \abs{u}^2\wedge \abs{u}^{-6}$.
For $d\in\{1,2\}$ the Hellinger bound becomes in terms of the parametric rate $v_n^{par}$ from Proposition \ref{PropParametricSource}, using Lemma \ref{lem:aux:integrals}(a) below,
\begin{align*}
&(v_n^{par})^{-2}\int_{\R^d} \abs{{\cal F}\delta(r)}^2 (1+(\nu\eps)^{1/2}\abs{r})^{-4}dr\\
&\lesssim (v_n^{par})^{-2}h^{2\alpha+d}\int_0^\infty (z^{d+3}\wedge z^{d-13}) (1+(\nu\eps)^{1/2}h^{-1}z)^{-4}dz\\
&\thicksim (v_n^{par})^{-2}h^{2\alpha+d}(1+(\nu\eps)^{1/2}h^{-1})^{-4}.
\end{align*}
Hence,  for $h\ge (\nu\eps)^{1/2}$ we choose $h\thicksim (v_n^{par})^{2/(2\alpha+d)}$, while for $h\le (\nu\eps)^{1/2}$ we choose $h\thicksim (\nu\eps v_n^{par})^{2/(2\alpha+d+4)}$.
For $3\le d\le 9$ the Hellinger bound becomes
\begin{align*}
&(v_n^{par})^{-2}\int_{\R^d} \abs{{\cal F}\delta(r)}^2 (1+(\nu\eps)^{1/2}\abs{r})^{d-7}dr\\
&\lesssim (v_n^{par})^{-2}h^{2\alpha+d}\int_0^\infty (z^{d+3}\wedge z^{d-13}) (1+(\nu\eps)^{1/2}h^{-1}z)^{d-7}dz\\
&\thicksim (v_n^{par})^{-2}h^{2\alpha+d}(1+(\nu\eps)^{1/2}h^{-1})^{d-7}.
\end{align*}
For $h\ge(\nu\eps)^{1/2}$ we choose $h\thicksim (v_n^{par})^{2/(2\alpha+d)}$, while for $h\le (\nu\eps)^{1/2}$ we choose $h\thicksim ((\nu\eps)^{(7-d)/4}v_n^{par})^{2/(2\alpha+7)}$.
Writing the condition $h\le (\nu\eps)^{1/2}$ in terms of $v_n^{par}$, we obtain the asserted lower bound rate $v_n\thicksim \abs{\theta(0)-1}\thicksim h^\alpha$ via Theorem \ref{ThmLB}, using ${\bf 1}, \theta\in\Theta_{sou}(\alpha, R)$ for large $n$.

It remains to verify the condition for Proposition \ref{PropOrderDeltaplustheta} with $L=h^\alpha K$, that is $h^{\alpha-1/2}\lesssim \nu^{-1/4}\eps^{-1}$ with a sufficiently small constant.
In case $h\ge(\nu\eps)^{1/2}$ this reduces to $1\lesssim T^{2\alpha-1}\nu^{-(d+1)\alpha}\eps^{-(4\alpha+2d)}$ for $d=1$ and $1\lesssim T^{2\alpha-1}\nu^{-(d+1)\alpha}\eps^{-((d+2)\alpha+1.5d+1)}$ for $d\ge 3$.
In case $h\le(\nu\eps)^{1/2}$ we have to check
$1\lesssim T^{2\alpha-1}\nu^{-(d+5)\alpha}\eps^{-(8\alpha+2d+6)}$ for $d=1$ and
$1\lesssim T^{2\alpha-1}\nu^{-8\alpha}\eps^{-(9\alpha+11.5)}$ for $d\ge 3$.
The case $d=2$ is treated analogously to $d=1$.
By Lemma \ref{lem:aux:alpha} below, all relations hold due to the condition on $\alpha$.
\end{proof}

\appendix

\section{Notation}\label{SecNotation}

Consider real numbers $a_n,b_n,a,b$. We write $a_n\lesssim b_n$ or $a_n={\cal O}(b_n)$ if $a_n\le Cb_n$ holds for a constant $C>0$, uniform over all parameters involved. In an analogous way $a_n\gtrsim b_n$ is defined and $a_n\thicksim b_n$ means $a_n\lesssim b_n$ and $b_n\lesssim a_n$. Moreover, $a_n=o(b_n)$ stands for $a_n/b_n\to 0$. We set $a\wedge b:=\min(a,b)$, $a\vee b:=\max(a,b)$, $a_+=a\vee 0$ and $\abs{a}_T^{-p}:=\abs{a}^{-p}\wedge T^p$, $\abs{a}_{T^{-1}}^{p}:=\abs{a}^p\vee T^{-p}=(\abs{a}_T^{-p})^{-1}$ for $p,T>0$ with $\abs{0}_T^{-p}:=T^p$. The latter is also applied to complex $a$.

For real random variables $X_n$ and deterministic $a_n>0$ we say that $X_n={\cal O}_{\PP_\theta}(a_n)$ holds uniformly over $\theta\in[\underline\theta_n,\bar\theta_n]$ if
\[ \lim_{R\to\infty}\limsup_{n\to\infty}\sup_{\theta\in[\underline\theta_n,\bar\theta_n]}\PP_\theta(\abs{X_n}>Ra_n)=0,\]
that is the laws of $a_n^{-1}X_n$ are uniformly tight under all $\PP_\theta$ and over all $n$.

For $v,w$ in a Hilbert space $H$ the corresponding norm and scalar product are canonically denoted by $\norm{v}$ and $\scapro vw$, respectively. Yet, when $H=\R^d$, we write instead $\abs{v}$ for the Euclidean norm. We denote by $\Id$ the identity operator on $H$ and by $\im(Q)$ the range or image of an operator $Q$. For (possibly unbounded) self-adjoint operators $A, B$ on a Hilbert space $H$  write $A\preccurlyeq B$ and $B\succcurlyeq A$ if  $\dom(B)\subseteq\dom(A)$ and $\scapro{(B-A)v}{v}\ge 0$ for all $v\in\dom(B)$.

For $T>0$ let $L^2([0,T];H)=L^2(H)$ be the Hilbert space of all Borel-measurable $f:[0,T]\to H$ with $\norm{f}_{L^2(H)}^2:=\int_0^T \norm{f(t)}^2dt<\infty$. We write $\norm{A}$ for the operator (or spectral) norm and $\norm{A}_{HS(H)}=\norm{A}_{HS}=\trace(A^\ast A)^{1/2}$ for the Hilbert-Schmidt (or Frobenius) norm of a linear operator $A:H\rightarrow H$. In analogy to real-valued Sobolev spaces we consider ${\cal H}^1([0,T];H)={\cal H}^1(H)$ with
\[\textstyle {\cal H}^1(H)=\Big\{f\in L^2(H)\,\Big|\,\exists g\in L^2(H)\,\forall t\in[0,T]:\;f(t)=f(0)+\int_0^t g(s)ds\Big\}
\]
and ${\cal H}_0^1(H)=\{f\in {\cal H}^1(H)\,|\,f(0)=0\}$, ${\cal H}_T^1(H)=\{f\in {\cal H}^1(H)\,|\,f(T)=0\}$.
Note that point evaluations at $t\in[0,T]$ are well-defined by Sobolev embedding.
Further, using the same notation, define the operator $\partial_t:{\cal H}^1(H)\rightarrow L^2(H)$ via $\partial_tf:=g$.

For $\alpha>0$ we consider the H\"older spaces $C^\alpha([0,T];\R^d)$ of $\floor{\alpha}$-times continuously differentiable functions $f:[0,T]\to\R^d$ with $L:=\sup_{t\not=s}\frac{\abs{f^{(\floor{\alpha})}(t)-f^{(\floor{\alpha})}(s)}}{\abs{t-s}^{\alpha-\floor{\alpha}}}<\infty$, where $f^{(\floor{\alpha})}$ denotes the corresponding derivative.  The H\"older norm is  $\norm{f}_{C^\alpha}=L+\max_{k=0,\ldots\floor{\alpha}}\norm{f^{(k)}}_\infty$, where for integer $\alpha$ we require instead that $f^{(\alpha-1)}$ is $L$-Lipschitz continuous.

Let $\norm{\nabla u}=(\sum_{i=1}^d\int_{\R^d}(\partial_{x_i}u)^2(x)\,dx)^{1/2}$ denote the  $L^2(\R^d;\R^d)$-norm of the gradient and $\norm{\nabla^2u}$ the  $L^2(\R^d;HS(\R^{d\times d}))$-norm of the Hessian for $u$ in a Sobolev space of sufficient regularity. Using integration by parts, we deduce $\norm{\nabla u}^2=\norm{(-\Delta)^{1/2}u}^2$, $\norm{\nabla^2u}^2 = \norm{\Delta u}^2$. Let $\norm{\nabla L}_\infty:= \sup_{x\in\R^d}\norm{\nabla L(x)}_{\R^d}$ and $\norm{\nabla^2 L}_\infty^2:=\sup_{x\in\R^d}\sum_{j,k=1}^d(\partial_{x_j}\partial_{x_k}L(x))^2$ for sufficiently smooth $L:\R^d\to\R$.
We frequently use the interpolation inequality
\begin{equation}\label{EqInterp}
\norm{(-\Delta)^{(1-\alpha)\rho_1+\alpha\rho_2}u}\le \norm{(-\Delta)^{\rho_1}u}^{1-\alpha} \norm{(-\Delta)^{\rho_2}u}^{\alpha},\quad \alpha\in[0,1],\,\rho_1,\rho_2\ge 0,
\end{equation}
whenever the right-hand side is finite. This follows for instance by applying H\"older's inequality in the Fourier representation.

The $L^2$-Fourier transform is given by extending ${\cal F}g(u):=\int_{\R^d}g(x)e^{i\scapro ux}dx$ from $g\in L^1(\R^d)$ to $L^2(\R^d)$.
Denote by $M[f]g:= fg$ the  multiplication operator with $f$, and by $C[f]g:=f*g$ the convolution operator for functions on $\R^d$, whenever these quantities are defined.

\section{Further proofs}\label{SecProofs}

\subsection{Proofs for Section \ref{SecSEE}}\label{SecProofSEE}

\begin{proof}[Proof of Proposition \ref{PropRSnorm}]
 For $f\in L^2([0,T];\dom(A_\theta))$ put $g(t):=e^{-iJ_\theta t}f(t)$ and note $\norm{g}_{L^2(H)}=\norm{f}_{L^2(H)}$. Moreover, it is convenient to  extend $g$ to all of $\R$ via $g(t):=0$ for $t\in\R\setminus[0,T]$. Then we obtain from Lemma \ref{LemCformula}
\begin{align*}
&\scapro{C_\theta R_\theta f}{R_\theta f}_{L^2(H)}\\
&=\frac12\int_0^T\int_0^T \bscapro{e^{iJ_\theta t} R_\theta\Big(\int_{\abs{t-s}}^{t+s} e^{R_\theta v}dv\Big)R_\theta e^{-iJ_\theta s} f(s)}{f(t)}\,dsdt\\
&=\frac12\int_{0}^T\int_0^T \scapro{ R_\theta(e^{R_\theta(t+ s)}-e^{R_\theta\abs{t-s}})g(s)}{g(t)}\,dsdt.
\end{align*}
Observing
\[\int_0^T\int_0^T\scapro{ (R_\theta)_- e^{R_\theta(t+ s)}g(s)}{g(t)}\,dsdt=\bnorm{\int_0^T (R_\theta)_-^{1/2}e^{R_\theta t}g(t)\,dt}^2\ge 0,\]
we have for the negative part
\begin{align*}
&\int_0^T\int_0^T \scapro{ -(R_\theta)_-(e^{R_\theta(t+ s)}-e^{R_\theta\abs{t-s}})g(s)}{g(t)}\,dsdt \\
&\le \int_0^T\int_0^T \scapro{ (R_\theta)_- e^{R_\theta\abs{t-s}}g(s)}{g(t)}\,dsdt\\
&= \int_{-\infty}^\infty\int_{-\infty}^\infty \scapro{ (R_\theta)_- e^{-(R_\theta)_-\abs{t-s}}g(s)}{g(t)}\,dsdt.
\end{align*}
Expressing the scalar product via the trace involving the tensor product of vectors, we can disentangle the integrals over $g$ and the kernel:
\begin{align*}
&\int_{-\infty}^\infty\int_{-\infty}^\infty \scapro{ (R_\theta)_-e^{-(R_\theta)_-\abs{t-s}}g(s)}{g(t)}\,dsdt\\
&=\int_{-\infty}^\infty\trace\Big( (R_\theta)_-e^{-(R_\theta)_-\abs{r}}\Re\Big(\int_{-\infty}^\infty g(t+r)\otimes \overline{g(t)}\,dt\Big)\Big)\,dr.
\end{align*}
Now  for $v\in H$ by  Cauchy-Schwarz inequality
\begin{align*}
\bscapro{\Big(\int_{-\infty}^\infty g(t+r)\otimes \overline{g(t)}\,dt\Big)v}{v}
&=\int_{-\infty}^\infty \scapro{g(t+r)}{v}\scapro{v}{g(t)}\,dt\\
&\le \int_{-\infty}^\infty \abs{\scapro{g(t)}{v}}^2dt
=\bscapro{\Big(\int_{-\infty}^\infty g(t)\otimes \overline{g(t)}\,dt\Big)v}{v}
\end{align*}
holds. We thus obtain by Lemma \ref{LemTraceBound}(a) below
\begin{align*}
& \int_{-\infty}^\infty\trace\Big( (R_\theta)_-e^{-(R_\theta)_-\abs{r}}\Re\Big(\int_{-\infty}^\infty g(t+r)\otimes \overline{g(t)} \,dt\Big)\Big)\,dr\\
&\le \int_{-\infty}^\infty\int_{-\infty}^\infty\trace\big( (R_\theta)_-e^{-(R_\theta)_-\abs{r}} g(t)\otimes \overline{g(t)}\big)\,dtdr\\
&= \int_{-\infty}^\infty\trace\Big( \Big(\int_{-\infty}^\infty(R_\theta)_-e^{-(R_\theta)_-\abs{r}}dr\Big) g(t)\otimes \overline{g(t)}\Big)\,dt\\
&= \int_{-\infty}^\infty\trace\Big( 2g(t)\otimes \overline{g(t)}\Big)\,dt=2\norm{g}_{L^2(H)}^2.
\end{align*}
For the positive part we argue directly, using $(R_\theta)_+e^{R_\theta(t+s)}=(R_\theta)_+e^{(R_\theta)_+(t+s)}$ as well as that $(R_\theta)_+$ is a bounded operator whose semigroup can be extended to negative times, and applying the Cauchy-Schwarz inequality twice:
\begin{align*}
&\int_0^T\int_0^T \scapro{ (R_\theta)_+(e^{R_\theta(t+ s)}-e^{R_\theta\abs{t-s}})g(s)}{g(t)}\,dsdt\\
&=\int_0^T\int_0^T \scapro{(\Id-e^{(R_\theta)_+(\abs{t-s}-(t+s))}) (R_\theta)_+^{1/2}e^{(R_\theta)_+ s} g(s)}{(R_\theta)_+^{1/2}e^{(R_\theta)_+ t} g(t)}\,dsdt\\
&\le\int_0^T\int_0^T \norm{\Id-e^{-2(R_\theta)_+(t\wedge s)}}\norm{(R_\theta)_+^{1/2}e^{(R_\theta)_+ s}}\norm{g(s)} \norm{(R_\theta)_+^{1/2}e^{(R_\theta)_+ t}}\norm{ g(t)}\,ds dt\\
&\le\int_0^T \norm{(R_\theta)_+^{1/2}e^{(R_\theta)_+ t}}^2dt \int_0^T\norm{g(t)}^2 dt\\
&\le \tfrac{1}{2} (e^{2\norm{(R_\theta)_+}T}-1)\norm{g}_{L^2(H)}^2,
\end{align*}
using $\norm{h((R_\theta)_+)} \le h(\norm{(R_\theta)_+})$ for non-negative increasing functions $h$.
We conclude
\begin{align*}
&\norm{S_\theta^\ast R_\theta f}_{L^2(H)}^2= \scapro{ C_\theta R_\theta f}{R_\theta f}_{L^2(H)}\\
&\le \frac12\Big(2+\tfrac{1}{2} (e^{2\norm{(R_\theta)_+}T}-1)\Big)\norm{g}_{L^2(H)}^2
 = \Big(\tfrac34+\tfrac{1}{4} e^{2\norm{(R_\theta)_+}T}\Big)\norm{f}_{L^2(H)}^2.
\end{align*}
Since $L^2([0,T];\dom(A_\theta))$ is dense in $L^2(H)$, this shows that $S_\theta^\ast R_\theta$ and its adjoint $R_\theta S_\theta$ extend to bounded linear operators on $L^2(H)$.
\end{proof}

\begin{proof}[Proof of Proposition \ref{PropSInverse}] \label{proof:PropSInverse}
We present the proof only for $S_\theta$, that for $S_\theta^\ast$ is completely analogous. We proceed stepwise.

\noindent {\it Proof of (a).}
Proposition \ref{PropRSnorm} shows that $R_\theta S_\theta$ extends to a bounded operator on $L^2(H)$, whence $S_\theta$ maps $L^2(H)$ into $L^2([0,T];\dom(R_\theta))$ and the first statement follows from the assumption $\dom(A_\theta)= \dom(R_\theta)$.

For the second property we use standard differentiability properties of Bochner integrals. For $f\in {\cal H}^1(H)$ we obtain by substitution and the chain rule
\begin{equation}\label{EqDtS}
\partial_tS_\theta f(t)=\partial_t\Big(\int_0^t e^{A_\theta u}f(t-u)\,du\Big)=e^{A_\theta t}f(0)+\int_0^t e^{A_\theta u}f'(t-u)\,du
\end{equation}
and the right-hand side is in $L^2(H)$.
This shows  $S_\theta f\in {\cal H}_0^1(H)$ in view of $S_\theta f(0)=0$ by definition.

\noindent {\it Proof of (b).} For $f\in {\cal H}^1(H)$ we have by (a) that $S_\theta f\in {\cal H}_0^1(H)\cap L^2([0,T];\dom(A_\theta))$. Thus, by formula \eqref{EqDtS} and partial integration
\begin{align*}
&(\partial_t-A_\theta)S_\theta f(t)\\
&=e^{A_\theta t}f(0)+\int_0^t e^{A_\theta u}f'(t-u)\,du-\int_0^t A_\theta e^{A_\theta u}f(t-u)\,du\\
&=e^{A_\theta t}f(0)+\int_0^t e^{A_\theta u}f'(t-u)\,du-\Big(e^{A_\theta u}f(t-u)\Big|_{u=0}^t+\int_0^t e^{A_\theta u}f'(t-u)\,du\Big)\\
&=f(t).
\end{align*}
This gives $(\partial_t-A_\theta)S_\theta=\Id$ on ${\cal H}^1(H)$. Clearly, $A_\theta$ and $S_\theta$ commute on $L^2([0,T];\dom(A_\theta))$. Moreover, for $f\in {\cal H}_0^1(H)$ we find by formula \eqref{EqDtS} with $f(0)=0$
\begin{equation}\label{EqdtScomm} \partial_tS_\theta f(t)=\int_0^t e^{Au}f'(t-u)\,du=S_\theta \partial_t f(t).\end{equation}
This finishes the proof of (b).

\noindent {\it Proof of (c).} For $f\in {\cal H}^1(H)$ we have $S_{\bf 0}f,S_{\bf 1}f\in {\cal H}_0^1(H)\cap L^2([0,T];\dom(A_\theta))$ by (a). Part (b) yields
    \begin{align*}
    	S_{\bf 0}(A_{\bf 1}-A_{\bf 0})S_{\bf 1}f&=S_{\bf 0}(\partial_t-A_{\bf 0})S_{\bf 1}f-S_{\bf 0}(\partial_t-A_{\bf 1})S_{\bf 1}f=S_{\bf 1}f-S_{\bf 0}f, \\
    	S_{\bf 1}(A_{\bf 1}-A_{\bf 0})S_{\bf 0}f&=S_{\bf 1}(\partial_t-A_{\bf 0})S_{\bf 0}f-S_{\bf 1}(\partial_t-A_{\bf 1})S_{\bf 0}f=S_{\bf 1}f-S_{\bf 0}f.
    \end{align*}
    Hence, $S_{\bf 1}-S_{\bf 0}=S_{\bf 0}(A_{\bf 1}-A_{\bf 0})S_{\bf 1}=S_{\bf 1}(A_{\bf 1}-A_{\bf 0})S_{\bf 0}$ holds on ${\cal H}^1(H)$. Since all three terms present bounded linear operators on $L^2(H)$ and ${\cal H}^1(H)$ is dense in $L^2(H)$ \cite[Eq. (1.2.3)]{amann2019}, the assertion follows by a continuity argument.
\end{proof}

\subsection{Proofs for Section \ref{SecParRates}}\label{SecUpperBoundProofs}

We prove the parametric upper bound in Theorem \ref{ThmUpperBound} by establishing several intermediate results. Throughout Assumption \ref{AssPar} is in force.

In Definition \ref{Defhattheta} we write short $Z_n=-\scapro{(\partial_t+M_n)\Lambda K^{(n)}\dot Y}{\dot Y}_{L^2}$, $N_n=\scapro{\abs{\Lambda}^2 K^{(n)} \dot Y}{\dot Y}_{L^2}$
and we shall mostly use this formal notation with scalar products in $L^2=L^2([0,T]^2;H)$ and It\^o differentials, interpreting $\dot Y_tdt=dY_t$ and taking the inner integral over the entry of the $L^2$-scalar product, in the sequel, making sure that the integrands are always adapted. By definition,  $K^{(n)}(t,s)=0$ holds for $t\le s$ and for $t=T_n$. For $s<t$ we have $\psi_{t-s,s}^{(n)}(a)=0$ for $\abs{a}\ge (t-s)^{-1}$ so that $K^{(n)}(t,s)$ introduces the projection ${\bf 1}(\abs{\bar A_n}\le (t-s)^{-1})$ onto a finite-dimensional subspace by the assumption of a compact resolvent. This shows that $K^{(n)}(t,s)$ is for all $t,s\in[0,T_n]$ a finite-rank operator. Since $\psi_{t-s,s}^{(n)}$ is weakly differentiable in $t$, we have $(\partial_t+M_n)\Lambda K^{(n)}(t,s), \abs{\Lambda}^2K^{(n)}(t,s)\in HS(H)$ for all $t,s\in[0,T_n]$. The stochastic integrals in the definition are well-defined if $(\partial_t+M_n)\Lambda K^{(n)},\abs{\Lambda}^2 K^{(n)}\in HS(L^2):=HS(L^2([0,T_n];H))$ holds (cf. \cite{DPZa2014}), which depends on the behaviour of $K^{(n)}(t,s)$ as $t-s\downarrow 0$.
In Propositions \ref{PropZ} and \ref{PropN} below, we shall establish conditions ensuring that $Z_n-\theta N_n$ and $N_n$ are well defined. Before, we work implicitly under this  assumption.

\begin{lemma}
The estimation error of $\hat\theta_n$ allows on $\{N_n\not=0\}$ the decomposition
\begin{equation*} \hat\theta_n-\theta=\frac{\scapro{\Lambda K^{(n)}\dot Y}{B_n\dot W}_{L^2}-\scapro{(\partial_t+A_{\theta}^\ast) \Lambda K^{(n)}\dot Y}{ \eps_n \dot V}_{L^2}}{ N_n}.
\end{equation*}
\end{lemma}

\begin{proof}
By partial integration in $t$, using  $K^{(n)}(T_n,s)=K^{(n)}(s,s)=0$ for the boundary values and noting that the inner integral is adapted due to $\supp(K^{(n)}(t,\cdot))\subset[0,t]$, we have
\begin{align*}
&-\scapro{(\partial_t+A_{\theta}^\ast) \Lambda K^{(n)}\dot Y}{B_nX}_{L^2}\\
 &=-\int_0^{T_n}\int_0^{T_n} \scapro{\partial_t+A_{\theta}^\ast) \Lambda K^{(n)}(t,s)\,dY_s}{B_nX_tdt}\\
 &=\int_0^{T_n}\int_0^t \Big(\scapro{\Lambda K^{(n)}(t,s)\,dY_s}{B_ndX_t}-\scapro{\Lambda K^{(n)}(t,s)\,dY_s} {A_{\theta}B_n X_tdt}\Big)\\
&=\int_0^{T_n}\int_0^{T_n} \scapro{\Lambda K^{(n)}(t,s)dY_s}{B_ndW_t}=\scapro{\Lambda K^{(n)}\dot Y}{B_n\dot W}_{L^2},
\end{align*}
by commutativity of $A_\theta$ and $B_n$.
This gives
\begin{align}
Z_n-\theta N_n &= -\scapro{(\partial_t+A_{\theta}^\ast)\Lambda K^{(n)}\dot Y}{B_nX+\eps_n \dot V}_{L^2}\nonumber\\
&= \scapro{\Lambda K^{(n)}\dot Y}{ B_n\dot W}_{L^2}-\scapro{(\partial_t+A_{\theta}^\ast) \Lambda K^{(n)}\dot Y}{\eps_n \dot V}_{L^2}.\label{EqErrDecomp}
\end{align}
It remains to divide by $N_n$ for $N_n\not=0$.
\end{proof}

\begin{proposition}\label{PropZ}
$Z_n-\theta N_n$ is well defined and we have
\[ \E[(Z_n-\theta N_n)^2]\lesssim  T_n\trace\Big(K_n^2\abs{\Lambda}^2  \big(\eps_n^4\abs{\bar A_n}_{T_n}^{-1}+B_n^4 \abs{\bar A_n}_{T_n}^{-4} \abs{R_{\theta}}_{T_n}^{-1}\big)\Big)
\]
if the right-hand side is finite.
\end{proposition}

\begin{proof}
The representation in \eqref{EqErrDecomp} shows that $Z_n-\theta N_n$ is the sum of stochastic integrals with respect to $dW$ and $dV$. This implies $\E[Z_n-\theta N_n]=0$, provided the corresponding variance expression is finite, which will be established next.

By Lemma \ref{LemGVar} below we obtain for the variance
\begin{align*}\
\Var(Z-\theta N)=\Var(\scapro{(\partial_t+A_{\theta}^\ast)\Lambda K^{(n)}\dot Y}{\dot Y}_{L^2})
&\le 2\norm{Q_{\theta}(\partial_t+A_{\theta}^\ast)\Lambda K^{(n)}}_{HS(L^2)}^2.
\end{align*}
By $S_{\theta}^\ast(\partial_t+A_{\theta}^\ast)=-\Id$ on $L^2([0,T_n];\dom(A_{\theta}))\cap {\cal H}^1(H)$ from Proposition \ref{PropSInverse} and commutativity we obtain
\begin{align}
\Var(Z_n-\theta N_n) &\le 2\norm{(\eps_n^2\Id+B_nS_{\theta} S_{\theta}^\ast B_n) (\partial_t+A_{\theta}^\ast) \Lambda K^{(n)}}_{HS(L^2)}^2\nonumber\\
&\le  4\eps_n^4\norm{\Lambda(\partial_t+A_{\theta}^\ast) K^{(n)}}_{HS(L^2)}^2+4\norm{\Lambda B_n^2 S_{\theta}  K^{(n)}}_{HS(L^2)}^2\label{EqVarZ}
\end{align}
Inserting $K^{(n)}(t,s)=K_n\psi^{(n)}_{t-s,s}(\bar A_n)$, the Hilbert-Schmidt norm representation via kernels according to Lemma \ref{lem:HS:kernel} below yields
\begin{align*}
\norm{\Lambda(\partial_t+A_{\theta}^\ast) K^{(n)}}_{HS(L^2)}^2
&= \int_0^{T_n}\int_s^{T_n}\norm{\Lambda K_n (\partial_t+A_{\theta}^\ast)\psi_{t-s,s}^{(n)}(\bar A_n)}_{HS(H)}^2dtds\\
 &= \int_0^{T_n}\trace\Big(K_n^2\abs{\Lambda}^2 \int_0^{T_n-s}\abs{(\partial_v+A_{\theta}^\ast)\psi_{v,s}^{(n)}(\bar A_n)}^2dv \Big)\,ds.
\end{align*}
The simple identities
\begin{equation}\label{Eqpsi}
\int_0^{T_n-s}\abs{\psi_{v,s}^{(n)}(a)}^2dv=\tfrac1{12} \abs{a}_{T_n-s}^{-3},\quad \int_0^{T_n-s}\abs{\partial_v\psi_{v,s}^{(n)}(a)}^2dv
= \abs{a}_{T_n-s}^{-1}
\end{equation}
and $\abs{A_{\theta}}^2\preccurlyeq\abs{\bar A_n}^2$ yield further, using Lemma \ref{LemTraceBound}(a) below,
\begin{align}
&\norm{\Lambda(\partial_t+A_{\theta}^\ast) K^{(n)}}_{HS(L^2)}^2\label{EqVarZ1}\\ &\lesssim
T_n\trace\Big(K_n^2\abs{\Lambda}^2  \big(\Id+\abs{A_{\theta}}^2\abs{\bar A_n}_{T_n}^{-2}\big)\abs{\bar A_n}_{T_n}^{-1}\Big)
\lesssim
T_n\trace\Big(K_n^2\abs{\Lambda}^2 \abs{\bar A_n}_{T_n}^{-1}\Big).\nonumber
\end{align}

For the second term in \eqref{EqVarZ} we use the kernel of $S_{\theta}$ from Lemma \ref{LemCovOp},
\begin{equation}\label{Eqpsi2}
\int_0^{T_n-s}\psi_{v,s}^{(n)}(a)\,dv=\tfrac14 \abs{a}_{T_n-s}^{-2}
\end{equation}
as well as (note $\abs{R_\theta}w\preccurlyeq \Id$ on the support  of $\psi_{w,s}^{(n)}(\bar A_n)$)
\begin{equation}\label{EqRbounded}
e^{-R_\theta w}\psi_{w,s}^{(n)}(\bar A_n)\preccurlyeq e^1\psi_{w,s}^{(n)}(\bar A_n),\quad w\ge  0,
\end{equation}
 to find for $s\le t$
\begin{align*}
 \abs{(S_{\theta,n}K^{(n)})(t,s)}&=\babs{\int_0^t e^{A_{\theta}(t-u)}K_n\psi_{u-s,s}^{(n)}(\bar A_n)\,du}\\
 &\precsim K_n e^{R_{\theta}(t-s)}\int_s^t e^{R_{\theta}(s-u)}\psi_{u-s,s}^{(n)}(\bar A_n)\,du
 \precsim K_n e^{R_{\theta}(t-s)}\abs{\bar A_n}_{T_n}^{-2}.
\end{align*}
We thus obtain
\begin{align*}
\norm{\Lambda B_n^2 S_{\theta}  K^{(n)}}_{HS(L^2)}^2&=\int_0^{T_n}\int_0^t \norm{\Lambda B_n^2 (S_{\theta} K^{(n)})(t,s)}_{HS(H)}^2dsdt\\
 &\le \trace\Big(K_n^2\abs{\Lambda}^2 B_n^4 \abs{\bar A_n}_{T_n}^{-4}\int_0^{T_n}\int_0^t e^{2R_{\theta}s}dsdt\Big)\\
 &\le T_n\trace\Big(K_n^2\abs{\Lambda}^2 B_n^4 \abs{\bar A_n}_{T_n}^{-4}\abs{R_{\theta}}_{T_n}^{-1}\Big).
\end{align*}
Inserting \eqref{EqVarZ1} and this bound into \eqref{EqVarZ}, we arrive at
\begin{align*}
\Var(Z_n-\theta N_n)&\lesssim  T_n\trace\Big(K_n^2\abs{\Lambda}^2  \big(\eps_n^4\abs{\bar A_n}_{T_n}^{-1}+B_n^4 \abs{\bar A_n}_{T_n}^{-4} \abs{R_{\theta}}_{T_n}^{-1}\big)\Big),
\end{align*}
which gives the claim. Arguing from bottom to top, $Z_n-\theta N_n$ is well defined if the right-hand side is finite.
\end{proof}

\begin{proposition}\label{PropN}
We have
\begin{align*}
\E[N_n] &\gtrsim  T_n\trace\Big(K_n\abs{\Lambda}^2B_n^2   \abs{\bar A_n}_{T_n}^{-2}\abs{R_{\theta}}_{T_n}^{-1} \Big),\\
\Var(N_n) & \lesssim T_n\trace\Big(K_n^2\abs{\Lambda}^4\big(\eps_n^4\abs{\bar A_n}_{T_n}^{-3} + B_n^4\abs{\bar A_n}_{T_n}^{-4}\abs{R_{\theta}}_{T_n}^{-3}\big)\Big),
\end{align*}
where $N_n$ is well defined if both expressions on the right are finite.
\end{proposition}

\begin{proof}
By $\E[\scapro{\abs{\Lambda}^2 K^{(n)}\dot Y}{\dot V}_{L^2}]=0$  and by independence between $X$ and $V$ we have
\begin{align*}
 \E[N_n]&=\E[\scapro{\abs{\Lambda}^2 K^{(n)} \dot Y}{B_nX}_{L^2}]\\
 &=\E[\scapro{B_n\abs{\Lambda}^2 K^{(n)} B_nX}{X}_{L^2}]\\
 &=\int_0^{T_n}\int_0^t\trace\Big(K_n\abs{\Lambda}^2B_n^2\psi_{t-s,s}^{(n)}(\bar A_n)  \tfrac12 \Re\big(e^{A_{\theta} (t-s)}\big)\int_{0}^{2s} e^{R_{\theta} v}dv \Big)\,dsdt,
\end{align*}
where we used the real part of the covariance kernel for $X$ from Lemma \ref{LemCformula}, noting that $\E[N_n]$ is real-valued.
 Now we have for $s<t$ due to $\abs{A_{\theta}}\preccurlyeq\abs{\bar A_n}$ and $\min_{z\in\C,\abs{z}\le 1}\Re(e^z)\ge e^{-1}\cos(1)$
\begin{align*}
\psi_{t-s,s}^{(n)}(\bar A_n)\Re(e^{A_{\theta}(t-s)}) &=\psi_{t-s,s}^{(n)}(\bar A_n){\bf 1}(\abs{\bar A_n}\le (t-s)^{-1})\Re(e^{A_{\theta}(t-s)})\\
&\succcurlyeq e^{-1}\cos(1)\psi_{t-s,s}^{(n)}(\bar A_n).
\end{align*}
With \eqref{Eqpsi2}
and $\int_0^{t}\int_0^{2s}e^{-rv}dvds=\int_0^t r^{-1}(1-e^{-2rs})ds\ge \frac t2(r^{-1}\wedge t)$ for $r,t> 0$ this implies, again using Lemma \ref{LemTraceBound}(a),
\begin{align}
\E[N_n] &\ge \frac{\cos(1)}{2e}\int_0^{T_n}\trace\Big(K_n\abs{\Lambda}^2B_n^2    \int_0^{T_n-s}\psi_{v,s}^{(n)}(\bar A_n)\,dv\,\int_{0}^{2s} e^{R_{\theta} v}dv \Big)\,ds\nonumber\\
&= \frac{\cos(1)}{8e}\int_0^{T_n}\trace\Big(K_n\abs{\Lambda}^2B_n^2  \abs{\bar A_n}_{T_n-s}^{-2}\int_{0}^{2s} e^{R_{\theta} v}dv \Big)\,ds\nonumber\\
&\ge \frac{\cos(1)}{8e}\trace\Big(K_n\abs{\Lambda}^2B_n^2  \abs{\bar A_n}_{T_n/2}^{-2}\int_0^{T_n/2}\int_{0}^{2s} e^{R_{\theta} v}dvds \Big)\nonumber\\
&\gtrsim T_n\trace\Big(K_n\abs{\Lambda}^2B_n^2  \abs{\bar A_n}_{T_n}^{-2}\abs{R_{\theta}}_{T_n}^{-1} \Big).\label{EqNLB}
\end{align}
For the variance of $N_n$ we obtain from Lemma \ref{LemGVar} and Lemma \ref{lem:HS:kernel} below
\[ \Var(N_n) \le 2\int_0^{T_n}\int_0^{T_n}\norm{((\eps_n^2\Id+B_n^2C_{\theta})\abs{\Lambda}^2 K^{(n)})(t,s) }_{HS}^2dsdt
\]
in terms of the corresponding operator kernel. We bound the kernel of $C_{\theta}K^{(n)}$, using the kernel
 of $C_{\theta}$ from Lemma \ref{LemCformula}, \eqref{EqRbounded} and \eqref{Eqpsi2}:
\begin{align*}
 \abs{(C_{\theta}K^{(n)})(t,s)}&\preccurlyeq\frac12\int_0^{T_n} \Big(\int_{\abs{t-u}}^{t+u} e^{R_{\theta} v}dv\Big)K_n\psi_{u-s,s}^{(n)}(\bar A_n)\,du\\
&=\tfrac12K_ne^{R_{\theta}\abs{t-s}}\int_0^{T_n} \Big(\int_{\abs{t-u}-\abs{t-s}}^{t+u-\abs{t-s}} e^{R_{\theta} v}dv\Big)\psi_{u-s,s}^{(n)}(\bar A_n)\,du\\
&\preccurlyeq \tfrac12K_ne^{R_{\theta}\abs{t-s}}\int_0^{T_n-s} \Big(\int_{-w}^{2T_n} e^{R_{\theta} v}dv\Big)\psi_{w,s}^{(n)}(\bar A_n)\,dw\\
&\preccurlyeq \tfrac 12K_ne^{R_{\theta}\abs{t-s}}\int_0^{T_n} e\abs{R_{\theta}}_{3T_n}^{-1}\psi_{w,s}^{(n)}(\bar A_n)\,dw\\
&\preccurlyeq \tfrac e2K_ne^{R_{\theta}\abs{t-s}}\abs{R_{\theta}}_{3T_n}^{-1}\abs{\bar A_n}_{T_n}^{-2}
\end{align*}
and therefore, using Lemma \ref{LemTraceBound}(a) and \eqref{Eqpsi},
\begin{align*}
&\Var(N_n)\\ 
&\lesssim \trace\Big(K_n^2\abs{\Lambda}^4\int_0^{T_n}\int_0^{T_n} \Big(\eps_n^4\psi_{t-s,s}^{(n)}(\bar A_n)^2 + B_n^4e^{2R_{\theta}\abs{t-s}}\abs{R_{\theta}}_{T_n}^{-2}\abs{\bar A_n}_{T_n}^{-4}\Big)\,dsdt\Big) \\
&\lesssim T_n\trace\Big(K_n^2\abs{\Lambda}^4\big(\eps_n^4\abs{\bar A_n}_{T_n}^{-3} + B_n^4\abs{R_{\theta}}_{T_n}^{-3}\abs{\bar A_n}_{T_n}^{-4}\big)\Big).
\end{align*}
as asserted.
$N_n$ is well defined if this variance bound and $\E[N_n]$ are finite. In fact, by construction we have $\E[N_n]\ge 0$. Tracing the  lower bound \eqref{EqNLB} for $\E[N_n]$, we see that the converse inequalities also hold up to multiplicative constants, which establishes $\E[N_n]<\infty$ if the trace is finite as assumed.
\end{proof}

\begin{proof}[Proof of Theorem \ref{ThmUpperBound}]
From Proposition \ref{PropZ} and Lemma \ref{LemTraceBound}(a) below with $\abs{R_{\theta}}_{T_n^{-1}}\preccurlyeq\abs{\bar R_n}_{T_n^{-1}}$ we obtain
\[\E[(Z_n-\theta N_n)^2]\lesssim   T_n\trace\Big(K_n^2\abs{\Lambda}^2  \big(\eps_n^4\abs{\bar R_n}_{T_n^{-1}}+B_n^4 \abs{\bar A_n}_{T_n}^{-3}\big) \abs{\bar A_n}_{T_n}^{-1}\abs{R_{\theta}}_{T_n}^{-1}\Big),\]
which equals ${\cal I}_n(\theta)$ by inserting the choice of $K_n$. Equally,
$\E[N_n]\gtrsim {\cal I}_n(\theta)$  follows from Proposition \ref{PropN}.   In view of Proposition \ref{PropN}, Condition \eqref{EqVarNBound} then ensures that $\Var(N_n)^{1/2}/\E[N_n]\to 0$. This implies $N_n/\E[N_n]\xrightarrow{\PP_\theta} 1$ and in particular $\PP_\theta(N_n=0)\to 0$, while
\[(Z_n-\theta N_n)/\E[N_n]={\cal O}_{\PP_\theta}\big({\cal I}_n(\theta)^{1/2}/\E[N_n]\big)={\cal O}_{\PP_\theta}\big({\cal I}_n(\theta)^{-1/2}\big).\]
Since $Z_n-\theta N_n$, $N_n$ are well defined, $\hat\theta_n$ is well defined and we conclude that $\hat\theta_n-\theta=(Z_n{\bf 1}(N_n\not=0) -\theta N_n)/N_n ={\cal O}_{\PP_\theta}({\cal I}_n(\theta)^{-1/2})$ holds. The convergence of the error is uniform over $\theta$ because   \eqref{EqVarNBound} and all employed moment bounds hold uniformly in $\theta$.
\end{proof}

\subsection{Proofs for Section \ref{SecNonparEx}}\label{SecProofNonpar}

\begin{proof}[Proof of Proposition \ref{PropOrderDeltatheta}.] \label{proof:PropOrderDeltatheta}
In the proof the constants are tracked precisely and  the $C_{d, \beta}^{(i)}$ are used only to state the results concisely.
We use repeatedly that for $x,y\ge 0$
\begin{align}
	((x-y)_+)^2&\ge \tfrac{1}{2}x^2-y^2. \label{eq:LowerPowerBoundSquare}
\end{align}

\noindent {\it Proof of (a):}
For $u\in{\cal H}^6(\R^d)$ integration by parts and $\theta(x)\ge 1/2$ yield
\begin{equation}\label{EqInta1} \scapro{(-\Delta_\theta)^3u}{u}=\scapro{\theta\nabla \Delta_\theta u}{\nabla \Delta_\theta u}\ge \tfrac12\scapro{\nabla\Delta_\theta u}{\nabla\Delta_\theta u}=\tfrac12\norm{\nabla\Delta_\theta u}^2.
\end{equation}
From the identity (with matrix-vector multiplication)
\[ \nabla\Delta_\theta u=\theta\nabla\Delta u+(\nabla^2\theta)\nabla u+(\Delta u)(\nabla\theta)+(\nabla^2 u)(\nabla\theta)
\]
we derive by the inverse triangle inequality and $\theta(x)\ge 1/2$
\begin{equation}\label{EqInta2}
\norm{\nabla\Delta_\theta u}\ge\tfrac12\norm{\nabla\Delta u}-\norm{(\nabla^2\theta)\nabla u}-\norm{(\Delta u)(\nabla\theta)}-\norm{(\nabla^2 u)(\nabla\theta)}.
\end{equation}
We have, using Lemma \ref{lem:CompactProductInterpolation} below, \eqref{EqInterp}
as well as $AB\le A^2/8 + 2B^2$ for $A,B\ge0$ in the last step,
\begin{align*}
\norm{(\nabla^2\theta)\nabla u}
&\le (2d^3h)^{1/2} \norm{\nabla^2\theta}_\infty \norm{\Delta u}^{1/2}\norm{\nabla u}^{1/2}\\
&\le \sqrt{2d^3} h^{-2+1/2}\norm{\nabla^2L}_{\infty} \norm{\nabla\Delta u}^{1/2}\norm{u}^{1/2}\\
&\le \tfrac{1}{8}\norm{\nabla\Delta u}+4d^6h^{-3}\norm{\nabla^2 L}_\infty^2 \norm{u}.
\end{align*}
Similarly, we obtain, using $A^{5/6}B^{1/6}\le \frac56 A+\frac16B$,
\begin{align*}
\norm{(\Delta u)(\nabla\theta)}+\norm{(\nabla^2 u)\nabla \theta}
&\le 2(2d^3h)^{1/2}h^{-1}\norm{\nabla L}_\infty \norm{\nabla\Delta u}^{1/2}\norm{\Delta u}^{1/2}\\
&\le 2^{3/2} d^{3/2}h^{-1/2}\norm{\nabla L}_\infty \norm{\nabla\Delta u}^{5/6}\norm{u}^{1/6}\\
&\le \tfrac{1}{8}\norm{\nabla\Delta u}+ \tfrac{10^52^{13}}{3^6} d^9h^{-3}\norm{\nabla L}_\infty^6 \norm{u}.
\end{align*}
Hence, in \eqref{EqInta2} we have  obtained the bound
\begin{align*}
\norm{\nabla\Delta_\theta u}&\ge\tfrac{1}{4}\norm{\nabla\Delta u}-\Big(4d^6\norm{\nabla^2L}_\infty^2
 + \tfrac{10^52^{13}}{3^6}d^9\norm{\nabla L}_\infty^6\Big)h^{-3}\norm{u}.
\end{align*}
Using \eqref{eq:LowerPowerBoundSquare},  insertion into \eqref{EqInta1} and the denseness of ${\cal H}^6(\R^d)$ in ${\cal H}^3(\R^d)$ therefore yield
\begin{align*}
(-\Delta_\theta)^3&\succcurlyeq\tfrac{1}{64}(-\Delta)^3-c_{\theta,d}^2h^{-6}\Id
\end{align*}
with
$c_{\theta,d}=2^{3/2} d^6\norm{\nabla^2L}_\infty^2 + \tfrac{10^52^{25/2}}{3^6} d^9\norm{\nabla L}_\infty^6$.

Generally, for positive operators $T_1,T_2$ and $C>0$, $\gamma\in[0,1]$   we obtain from the operator monotonicity of $t\mapsto t^\gamma$ \cite[Thm. V.1.9]{bhatia2013}
\begin{equation}\label{EqOpPowers} T_1\preccurlyeq T_2+C\Id \Rightarrow T_1^\gamma\preccurlyeq (T_2+C\Id)^\gamma \preccurlyeq T_2^\gamma+C^\gamma \Id.
\end{equation}
For $T_1=(-\Delta)^3$, $T_2=(-\Delta_\theta)^3$ and $\gamma=(2+2\beta)/3$ this yields
\begin{align*}
(-\Delta_\theta)^{2+2\beta}
&\succcurlyeq 2^{-4-4\beta}(-\Delta)^{2+2\beta}-c_{\theta,d}^{(4+4\beta)/3}h^{-(4+4\beta)}\Id,\\
\eps^2(-\Delta_\theta)^{2+2\beta}+\Id &\succcurlyeq \big(2^{-4-4\beta}\wedge (1-\eps^2c_{\theta,d}^{(4+4\beta)/3}h^{-(4+4\beta)})\big)\big( \eps^2(-\Delta)^{2+2\beta} +\Id\big)
\end{align*}
and thus by operator monotonicity of $x\mapsto -x^{-1}$ \cite[Prop. V.1.6]{bhatia2013} the bound in (a) for $h^{-2}c_{\theta,d}^{2/3}\le\frac12\eps^{-1/(1+\beta)}$.

{\it Proof of (b):}
Given the result in (a) and $\frac12(x+T^{-1})\le x\vee T^{-1}\le x+T^{-1}$, $x\ge 0$, it suffices to show
\begin{align}\label{eq:proofDiffusionPartBReduction}
	(-\Delta_\theta)(\eps^2(-\Delta_\theta)^{2+2\beta}+\Id)\succcurlyeq 2(C_{d, \beta}^{(4)})^{-1} (-\Delta)(\eps^2(-\Delta)^{2+2\beta}+\Id).
\end{align}
For $u\in {\cal H}^4(\R^d)$ we  bound by Lemma \ref{lem:CompactProductInterpolation} below, \eqref{EqInterp}
and $AB\le \tfrac{3}{4}A^{4/3}+\tfrac{1}{4}B^4$
\begin{align}
\norm{\Delta_\theta^2u} &= \norm{\theta\Delta \Delta_\theta u+\scapro{\nabla\theta}{\nabla\Delta_\theta u}}\nonumber\\
&\ge \tfrac12\norm{\Delta\Delta_\theta u}-(2d^2h)^{1/2}h^{-1}\norm{\nabla L}_\infty\norm{\Delta\Delta_\theta u}^{1/2}\norm{\nabla\Delta_\theta u}^{1/2}\nonumber\\
&\ge \tfrac12\norm{\Delta\Delta_\theta u}-\sqrt 2 d h^{-1/2}\norm{\nabla L}_\infty\norm{\Delta\Delta_\theta u}^{3/4}\norm{\Delta_\theta u}^{1/4}\nonumber\\
&\ge \tfrac14\norm{\Delta\Delta_\theta u}-27 d^4h^{-2}\norm{\nabla L}_\infty^4\norm{\Delta_\theta u}.\label{EqDeltatheta2}
\end{align}
Further, by expanding $\Delta\Delta_\theta$ and using the inverse triangle inequality, Lemma \ref{lem:CompactProductInterpolation} below
together with the bound $\norm{\Delta L}_\infty^2\le d\norm{\nabla^2 L}_\infty^2$, \eqref{EqInterp}
and  $A^\alpha B^{1-\alpha}\le \alpha A+(1-\alpha)B$ for $A,B\ge 0$, $\alpha\in[0,1]$,
\begin{align*}
&\tfrac{1}{2}\norm{\Delta^2 u}-\norm{\Delta\Delta_\theta u}\\
&\le \norm{3\scapro{\nabla\theta}{\nabla\Delta u}+(\Delta\theta)(\Delta u)+2\scapro{\nabla^2\theta}{\nabla^2u}_{HS}+\scapro{\nabla\Delta\theta}{\nabla u}}\\
&\le 3(2d^2h)^{1/2}h^{-1}\norm{\nabla L}_\infty\norm{\Delta^2 u}^{1/2} \norm{\nabla\Delta u}^{1/2} \\
&\quad +3(2d^4h)^{1/2}h^{-2}\norm{\nabla^2 L}_\infty\norm{\nabla\Delta u}^{1/2}\norm{\Delta u}^{1/2} \\
&\quad +(2d^2h)^{1/2}h^{-3}\norm{\nabla\Delta L}_\infty\norm{\Delta u}^{1/2}\norm{\nabla u}^{1/2}\\
&\le \sqrt{18}dh^{-1/2}\norm{\nabla L}_\infty\norm{\Delta^2 u}^{5/6}\norm{\nabla u}^{1/6}\\
&\quad +\sqrt{18}d^2h^{-3/2}\norm{\nabla^2 L}_\infty\norm{\Delta^2 u}^{1/2}\norm{\nabla u}^{1/2} \\
&\quad +\sqrt{2} d h^{-5/2}\norm{\nabla\Delta L}_\infty\norm{\Delta^2 u}^{1/6}\norm{\nabla u}^{5/6}\\
&\le \tfrac{1}{12}\norm{\Delta^2 u}+ 15^52^7d^6h^{-3}\norm{\nabla L}_\infty^6\norm{\nabla u} +\tfrac{1}{12}\norm{\Delta^2 u}+ 54 d^4h^{-3}\norm{\nabla^2 L}_\infty^2\norm{\nabla u} \\
&\quad +\tfrac{1}{12}\norm{\Delta^2 u}+ 2^{-1/5}3^{-1}5 d^{6/5}h^{-3}\norm{\nabla\Delta L}_\infty^{6/5}\norm{\nabla u}.
\end{align*}
Consequently,
\begin{align*}
 \norm{\Delta\Delta_\theta u}\ge \tfrac14\norm{\Delta^2 u}-\Big( & 15^52^7d^6\norm{\nabla L}_\infty^6+ 54 d^4 \norm{\nabla^2 L}_\infty^2 \\
	&+ 2^{-1/5}3^{-1}5d^{6/5} \norm{\nabla\Delta L}_\infty^{6/5}\Big)h^{-3}\norm{\nabla u}
\end{align*}
holds. Using $\norm{\theta}_\infty\le 3/2$, the other term in \eqref{EqDeltatheta2} can be simply bounded for any $\eta>0$ via
\begin{align*}
\norm{\Delta_\theta u} &\le \tfrac32\norm{\Delta^2u}^{1/3}\norm{\nabla u}^{2/3}+h^{-1}\norm{\nabla L}_\infty\norm{\nabla u}\\
&\le \eta \norm{\Delta^2u}+\big(2^{-1/2}\eta^{-1/2}+h^{-1}\norm{\nabla L}_\infty\big)\norm{\nabla u},
\end{align*}
Choosing $\eta=\frac{1}{32}\frac{1}{27}d^{-4}h^2\norm{\nabla L}_\infty^{-4}$, we obtain from \eqref{EqDeltatheta2}
\begin{align*}
\norm{\Delta_\theta^2u} &\ge \tfrac{1}{32}\norm{\Delta^2u}-\Big(15^52^5d^6\norm{\nabla L}_\infty^6+ \tfrac{27}{2} d^4 \norm{\nabla^2 L}_\infty^2+ 2^{-11/5}3^{-1}5d^{6/5} \norm{\nabla\Delta L}_\infty^{6/5}\\
&\quad +27\norm{\nabla L}_\infty^4\big(2^23^{3/2}d^6\norm{\nabla L}_\infty^2+d^4\norm{\nabla L}_\infty\big)\Big)h^{-3}\norm{\nabla u}\\
&= \tfrac{1}{32}\norm{\Delta^2u}-\bar c_{\theta,d}h^{-3}\norm{\nabla u}
\end{align*}
where
\begin{align*}
	\bar c_{\theta,d}&=\big(15^52^5+2^2 3^{9/2} \big)d^6\norm{\nabla L}_\infty^6+ 27d^4\norm{\nabla L}_\infty^5 \\
 		&\quad\quad + \tfrac{27}{2} d^4 \norm{\nabla^2 L}_\infty^2+ 2^{-11/5}3^{-1}5d^{6/5} \norm{\nabla\Delta L}_\infty^{6/5}.
\end{align*}
With \eqref{eq:LowerPowerBoundSquare} we thus deduce
\begin{align*}
\Delta_\theta^4&\succcurlyeq 2^{-11}\Delta^4-\bar c_{\theta,d}^2h^{-6}(-\Delta).
\end{align*}
The weighted geometric mean inequality  \citep[Section 3]{Kubo1980} asserts for positive operators $A\preccurlyeq C$, $B\preccurlyeq D$ and $\alpha\in [0,1]$
\[  A^{1/2}\big(A^{-1/2}BA^{-1/2})^\alpha A^{1/2}\preccurlyeq C^{1/2}\big(C^{-1/2}DC^{-1/2})^\alpha C^{1/2}.
\]
We use this in the simpler case when $A$ and $B$ as well as $C$ and $D$ commute such that $A^{1-\alpha}B^\alpha\preccurlyeq C^{1-\alpha}D^\alpha$.
Together with $(-\Delta_\theta)^4\succcurlyeq 0$, $(-\Delta_\theta)\succcurlyeq\tfrac{1}{2}(-\Delta)$ and \eqref{EqOpPowers} we find that
\begin{align*}
(-\Delta_\theta)^{3+2\beta}&=(-\Delta_\theta)^{(1-2\beta)/3} ((-\Delta_\theta)^4)^{(2+2\beta)/3}\\
&\succcurlyeq 
 2^{-(23+20\beta)/3}(-\Delta)^{3+2\beta}
-\bar c_{\theta,d}^{(4+4\beta)/3}2^{-(1-2\beta)/3} h^{-(4+4\beta)}(-\Delta).
\end{align*}
Proceeding as in part (a) and
gathering the numerical constants in $C_{d, \beta}^{(3)}$, this implies \eqref{eq:proofDiffusionPartBReduction} and therefore the result in (b).
\end{proof}

\begin{proof}[Proof of Proposition \ref{PropOrderDeltaplustheta}.] \label{proof:PropOrderDeltaplustheta}
For $u\in {\cal H}^2(\R^d)$ we have
\begin{align*}
\norm{(\nu\Delta-M[\theta]) u}&\ge \norm{(\nu\Delta-\Id) u}-\norm{L}_\infty\norm{ u}
\end{align*}
and thus, with \eqref{eq:LowerPowerBoundSquare},
\begin{equation}\label{EqA1bound} \norm{A_{\bf 1} u}^2\ge \tfrac12\norm{A_{\bf 0} u}^2-\norm{L}_\infty^2\norm{ u}^2.
\end{equation}
This shows for all test functions $u\in {\cal H}^4(\R^d)$ that due to $\norm{L}_\infty\le 1/2$ and $\eps\le 1$
\[ \scapro{(\eps^2A_{\bf 1}^2+\Id)u}{u}\ge \tfrac12 \scapro{(\eps^2A_{\bf 0}^2+\Id)u}{u}.
\]
By the denseness of ${\cal H}^4(\R^d)$ in $L^2(\R^d)$ and the operator monotonicity of $x\mapsto -x^{-1}$ we establish \eqref{EqDeltaplusthetaorder}.

For (b) let us first bound from below for $u\in {\cal H}^6(\R^d)$, $\eta>0$  by Lemma \ref{lem:CompactProductInterpolation} and $2AB\le A^2+B^2$
\begin{align*}
\norm{\nabla A_{\bf 1}u}&=
\norm{\nabla(A_{\bf 0}-M[L(x/h)]u}\\
&\ge \norm{\nabla A_{\bf 0}u}-(2dh)^{1/2}h^{-1}\norm{\nabla L}_\infty\norm{u}^{1/2}\norm{\nabla u}^{1/2}-\norm{L}_\infty\norm{\nabla u}\\
&\ge \norm{\nabla A_{\bf 0}u}-\tfrac12 dh^{-1}\norm{\nabla L}_\infty^2\eta^{-1}\norm{u}-(\eta+\norm{L}_\infty)\norm{\nabla u}.
\end{align*}
From \eqref{eq:LowerPowerBoundSquare} we have $(A-B-C)_+^2\ge \frac12 A^2-2B^2-2C^2$ and thus
\[ \scapro{(-\Delta)A_{\bf 1}u}{A_{\bf 1}u} \ge \tfrac12\norm{\nabla A_{\bf 0}u}^2-\tfrac12 d^2h^{-2}\norm{\nabla L}_\infty^4\eta^{-2}\norm{u}^2-2(\norm{L}_\infty+\eta)^2\norm{\nabla u}^2.
\]
In view of $\inf_{x\in\R^d}\theta(x)\ge 1/2$ (due to $\norm{L}_\infty<1/2$) together with \eqref{EqA1bound}, and choosing $\eta^2=d\nu^{1/2}h^{-1}\norm{\nabla L}_\infty^2$, we arrive at
\begin{align*}
\scapro{(-A_{\bf 1})^3u}{u}
&\ge
\tfrac\nu 2 \norm{\nabla A_{\bf 0} u}^2-\tfrac12 d^2\nu\eta^{-2}h^{-2}\norm{\nabla L}_\infty^4\norm{u}^2 -2\nu(\norm{L}_\infty+\eta)^2\norm{\nabla u}^2\\
&\quad
 +\tfrac{1}{4}\norm{A_{\bf 0}u}^2-\tfrac{1}{2}\norm{L}_\infty^2\norm{u}^2\\
&\ge
\tfrac1 4 \scapro{(-A_{\bf 0})^3u}{u}-\tfrac12 \big(d\nu^{1/2} h^{-1}\norm{\nabla L}_\infty^2+\norm{L}_\infty^2\big) \norm{u}^2 \\ &\quad -4\nu(\norm{L}_\infty^2+d\nu^{1/2} h^{-1}\norm{\nabla L}_\infty^2)\norm{\nabla u}^2 \\
&\ge \tfrac1 4 \scapro{(-A_{\bf 0})^3u}{u} -4(\norm{L}_\infty^2+d\nu^{1/2} h^{-1}\norm{\nabla L}_\infty^2)\scapro{(\Id-\nu\Delta)u}{u}.
\end{align*}
We thus obtain
\begin{align*}
&\scapro{(-A_{\bf 1}) (\eps^2A_{\bf 1}^2+\Id)u}{u}\\
&\ge \tfrac12\scapro{(-A_{\bf 0})u}{u}+\tfrac14\eps^2 \scapro{(-A_{\bf 0})^3u}{u} -4\eps^2(\norm{L}_\infty^2+d\nu^{1/2} h^{-1}\norm{\nabla L}_\infty^2)\scapro{(\Id-\nu\Delta)u}{u}\\
&\ge \tfrac14\scapro{(-A_{\bf 0})(\eps^2A_{\bf 0}^2+\Id)u}{u}
 -4\eps^2\big(\norm{L}_\infty^2+d\nu^{1/2} h^{-1}\norm{\nabla L}_\infty^2\big) \scapro{(-A_{\bf 0})u}{u} \\
&\ge  \big(\tfrac14-4\eps^2(\norm{L}_\infty^2+d\nu^{1/2} h^{-1}\norm{\nabla L}_\infty^2)\big) \scapro{(-A_{\bf 0})(\eps^2A_{\bf 0}^2+\Id)u}{u}.
\end{align*}
Together with part (a) and $(x+T^{-1})/2\le x\vee T^{-1}\le x+T^{-1}$ we have thus shown
\[(\Id-2TA_{\bf 1})(\eps^2A_{\bf 1}^2+\Id)\succcurlyeq \tfrac1{16} (\Id-2TA_{\bf 0})(\eps^2A_{\bf 0}^2+\Id),
\]
provided $\norm{L}_\infty^2+d\nu^{1/2} h^{-1}\norm{\nabla L}_\infty^2\le \tfrac{1}{32}\eps^{-2}$. By the operator monotonicity of $x\mapsto -x^{-1}$
and $A+B\le \eps^{-1}\Rightarrow A^2+B^2\le\eps^{-2}$, $A,B\ge 0$, this yields assertion (b).
\end{proof}

\section{Auxiliary Results}\label{SecAux}

In the following interpolation inequalities the dependence on the dimension $d$ is not always optimised, but some dimension dependence cannot be avoided.

\begin{lemma}\label{lem:CompactProductInterpolation}
Let $\theta\in C^2(\R^d)$ have support in $[-h/2,h/2]^d$.
Then
\begin{align}
\norm{\theta u}&\le (2h)^{1/2}\norm{\theta}_\infty \norm{u}^{1/2}\norm{\partial_i u}^{1/2},\quad u\in {\cal H}^1(\R^d). \label{EqPI_ref}
\end{align}
holds for any $i=1,\ldots,d$. In particular, we have
\begin{align}
\norm{(\nabla \theta) u}&\le (2dh)^{1/2}\norm{\nabla \theta}_\infty \norm{u}^{1/2}\norm{\nabla u}^{1/2},\quad u\in {\cal H}^1(\R^d), \\
\norm{\scapro{\nabla \theta}{\nabla u}}&\le (2d^2h)^{1/2}\norm{\nabla \theta}_\infty \norm{\nabla u}^{1/2}\norm{\Delta u}^{1/2},\quad u\in {\cal H}^2(\R^d), \\
\norm{(\nabla^2 \theta) \nabla u}&\le (2d^3h)^{1/2}\norm{\nabla^2 \theta}_\infty \norm{\nabla u}^{1/2}\norm{\Delta u}^{1/2},\quad u\in {\cal H}^2(\R^d), \\
\norm{(\nabla^2 u) \nabla \theta}&\le (2d^3h)^{1/2}\norm{\nabla \theta}_\infty \norm{\Delta u}^{1/2}\norm{\nabla \Delta u}^{1/2},\quad u\in {\cal H}^3(\R^d), \\
\norm{\scapro{\nabla^2 \theta}{\nabla^2 u}_{HS}}&\le (2d^4h)^{1/2}\norm{\nabla^2 \theta}_\infty \norm{\Delta u}^{1/2}\norm{\nabla \Delta u}^{1/2},\quad u\in {\cal H}^3(\R^d).
\end{align}
\end{lemma}

\begin{proof}
	To prove \eqref{EqPI_ref}, use integration by parts and obtain
\begin{align*}
&\norm{\theta u}^2 = \int_{\R^d}\theta^2(x) u^2(x)\,dx\\
&= \babs{- \int_{\R^d} \Big(\int_{-h/2}^{x_i\wedge(h/2)}\theta^2(x_1,\ldots,x_{i-1},y,x_{i+1},\ldots,x_d)\,dy\Big)2u(x)\partial_{i}u(x)
\,dx } \\
&\le 2h\norm{\theta}_\infty^2\int_{\R^d} \abs{u(x)}\abs{\partial_iu(x)}\,dx
\le 2h\norm{\theta}_\infty^2\norm{u}\norm{\partial_i u}.
\end{align*}
All other estimates follow by replacing $u$ and $\theta$ in \eqref{EqPI_ref} by partial derivatives of $u$ and $\theta$, where we take into account $\norm{\nabla^2u} = \norm{\Delta u}$ (and thus $\norm{\nabla^3u} = \norm{\nabla\Delta u}$, where $\norm{\nabla^3 u}^2=\sum_{i,j,k}\norm{\partial_{ijk}u}^2$).
For example, we have
\begin{align*}
	\norm{(\nabla^2 \theta) \nabla u}^2
		&= \sum_{i=1}^d\int_{\R^d}\Big(\sum_{j=1}^d\partial_{ij}\theta(x)\partial_ju(x)\Big)^2\,dx
		\le d\sum_{i,j=1}^d\int_{\R^d}(\partial_{ij}\theta)(x)^2(\partial_ju)(x)^2\,dx \\
		&\le d\sum_{i,j=1}^d 2h\norm{\partial_{ij}\theta}_\infty^2\norm{\partial_j u}\norm{\nabla\partial_j u}
		\le 2d^3h\norm{\nabla^2\theta}_\infty^2\norm{\nabla u}\norm{\Delta u}.
\end{align*}
The other estimates work analogously.
\end{proof}

\begin{lemma}\label{LemTraceBound} \
\begin{enumerate}
\item
For bounded selfadjoint operators $T_1,T_2,T_3$ with $T_1\succcurlyeq 0$ and $T_2\preccurlyeq T_3$
\[\trace(T_1T_2)\le \trace(T_1T_3)\]
holds, provided the right-hand side is finite. Also,
for operators $T_4,T_5,T_6$ with $T_4^\ast T_4\preccurlyeq T_5^\ast T_5$
\[ \norm{T_4T_6}_{HS}^2\le \norm{T_5T_6}_{HS}^2\]
holds, provided the right-hand side is finite.

\item  For  operators $R_1\succcurlyeq0$, $R_2\succcurlyeq0$ with bounded inverses and a bounded operator $T$ we have
\[ R_1\succcurlyeq TR_2T^\ast \Rightarrow T^\ast R_1^{-1}T\preccurlyeq R_2^{-1}.\]

\end{enumerate}
\end{lemma}

\begin{proof}
In (a) the operator  $T_1^{1/2}(T_3-T_2)T_1^{1/2}$ is positive semi-definite such that its trace is nonnegative, implying the first claim in (a):
\[\trace(T_1T_3)-\trace(T_1T_2)=\trace(T_1^{1/2}(T_3-T_2)T_1^{1/2})\ge 0.\]
In view of $\norm{ST}_{HS}^2=\trace(S^\ast STT^\ast)$ for operators $S,T$ the second claim follows from the first with $T_1=T_4^\ast T_4$, $T_2=T_5^\ast T_5$ and $T_3=T_6T_6^\ast$.
	
In (b) we have
\begin{align*}
	R_{\bf 1}\succcurlyeq TR_2T^\ast &\Rightarrow R_{\bf 1}^{-1/2}TR_2T^\ast R_{\bf 1}^{-1/2}\preccurlyeq\Id \\
		&\Rightarrow R_2^{1/2}T^\ast R_{\bf 1}^{-1}TR_2^{1/2}=(R_{\bf 1}^{-1/2}TR_2^{1/2})^\ast R_{\bf 1}^{-1/2}TR_2^{1/2}\preccurlyeq \Id\\
&\Rightarrow T^\ast R_{\bf 1}^{-1}T\preccurlyeq R_2^{-1},
\end{align*}
where $\norm{LL^\ast}=\norm{L^\ast L}$ was used for $L=R_{\bf 1}^{-1/2}TR_2^{1/2}$.
\end{proof}

\begin{lemma}\label{lem:HS:kernel}
	Let $K: L^2(H)\to L^2(H)$ be the operator given by $Kf(t)=\int_0^T k(t,s)f(s)\,ds$ for an operator-valued kernel function $k\in L^2([0,T]^2;HS(H))$. Then we have
\begin{align*}
	\norm{K}_{HS(L^2(H))}^2=\int_0^T\int_0^T \norm{k(t,s)}_{HS(H)}^2dtds.
\end{align*}
\end{lemma}

\begin{proof}
This extension from the classical result for scalar-valued kernels can be found in \citet[Thm. 11.3.6]{BiSo2012}.
\end{proof}

\begin{lemma}\label{LemGVar}
Suppose $\Gamma\thicksim {\cal N}_{cyl}(0,\Sigma)$ is a cylindrical Gaussian measure on a real Hilbert space $H$ and $L$ is a bounded normal operator on $H$ such that $\Sigma^{1/2}\Re(L)\Sigma^{1/2}$ is a Hilbert-Schmidt operator. Then
\[ \Var(\scapro{L\Gamma}{\Gamma})=2\norm{\Sigma^{1/2}\Re(L)\Sigma^{1/2}}_{HS}^2\]
holds. An upper bound is given by $2\norm{\Sigma L}_{HS}^2$.
\end{lemma}

\begin{proof}
Without loss of generality assume $\Sigma=\Id$, writing $\scapro{L\Gamma}{\Gamma}=\scapro{(\Sigma^{1/2}L\Sigma^{1/2})\Sigma^{-1/2}\Gamma}{\Sigma^{-1/2}\Gamma}$ and replacing $L$ by $\Sigma^{1/2}L\Sigma^{1/2}$ (if $\Sigma$ is not one-to-one restrict to the range of $\Sigma$).
Denote the eigenvalues of $L$ by $(\lambda_k)$ and the associated orthonormal basis of eigenvectors by $(e_k)$. Then due to $\scapro{L\Gamma}{\Gamma}=\scapro{L^\ast\Gamma}{\Gamma}$ and $\scapro{\Gamma}{e_k}\sim N(0,1)$
\begin{align*}
 \Var(\scapro{L\Gamma}{\Gamma})&=\Var(\scapro{\Re(L)\Gamma}{\Gamma})=\textstyle\sum_k\Var(\Re(\lambda_k)\scapro{\Gamma}{e_k}^2)
 =2\sum_k \Re(\lambda_k)^2,
 \end{align*}
which is $2\norm{\Re(L)}_{HS}^2$. For the upper bound we note
\[\norm{\Sigma^{1/2}\Re(L)\Sigma^{1/2}}_{HS}^2=\scapro{\Sigma \Re(L)}{\Re(L)\Sigma}_{HS}\le \norm{\Sigma\Re(L)}_{HS}^2\le \norm{\Sigma L}_{HS}^2.\]
\end{proof}

For reference, we state the \citet{weyl1912} asymptotics of the Laplacian eigenvalues, see \cite{Shubin2001} for a general discussion.
\begin{lemma}\label{lem:Weyl}
	Let $\Delta$ be the Laplacian on a smooth bounded domain of dimension $d$ with Dirichlet, Neumann or periodic boundary conditions, or on a $d$-dimensional compact manifold without boundary. Then the ordered eigenvalues $(\lambda_k)_{k\ge 1}$ of $\Delta$ satisfy
	\begin{equation}\label{EqWeyl}
		-\lambda_k\thicksim k^{2/d},\quad k\ge 2.
	\end{equation}
Under Neumann or periodic boundary conditions we have $\lambda_1=0$, under Dirichlet boundary conditions $-\lambda_1\thicksim 1$ holds.
\end{lemma}

\begin{lemma} \label{lem:aux:integrals} \
	\begin{enumerate}
		\item Let $a, b, c\in\R$ with
		$(a+c)\wedge a > -1 > (b+c)\vee b$.
		Then uniformly in $t\ge 0$
		\begin{align*}
			\int_0^\infty (z^a\wedge z^b)(1 + tz)^c \,dz \sim (1+t)^c.
		\end{align*}
		
		\item
		Let $a,b,c\in\R$ with $a>b$ and $a>-1>(b+c)\vee c$.
		Then uniformly in $t\ge 0$
		\begin{align*}
			\int_0^\infty (z^a\wedge z^{b})(1\wedge\abs{z-t}^{c})\,dz \lesssim (1+t)^{b\vee c}.
		\end{align*}
		In particular, if $a\ge 0$ and $b=a+c$, the right-hand side is of order $(1+t)^{b}$.
	\end{enumerate}
\end{lemma}

\begin{proof}
		For $t, z\ge 0$ we have $1\wedge z\le\tfrac{1+tz}{1+t}\le 1\vee z$, so in (a)
		\begin{align*}
			0<\int_0^\infty (z^a\wedge z^b)(1\wedge z^c) \,dz &\le \int_0^\infty (z^a\wedge z^b)\big(\tfrac{1+tz}{1+t}\big)^c \,dz
				\le \int_0^\infty (z^a\wedge z^b)(1\vee z^c) \,dz,
		\end{align*}
		and for $a>b$ this is finite if and only if $(a+c)\wedge a>-1$ and $(b+c)\vee b<-1$.
		
		For $0\le t\le 2$ the claim in (b) reduces to $1\sim 1$, so assume $t\ge 2$. We split the integral $\int_0^\infty=\int_0^1+\int_1^{t-1}+\int_{t-1}^{t+1}+\int_{t+1}^\infty$ and treat each part separately.
		First, $\int_0^1z^a(t-z)^{c}dz\lesssim (t-1)^{c}\lesssim t^{c}$.
		Next,
		\begin{align*}
			\int_{1}^{t-1}z^{b}(t-z)^{c}\,dz
				&\lesssim t^{1+b+c}\Big(\int_{1/t}^{1/2}u^{b}\,du + \int_{1/2}^{1-1/t}(1-u)^{c}\,du\Big) \\
				&\lesssim t^{1+b+c}(1\vee\log(t)\vee t^{-b-1} + t^{-c-1}) \lesssim t^{b\vee c},
		\end{align*}
		then $\int_{t-1}^{t+1}z^{b}dz\lesssim t^{b}$ and finally
		\begin{align*}
			\int_{t+1}^\infty z^{b}(z-t)^{c}\,dz
				&\lesssim t^{1+b+c}\Big(\int_{1+1/t}^2(u-1)^{c}\,du + \int_2^\infty u^{b}(u-1)^{c}\,du\Big) \\
				&\lesssim t^{1+b+c}(t^{-c-1} + 1) \sim t^{b}.
		\end{align*}
\end{proof}

\begin{lemma}\label{lem:aux:alpha}
	Let $\alpha>0$ and $p_i,q_i\in\R$ for $i=1,2,3$. Then
	\begin{align}\label{eq:aux:alpha:rate}
		 T_n^{p_1\alpha +q_1}\eps_n^{-(p_2\alpha+q_2)}\nu_n^{-(p_3\alpha+q_3)}\gtrsim 1
	\end{align}
	holds for positive sequences $(T_n)$, $(\eps_n)$, $(\nu_n)$ if $\abs{p_1\log(T_n)+p_2\log(\eps_n^{-1})+p_3\log(\nu_n^{-1})}\rightarrow \infty$ and
	\begin{align*}
		\alpha> \limsup_{n\rightarrow\infty}\frac{-q_1\log(T_n)-q_2\log(\eps_n^{-1})-q_3\log(\nu_n^{-1})}
{p_1\log(T_n)+p_2\log(\eps_n^{-1})+p_3\log(\nu_n^{-1})}.
	\end{align*}
\end{lemma}

\begin{proof}
By taking logarithms, \eqref{eq:aux:alpha:rate} is equivalent to \[(p_1\alpha+q_1)\log(T_n)+(p_2\alpha+q_2)\log(\eps_n^{-1})+(p_3\alpha+q_3)\log(\nu_n^{-1})\ge \log(c),\]
where $c>0$ is the proportionality constant. Solve for $\alpha$.
\end{proof}

\bibliographystyle{apalike2}
\bibliography{SPDENoiseLowerBound}

\end{document}